%% file: main.tex
\documentclass[reqno,10pt,letterpaper]{amsart}

\usepackage{xcolor}
\definecolor{winered}{rgb}{0.6,0,0}
\definecolor{lessblue}{rgb}{0,0,0.7}

\usepackage{graphicx}
\usepackage{import}
 \usepackage{amsmath,amssymb,amsthm,mathrsfs,cancel}
\newtheorem{lemma}{Lemma}[section]
\newtheorem{theorem}{Theorem}[section]
\newtheorem{proposition}{Proposition}[section]
\newtheorem{definition}{Definition}[section]

\theoremstyle{remark}
\newtheorem{remark}{Remark}[section]

\usepackage{pdfpages}
\newcommand\restr[2]{{
  \left.\kern-\nulldelimiterspace 
  #1 
  \vphantom{\big|} 
  \right|_{#2} 
  }}
\usepackage[pdftex,colorlinks=true,linkcolor=winered,citecolor=lessblue,urlcolor=lessblue,breaklinks=true,bookmarksopen=true]{hyperref}

\begin{document}
\bibliographystyle{amsalpha}
\begin{abstract}
    Using an approach similar to \cite{hintz2024stability}, we give a new proof of the nonlinear stability of de Sitter space as a solution to the Einstein vacuum equations with positive cosmological constant in $n+1$ dimensions, with $n\geq3$. Using the gauge freedom of the equations, we are able to prove a precise expansion of the perturbed spacetime at the conformal boundary. In $n=$ odd spatial dimensions, the conformally rescaled metric is smooth up to the future conformal boundary and in $n=$ even spatial dimensions it is smooth if and only if the obstruction tensor of the boundary metric vanishes; if not, then the conformally rescaled metric is log smooth at the boundary. These results also hold for asymptotically de~Sitter spaces. Using the results of \cite{fefferman1985,fefferman2008ambientmetric,Rodnianski2018-lt,Hintzasydesi}, the structure of our expansion allows us to establish a 1-1 correspondence between solutions to the Einstein vacuum equations close to de Sitter space and scattering data prescribed on the conformal boundary in general dimension.
\end{abstract}

\title{Stability of de Sitter Space and Expansion at the Conformal Boundary}
\date{\today}
\author{Maurus Leimbacher}
\address{Department of Mathematics, ETH Zurich, Rämistrasse 101, 8092 Zürich, Switzerland}
\email{lmaurus@ethz.ch}

\maketitle

\section{Introduction}
We study the stability of solutions to the Einstein vacuum equation
\begin{equation}
\label{eq:Einstein}
    \operatorname{Ric}g-\Lambda g=0
\end{equation}
in $n+1$ space-time dimensions with $n\geq3$. Here $g$ is a Lorentzian metric of signature $(-,+,\ldots,+)$ and $\Lambda$ is the cosmological constant and assumed positive; we fix $\Lambda=n$ as can always be achieved by scaling. The basic example is the de~Sitter solution
\begin{equation}
    \label{eq:de Sitter}
    M_{\mathrm{dS}}=(0,\pi)_s\times \mathbb{S}^n_\omega,\quad g_{\mathrm{dS}}=\frac{-\mathrm{d}s^2+\cancel{g}}{\sin(s)^2},
\end{equation}
where $\cancel{g}$ is the standard metric on $\mathbb{S}^n$. The conformal rescaling $\sin(s)^2g=-\mathrm{d}s^2+\cancel{g}$ gives the Einstein static universe which is smooth up to the boundary of $[0,\pi]\times\mathbb{S}^n$, called the \textit{conformal boundary}. It consists of $\mathcal{I}^+:=\{s=0\}\times\mathbb{S}^n$ and $\mathcal{I}^-:=\{\pi\}\times\mathbb{S}^n$, which we call the future and past conformal infinities, respectively.\footnote{Working near $s=0$ is notationally easier and in order to consider the future Cauchy problem we chose this convention. As de~Sitter space is symmetrical around $s=\pi/2$, we could have equivalently switched the naming convention throughout this paper.} As a consequence, in this paper, we consider the $-\partial_s$ vector to be future-directed.\\

\subsection{Stability}
We consider the stability problem. Global stability of de Sitter space was proved in $3+1$ dimensions by Helmut Friedrich in \cite{Friedrich1986}, rewriting the Einstein equations into conformal field equations. The conformal approach breaks down in general dimensions, but Anderson \cite{Anderson2005} was able to generalize the result to arbitrary dimension $n+1$ where $n$ is odd. He used properties of the ambient obstruction tensor of \cite{fefferman1985} which is identically 0 in the dimensions considered and was also able to characterize the asymptotic behavior. Both Friedrich and Anderson showed that the metrics can be written as $\Omega^{-2}\bar{g}$, where the conformal factor $\Omega$ vanishes simply at the conformal boundary of the perturbed spacetime and $\bar{g}$ is the unphisical metric which extends smoothly to the conformal boundary. Stability in the case of fully general dimension was proved in \cite{Ringstrom2008-zu}. In fact, Ringström proves a more general setting with a scalar field present, but this result also encompasses the cosmological constant case. Ringström's method on the other hand only gives the leading order term in the expansion at conformal infinity. Our approach (adapted from \cite{hintz2024stability}) to proving stability is similar in spirit to Ringström's approach, but we improve on it by explaining in a more conceptual fashion the reason for stability in the chosen gauge via constraint damping and indicial roots. An indicial roots perspective on de~Sitter type spaces was already emphasized in Vasy's work on the wave equation on asymptotically de~Sitter-like spaces \cite{Vasy2007-to}. Furthermore, the method we use can then be used to prove the precise asymptotic behavior of the solution in arbitrary dimension. For more details on the context of this expansion and other related results, see the discussion after the rough statement of the two main results of this paper. Cicortas \cite{cicortas2024nonlinearscatteringtheoryasymptotically} also proves the nonlinear stability in $n+1$ dimensions, $n$ even, together with a sharp scattering isomorphism which we discuss in more detail below.\\

We will study the future Cauchy problem for equation \eqref{eq:Einstein} with initial data posed at a hypersurface $\Sigma=\{s_0\}\times\mathbb{S}^n$ of $M_{\mathrm{dS}}$. The initial data for this system are the first and second fundamental forms $\gamma$ and $k$ of $\Sigma$, satisfying the constraint equations
$$
R_\gamma-|k|_\gamma^2+\left(\operatorname{tr}_\gamma k\right)^2-2 \Lambda=0, \quad \delta_\gamma k+\mathrm{d} \operatorname{tr}_\gamma k=0.
$$
Here, $R_\gamma$ is the scalar curvature, and $\left(\delta_\gamma k\right)_\mu=-\gamma^{\kappa \lambda} k_{\mu \kappa ; \lambda}$.\\

We consider smooth initial data close to the first and second fundamental form on $\Sigma$ induced by $g_{\mathrm{dS}}$, as measured in a suitable Sobolev space. By working in a generalized harmonic gauge (see section \ref{se:stability}) on the domain 
$$
\Omega=[0,s_0]\times\mathbb{S}^n,
$$
we are able to prove stability on $\Omega$. Gluing together the past solution using a finite in time result, we are able to prove global stability. The precise version is the content of Theorem \ref{thm:solution}.
\begin{theorem}[Stability of de Sitter space, rough version.]
    The development $(g,M_{\mathrm{dS}})$ of this set of initial data can on $\Omega$ be written as $g_{\mathrm{dS}} + \frac{\chi(s) h_{(0)}}{s^2} + \tilde{h}$. The first correction term encapsulates the first-order change $h_0(s,\omega;\mathrm{d}\omega):=\frac{h_{(0)}(\omega; \mathrm{d}\omega)}{s^2}$ of $g$ from $g_{\mathrm{dS}}$ at $\mathcal{I}^+$. Here, $h_{(0)}$ is a smooth section of $S^2T^*\mathbb{S}^n$ independent of $s$ and $\chi(s)$ is a smooth cutoff, compactly supported in $[0,s_0/2)$. The second term, $\tilde{h}$, is a smooth section of $ S^2 T^*((0,s_0]\times\mathbb{S}^n)$ for which $s^2\tilde{h}$ extends continuously and decays at $\mathcal{I}^+$. It describes further corrections that decay in comparison to $g_{\mathrm{dS}}$ and $h_0$ at future conformal infinity and ensures the initial conditions are satisfied. Both correction terms are small in their respective (in the case of $\tilde{h}$ weighted-) Sobolev space. Furthermore, the metric $g$ in a neighborhood around the past conformal boundary $\mathcal{I}^-$ is of the same form, if written in the coordinate $\pi-s$. Lastly, the constructed space-time is future and past geodesically complete.
\end{theorem}
\subsection{Asymptotic expansion}
The second half of the paper is concerned with proving a precise asymptotic expansion of the solution metric $g$ at $\mathcal{I}^+$, namely bringing it into Fefferman-Graham form. This is achieved by pulling back the already obtained solution by suitable diffeomorphisms to improve the behavior at $\mathcal{I^+}$. The resulting expansion is different depending on whether the spatial dimension $n\geq3$ is odd or even. We give here a rough version of the result, the precise statement is Theorem \ref{thm:expansion}.\\

\begin{theorem}[Main result, rough version]
The solution to the Cauchy problem \eqref{eq:Einstein} with initial data close to de~Sitter data posed at $\Sigma$ is isometric to $(g,M_{\mathrm{dS}})$, where in a neighborhood of $\mathcal{I}^+$, $g$ is block diagonal,
$$
g=\frac{-\mathrm{d}s^2+H(s,\omega;\mathrm{d}\omega)}{s^2}.
$$
Furthermore, we have the following power series structure of $g$. Setting $g_{(0)}:= (s^2g)\left|_{\mathcal{I}^+}\right.$, we have: \\

\begin{enumerate}
    \item If the dimension $n$ is odd, the conformally rescaled metric $s^2g$ is smooth down to $\mathcal{I}^+$. More precisely, there exist $g_{(i)}\in C^\infty(\mathcal{I}^+;S^2T^*\mathbb{S}^n)$ for $ i \in \mathbb{N}_{\geq2}$ such that for all $N\in \mathbb{N}$ we have that
    \begin{equation}
    \label{eq:expansionoddpre}
         g(s,\omega)-\left(\frac{-\mathrm{d}s^2+g_{(0)}(\omega)}{s^2} +s^{-2}\sum_{i=2}^Ns^{i}g_{(i)}(\omega)\right)
    \end{equation}
    is smooth and decays faster than $s^{N-2}$ for $s\rightarrow0$. Furthermore, for $i<n$, $g_{(i)}$ is computable using only $g_{(0)}$. In fact, for odd $i<n$ $g_{(i)}=0$, therefore the first odd order that appears in the expansion is $g_{(n)}$. It is a transverse traceless tensor, meaning $\operatorname{tr}_{g_{(0)}}g_{(n)}=0$ and $\delta_{g_{(0)}}g_{(n)}=0$. All further $g_{(i)}$ can then be computed using $g_{(0)},g_{(n)}$.\\
    \item  If the dimension $n$ is even, the conformally rescaled metric $s^2g$ is smooth down to $\mathcal{I}^+$ if and only if the obstruction tensor (a notion explained below) of $g_{(0)}$ vanishes. More precisely, for all $i\in \mathbb{N}_{\geq2}$ and $m\leq \lfloor i/n\rfloor$ there exist $g_{(i)}^m\in \mathcal{C}^\infty(\mathcal{I}^+;S^2T^*\mathbb{S}^n)$ such that for all $N\in\mathbb{N}$
    \begin{equation}
        g(s,\omega)-\left(\frac{-\mathrm{d}s^2+g_{(0)}(\omega)}{s^2}+s^{-2}\sum_{i=2}^N\sum_{m=0}^{\lfloor i/n\rfloor} s^{i}\log(s)^m g_{(i)}^m(\omega)\right)
    \end{equation}
    is of class $\mathcal{C}^{N-2}$ and decays faster than $s^{N-2}$ for $s\rightarrow0$. Furthermore, $g_{(i)}^m=0$ if $i$ is odd and thus only even powers of $s$ appear in the expansion. For $i<n$, the $g_{(i)}^m=g_{(i)}^0$ can be computed solely using $g_{(0)}$. Further, $g_{(n)}^0$ is a traceless tensor whose divergence is a certain one-form, and all further $g_{(i)}^m$ are then determined by $g_{(0)},g_{(n)}^0$. The first log-order that appears is $g_{(n)}^1$, which vanishes if and only if the obstruction tensor of $g_{(0)}$ vanishes. If it does, all further $g_{(i)}^m$ for $m>0$ vanish.
\end{enumerate} 
In a neighborhood of $\mathcal{I}^-$, $g$ can be brought into the same structural form in the coordinates $\tilde{s}=\pi-s$ through isometry.
\end{theorem}
This result shows that initial data close to de Sitter space evolve into a spacetime with two asymptotic functional degrees of freedom, both in odd and even dimensions. Those degrees of freedom are the first order metric $g_{(0)}$ and the order $n$ contribution $g_{(n)}$ (or $g_{(n)}^0$ in the even dimensional case), which is traceless. Depending on the parity of $n$, $g_{(n)}$ is either transverse or its divergence is a certain one-form. Data of this type is called \textit{scattering data}.\\

It was shown by Fefferman-Graham in \cite{fefferman1985}\cite{fefferman2008ambientmetric} that for such scattering data, there exists a formal solution to the Einstein vacuum equations \eqref{eq:Einstein} with uniquely determined expansion coefficients at the conformal boundary. Furthermore, this expansion has the exact same form as the one in our result. It was then proved that this Taylor expansion (in even spacial dimensions also allowing log terms) converges when the scattering data are real-analytic \cite{Kichenassamy2004, Rendall2004}. More generally, the formal solution (also in the non analytic case) can be upgraded to a true solution in a neighborhood of the conformal boundary by adding a tensor that vanishes to infinite order. This was first shown in \cite{Rodnianski2018-lt} in a more general setting, and then, using a simpler argument, in \cite{Hintzasydesi} for asymptotically de Sitter metrics.\\

This shows that the description of the metric we achieved is optimal. Given the solution $g$ from our paper, compute the $g_{(0)},g_{(n)}$ found in our proof. This specifies scattering data. If, in turn, we are given these scattering data, the expansion constructed by Fefferman-Graham recovers the full expansion of this paper. This can then be upgraded to a unique true solution, recovering our initial solution $g$, up to isometry. This proves a 1-1 correspondence between solutions to the Einstein vacuum equation close to the de Sitter solution and scattering data close to de~Sitter scattering data prescribed on the boundary. This result can be generalized to asymptotically de~Sitter type metrics with a compact spatial manifold, as stated in remark \ref{re:generalX}. In $3+1$ dimensions, such a correspondence was shown in \cite{FRIEDRICH1986101}, which was then generalized to $n+1$ dimensions for $n$ odd in \cite{Anderson2005}. More recently, the full stability and scattering theory of asymptotically de~Sitter spaces in the case of $n$ even was solved by Cicortas in \cite{cicortas2024nonlinearscatteringtheoryasymptotically}. Said result is sharper than the one proved in this paper in the sense that it proves an isomorphism between finite-regularity Sobolev spaces between scattering data at the future and past conformal boundary, although at significant technical expense. Also, it crucially uses $n$ even. Our result goes through in all dimensions $n\geq3$. It is not based on any conformal properties of the field equations, but rather proves stability (much like \cite{Ringstrom2008-zu}), using methods of asymptotic analysis. It also shows the maximally precise asymptotics of theorem \ref{thm:expansion} for perturbations of asymptotically de~Sitter spaces in all dimensions (see remark \ref{re:generalX}), which was left partially open by previous works.\\

In even spatial dimensions, the determining factor whether the metric is smooth across $\mathcal{I}^+$ is the so-called \textit{obstruction tensor}. It arose in \cite{fefferman1985} as a conformal invariant, which obstructs the existence of a formal power series expansion for the ambient metric with prescribed conformal structure. Alternatively, it can be viewed as an obstruction to a power series expansion of a Poincaré metric for a given metric on an even dimensional manifold \cite{obstruction}. This is the way the obstruction tensor will appear in our proof and therefore it is not surprising to find it governs the smoothness of the solution. For a concrete definition in our setting, see \eqref{eq:obstructiontensor}.\\
\begin{remark}
\label{re:generalX}
Our proof of stability and expansion at the conformal boundary generalizes to the following setting. Let $X$ be a compact manifold of dimension $n$. Consider $M:=(0,\epsilon)_s\times X_\omega$ and suppose that there exists a smooth Lorentzian metric $g$ of signature $(-,+,\ldots,+)$ on $M$ such that $s^2 g$ is the sum of a smooth 2-tensor and a decaying (as $s\searrow 0$) conormal (infinite regularity along $s\partial_s$ and vector fields on $X$) tensor up to $\mathcal{I}^+$. Assume further that $g$ is asymptotically de~Sitter, meaning the leading-order term of $g$ near $\mathcal{I}^+$ is of the form $\frac{-\mathrm{d}s^2+H(\omega;\mathrm{d}\omega)}{s^2}$. If this is the case, the entire proof of this paper goes through very similarly. The only major difference is that we no longer have the two concrete charts of $\mathbb{S}^n$ as defined in section \ref{se:manifold}. Instead, we get a finite covering of $X$ by $V_i$ and then define $\Omega_i:=[0,\epsilon]\times V_i$. When proving the energy estimates, this requires different regions with space-like boundaries where the outward normal points in the correct direction. We constructed them explicitly ($U_{s_1}$ in section \ref{se:domainforesti}) but such domains can always be constructed. The other calculations to prove stability work analogously. If $X$ is not compact, our method of proof could still be applied to conclude the results of our paper in the future domain of dependence of local patches. But to get a global result, further arguments would be required to glue the solutions together.\\

Such asymptotically de~Sitter metrics can be constructed in a neighborhood around $\mathcal{I}^+$ from scattering data on a compact manifold $X$, as shown in \cite{Rodnianski2018-lt} or \cite{Hintzasydesi}. Our argument shows that small perturbations of such spacetimes at a spacelike hypersurface then evolve into future geodesically complete spacetimes and can be brought into the Fefferman-Graham form described above.
\end{remark}

\subsection{Method of proof and outline}
Our proof follows the method used in \cite{hintz2024stability}, where it was used to prove a local stability result and an asymptotic expansion for Kerr-de-Sitter space.\\

We control the linearized Einstein equations using energy methods (to control regularity) and indicial operator arguments (to control the decay of the solution near the conformal boundary).

As we work in general dimension, we find generalizations of the modified harmonic gauge and the constraint damping of \cite{hintz2024stability}. The requirement is that the indicial roots of the gauged Einstein operator all be nonnegative. We find a generalization that works in all $n+1$ dimensions simultaneously.

Unlike \cite{hintz2024stability}, in this paper we are interested in a \textit{global} stability result on $(0,\pi)\times\mathbb{S}^n$. We can therefore not work with global bundle splittings and Cartesian coordinates but rather have to control solutions and gauges globally throughout. This introduces modest complications compared to \cite{hintz2024stability} during the stability proof. 

To complete the non-linear analysis of the Einstein equations, we use tame estimates together with a Nash-Moser iteration scheme, as in \cite{hintz2024stability}.\\

The proof for the asymptotic expansion goes along the same lines as \cite{hintz2024stability}, but is more involved. The first reason is the global nature of our result. This requires us to work with flows along vector fields instead of concrete diffeomorphisms. To show that the behavior close to the conformal boundary is as desired, the somewhat technical Lemmas \ref{le:flowclose},\ref{le:flowsobolev},\ref{le:flowdif} and \ref{le:pullbackdiff} are required. With those in hand, we proceed similarly to
\cite{hintz2024stability} to show that the metric can be made log-smooth at $\mathcal{I}^+$ using a Mellin transformation-type argument.

The second complication comes from the general dimension in which we work. This leads to logarithmic terms in the expansion for even spatial dimensions not present in \cite{hintz2024stability}.

Lastly, we show that the metric can be brought into block diagonal form near the conformal boundary, a step not done in \cite{hintz2024stability}. We accomplish this by constructing boundary normal coordinates, as done in related settings in \cite{GRAHAM1991186} and \cite{Chrusciel2004-xh}.\\

We provide a short overview of the structure of this paper. In Section \ref{se:analysis}, we specify the domain on which we solve the equation and define the specific charts we will be working in. Next, we discuss the analysis near the conformal boundary and introduce the notion of 0- and b-operators in the terminology of \cite{Melrose1981-mz,Melrose1983EllipticOO,Melrose1993-uz} adapted to our setting. Lastly, we introduce the corresponding weighted Sobolev spaces and state some properties required for the analysis of the Einstein equation.

In Section \ref{se:stability}, we define the gauge-fixed Einstein operator and prove tame estimates for it. We compute the indicial operator of the linearization and compute its roots. Next we control the solution to the linearized equation using energy estimates, which is done by using the invertibility of the indicial operator. This allows us to prove our first main result, Theorem \ref{thm:solution}, the stability of de Sitter space.

In the last section, \ref{se:expansion}, we start by constructing pullbacks of the metric and analyzing their behavior. We use this together with the specific form of the gauge to show log-smoothness of the metric at the conformal boundary. The last part is then to extract the exact asymptotics, using log-smoothness of the metric together with the structure of the (ungauged) Einstein equation \eqref{eq:Einstein}. We also show that the metric can be brought into block diagonal form. This is the content of Theorem \ref{thm:expansion} and our second main result.
\subsection*{Acknowledgments}
The material in this paper is based on my Master's Thesis at ETH Zurich, written under the supervision of Professor Peter Hintz. I am extremely grateful for his ideas and all his help throughout the process of writing this paper. His enthusiasm for the subject was contagious, and his guidance always pointed me in the right direction to figure out the next step.

\section{Analysis}\label{se:analysis}
\subsection{Domain}
\label{se:manifold}
We consider de Sitter space $(M_{\mathrm{dS}},g_{\mathrm{dS}})$, with coordinates $s,\omega$:
\begin{equation}
M_{\mathrm{dS}}=(0, \pi)_s \times \mathbb{S}_{\omega}^n, \quad g_{\mathrm{dS}}=\frac{-\mathrm{d} s^2+\cancel{g}}{\sin^2(s)}.
\end{equation}
where $\cancel{g}$ is the standard metric on $\mathbb{S}^n$. We wish to study the behavior of perturbations to de Sitter space near the conformal boundaries.
\begin{definition}[Manifold]
    We define \begin{equation}
    M:=[0, \pi/2]_s\times \mathbb{S}^n_\omega.
    \end{equation}
    Equipped with $g_{\mathrm{dS}}$, $M^\circ=(0, \pi/2)\times\mathbb{S}^n$ is a piece of de Sitter space, and $M$ together with $s^2g_{\mathrm{dS}}$ is a pseudo-Riemannian manifold with boundary 
    \begin{equation}
        \mathcal{I}^+\cup\Sigma_0,\quad \mathcal{I}^+:=s^{-1}(0),\quad\Sigma_{0}:=s^{-1}( \pi/2).
    \end{equation}
    We call $\mathcal{I}^+$ future timelike infinity and consider the cone containing $-\partial_s$ to be future directed. Notice that $\Sigma_{0}$ is space-like.
\end{definition}
A schematic of the manifold is shown in figure \ref{fig:manifold}.\\

In order to do explicit calculations, we need explicit coordinate systems on $M$ and therefore $\mathbb{S}^n$. We cover $\mathbb{S}^n$ by two charts. Consider the stereographic projection $^\uparrow p:\mathbb{S}^n\backslash(0,...,0,-1)_{\omega}\rightarrow\mathbb{R}^n_{x}$. If we restrict $^\uparrow p$ to the set where $|^\uparrow p(w)|< 3$, we get a chart that covers the closed upper hemisphere, call it $^\uparrow\mathbb{S}^n$. Furthermore, the coefficients of the standard metric and its inverse in this coordinate system are smooth and bounded. This immediately gives us charts on any hypersurface of the form $\{s\}\times\mathbb{S}^n$ for an $s\in [0,\pi/2]$. We call $^\uparrow M := [0, \frac{\pi}{2}]\times ^\uparrow\mathbb{S}^n$ and on it we use the map $^\uparrow P(s,\omega):= (s,^\uparrow p(\omega))$. Performing the analogous construction that excludes the north pole, $(0,...,0,1)$, we get the chart $(^\downarrow M,^\downarrow P)$ around the south pole, and those two charts together cover $ M $. We fix a partition of unity $^\uparrow\varphi +{}^\downarrow\varphi=1$ subordinate to $^\uparrow M \cup{}^\downarrow M $, where we require $^\uparrow\varphi(s,\omega)>1/3$ for $\omega$ with $|^\uparrow p(\omega)|<5/2$ and analogously $^\downarrow\varphi(s,\omega)>1/3$ if $\omega$ satisfies $|^\downarrow p(\omega)|<5/2$. (This ensures a margin around the hemispheres, where we can divide out the cutoffs). Notice that the change of coordinates is non-singular on the intersection of the two charts.\\

Our methods allow us to solve the equation on any domain on which $s\leq s_0$ for an $s_0\in (0,\pi/2)$. This would require us to define the chart domain on $^\uparrow M$ as the set, where$|^\uparrow p(\omega)|$ is bounded by a function of $s_0$ that diverges for $s_0$ going to $\pi/2$. For the sake of concreteness, we define the domain on which we will solve the equation. The concrete value of $s_0$ is of no consequence, as any finite in time gaps will be bridged by finite in time results, namely \cite{ChoquetBruhatLocalEinstein}.
\begin{definition}[Domain]
    Let $s_0=0.1$. We define 
    \begin{equation}
        \Omega:=[0,s_0]_s\times\mathbb{S}^n_\omega\subset M
    \end{equation}
    Its boundary is 
    $$
    \Sigma\cup\mathcal{I}^+,\quad \Sigma:=s^{-1}(s_0).
    $$
    Thus, $\Sigma$ is space-like for $g_{\mathrm{dS}}$. Lastly, set $^\uparrow\Omega:=\Omega\cap{}^\uparrow M$ and $^\downarrow\Omega:=\Omega\cap{}^\downarrow M$.
\end{definition}
\begin{figure}[ht!]
    \centering
    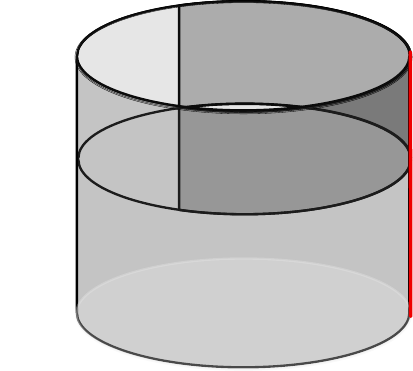
    \caption{Schematic drawing of the manifold $M$. The circles represent the level sets of $s$ of the form $\{s\}\times\mathbb{S}_\omega^n$. The red line corresponds to the set $\{[0,\frac{\pi}{2}]_s\times(0,\dots,0,1)_\omega\}$, the north pole of the sphere for all $s$. The chart domain containing the north pole, $^\uparrow\Omega$, is shown. In this schematic, $^\downarrow \Omega$ would correspond the same area as $^\uparrow \Omega$, but mirrored horizontally.}
    \label{fig:manifold}
\end{figure}

In this paper, we will be using the frames $e^\mu$ in each coordinate system
\begin{equation}
\label{eq:framedefinition}
     e^0=\frac{\mathrm{d} s}{s}, e^i=\frac{\mathrm{d} x^i}{s}, \quad e_0=s \partial_s, e_i=s \partial_{x^i},
\end{equation}
where we will always use Greek letters for spacetime indices $\mu=0,1,...,n$ and Latin letters for space indices $i= 1,...,n$.
\\

Consider the Laurent series of $\frac{1}{\sin^2(s)}$ around 0 given by $\frac{1}{s^2}(1 + s^2\mathcal{C}^\infty([0,\pi/2])$. We use this to write \begin{equation}
\label{eq:gunderline}
    g_{\mathrm{dS}} = \underline{g} + \overline{g} =  \frac{-\mathrm{d} s^2+ \cancel{g}}{s^2} + s^2\mathcal{C}^\infty([0,\pi/2])\times\frac{-\mathrm{d} s^2+ \cancel{g}}{s^2}
\end{equation}
For small $s$, the correction $\overline{g}$ is arbitrarily small (relative to $\frac{\mathrm{d}s}{s},\frac{\mathrm{d}\omega}{s}$) and for the analysis near $\mathcal{I}^+$ we can therefore perform the calculation with $\underline{g}$ modulo a small error term.
\subsection{Bundles and vector fields}
In order to do analysis close to the conformal boundary, we define the 0-bundle $^0T^*M$ and its dual $^0TM$ over the manifold $M$ as
\begin{align}
    ^0T^*M:= \mathbb{R}\frac{ds}{s}\oplus s^{-1}T^*\mathbb{S}^n\\
    ^0TM := \mathbb{R}(s\partial_s)\oplus sT\mathbb{S}^n.
\end{align}
By this we mean that $^0TM=\underline{\mathbb{R}}\oplus T\mathbb{S}^n$, where $\underline{\mathbb{R}}=M\times\mathbb{R}$ is the trivial bundle over $M^\circ$, an element $(a,X)$, where $a\in\mathbb{R},X\in T\mathbb{S}^n$ is identified with $as\partial_s+sX\in TM^\circ$. In its dual, we perform the analogous construction.\\

With this we can formulate the following Lemma:
\begin{lemma}
    The metrics $g_{\mathrm{dS}}$ and $\underline{g}$ are smooth Lorentzian signature sections of $S^2\ ^0T^*M$ over $M$ with 
    \begin{equation}
        g_{\mathrm{dS}}-\underline{g} \in s^2C^{\infty}(M;S^2\ ^0T^*M)
    \end{equation}
    Furthermore, $\frac{\mathrm{d}s}{s}$ and $-s\partial_s$ are uniformly time-like for $g_{\mathrm{dS}}$ on $ M $.
\end{lemma}
The set of smooth sections of $^0TM$ is called $\mathcal{V}_0(M)$ and its elements are precisely the vector fields that vanish at $s=0$. We call them 0-vector fields. They are spanned over $\mathcal{C}^\infty(M)$ by $s\partial_s$ and $sX$ with $X\in\mathcal{V}(\mathbb{S}^n)$. $\mathcal{V}_0(M)$ is a subspace of the larger set of all vector fields that are tangent to $\mathcal{I}^+$, named $\mathcal{V}_b(M)$. Its elements are spanned over $\mathcal{C}^\infty(M)$ by $s\partial_s$ and $X\in\mathcal{V}(\mathbb{S}^n)$. We investigate the structural properties of these spaces: 
\begin{lemma}[Structure of 0- and b-vector fields]
\label{le:ideal} 
    Let $V\in \mathcal{V}_0(M)$ and $W\in \mathcal{V}_b(M)$. Then we have the following.
    \begin{enumerate}
        \item The commutator $[V,W]$ lies in $\mathcal{V}_0(M)$.
        \label{it:1}
        \item Writing $[s\partial_s,V]$ as $as\partial_s + sX$ with $a\in\mathcal{C}^\infty(M)$ and $X\in\mathcal{C}^\infty([0,\tfrac{pi}{2}]_s;\mathcal{V}(\mathbb{S}^n))$, we have $a\in s\mathcal{C}^\infty(M)$.
        \label{it:2}
    \end{enumerate}
\end{lemma}
\begin{proof}
    We work in $^\uparrow M$ and use the corresponding coordinates. Write $V=V^0s\partial_s + V^{i}s\partial_i$ and $W=W^0s\partial_s + W^{i}\partial_i$ for $V^0,V^{i},W^0,W^{i}\in\mathcal{C}^\infty(M)$ and $\partial_i$ for $i=1,\dots,n$ the spatial coordinate derivatives. After a short computation, we can see that each term of the commutator $[V,W]$ is of the form $s f\partial_s$ or $s f\partial_i$ for $f\in\mathcal{C}^\infty(M)$. Performing the same computation on $^\downarrow M$, this proves \ref{it:1}. Computing $[s\partial_s,V]$, we see that the only term proportional to $ \partial_s$ is given by $(\partial_sV^0)s^2\partial_s=s\cdot(\partial_sV^0)s\partial_s$, which is the content of \ref{it:2}.
\end{proof}
This result is the structural reason that the solutions to wave equations of 0-type are regular under applications of b-vector fields (which are stronger). We use this when proving higher order energy estimates in section \ref{se:stability}.
\\

While performing the calculations later on, we will often work in the following splittings of $S^2\ ^0T^*M$:
\begin{align}
    S^2\ ^0T^*M &= \mathbb{R}\frac{\mathrm{d}s^2}{s^2}\oplus \left(2 \frac{\mathrm{d}s}{s}\otimes_s s^{-1} T^*\mathbb{S}^n \right) \oplus s^{-2}S^2T^*\mathbb{S}^n
    \label{eq:splitting1}
    \\
    S^2T^*\mathbb{S}^n&= \mathbb{R} g_{(0)} \oplus \text{ker tr}_{g_{(0)}}.\label{eq:splitting2}
\end{align}
Here $g_{(0)}$ is determined dynamically from initial data, and varies in each step of the iteration scheme used to solve the gauge-fixed Einstein equations.
\\

The definition of the obstruction tensor in our case is the analogous definition as in \cite[Theorem~2.1]{obstruction}, the variant described in the footnote. For our concrete setting, this gives the following characterization.
\begin{definition}[Obstruction tensor]
    Let $n\geq4$ be even and $g_{(0)}$ be a Riemannian metric on $\mathbb{S}^n$. Let $g_+$ be a Lorentzian metric defined on $(0,\epsilon)_s\times\mathbb{S}^n$, such that $s^2g_+$ is smooth across $\{0\}\times\mathbb{S}^n$ and equals $g_{(0)}$ at the boundary. Assume further that $\operatorname{Ric}g_+-\Lambda g_+=\mathcal{O}(s^{n-2})$\footnote{By $\mathcal{O}(s^{n-2})$ we mean that each component of the tensor is of this order in a smooth chart. Also, remember $\Lambda =n$}. The obstruction tensor $O$ is given by 
    \begin{equation}
    \label{eq:obstructiontensor}
        O=c_n\operatorname{tf}(s^{2-n}(\operatorname{Ric}g_+-\Lambda g_+)\left|_{T(\{0\}\times\mathbb{S}^n)}\right.).\quad c_n\neq 0,
    \end{equation}
    where we denote by $\mathrm{tf}$ the tracefree part with respect to $g_{(0)}$. The obstruction tensor is independent of the choice of $g_+$.
\end{definition}
\subsection{Differential operators}
We define the set of (scalar) 0-differential operators of order $m$, $\operatorname{Diff}_0^m(M)$, as finite sums of up to $m$-fold compositions of elements of $\mathcal{V}_0$. If $E \rightarrow M$ is a vector bundle, we write $\operatorname{Diff}_{0}^m(M ; E)$ for operators $\mathcal{C}_{\mathrm{c}}^{\infty}\left(M^{\circ} ; E\right) \rightarrow \mathcal{C}_{\mathrm{c}}^{\infty}\left(M^{\circ} ; E\right)$ that in each local trivialization of $E$ are matrices of operators of class $\operatorname{Diff}_{0}^m$. We can then also consider weighted spaces, for example $s^{\beta}\operatorname{Diff}_{0}^m(M ; E)$, where in the trivializations, the operators are matrices with coefficients of the form $s^{\beta}L$ with $L\in \operatorname{Diff}_{0}^m$. Analogously, we define $b$-differential operators $\operatorname{Diff}_b^m(M;E)$ and its weighted counterparts using $\mathcal{V}_b$.
\\

Viewing $s\mathcal{V}(\mathbb{S}^n)$ as a subspace of $\mathcal{V}_0(M)$, we call $\operatorname{Diff}^m_{T,0}(M)$ the space of finite sums of up to $m$-fold compositions of elements in $s\mathcal{V}(\mathbb{S}^n)$, and then define the generalizations to differential operators acting on bundles. The purpose of this is to be able to write the leading-order behavior at $\mathcal{I}^+$ of an element of $L\in \operatorname{Diff}_{0}^m(M ; E)$ in the following way. Each such $L$ can be written as
$$
L=\sum_{l=0}^mL_l(s)(s\partial_s)^{m-l},
$$
where $L_l(s)$ are one-parameter families of operators of
class $\operatorname{Diff}^l_{T,0}(M;E)$. We define the \textit{indicial operator} $I(L)$ of $L$ as 
\begin{equation}
    I(L):=\sum_{l=0}^m L_l^{(0)}(0) (s\partial_s)^{m-l},
\end{equation}
where $L_l^{(0)}(s)$ is the $0$-order part of $L_l(s)$. This captures the leading-order behavior of $L$ as a b-differential operator, 
$$
L-I(L)\in s\operatorname{Diff}_b^m(M;E).
$$
The \textit{indicial family} $I(L,\lambda)$ is then obtained by formally conjugating $I(L)$ by the Mellin transform in $s$, so 
\begin{equation}
    I(L,\lambda):=\sum_{l=0}^m L_l^{(0)}(0) \lambda^{m-l}.
\end{equation}
In local coordinates and the induced local trivialization, we can directly compute the indicial operator in the following way. We can write operators $L \in \operatorname{Diff}_{0}^m(M ; E)$ as
\begin{equation}
L=\sum_{j+|\alpha| \leq m} \ell_{j \alpha}(s, \omega)(s\partial_x)^\alpha\left(s \partial_s\right)^j, 
\end{equation}
with $\ell_{j \alpha}$ matrices with coefficients in $\mathcal{C}^{\infty}(M)$. Here we mean that for a multi-index $\alpha\in \mathbb{N}^n$ $\left(s\partial_x\right)^{\alpha}=\left( s\partial_{x_1}\right)^{\alpha_1}...\left( s\partial_{x_n}\right)^{\alpha_n}$. The indicial operator is then given by
\begin{equation}
I(L)=\sum_{j \leq m} \ell_{j 0}(0, \omega)\left(s \partial_s\right)^j \in \operatorname{Diff}_{b}^m(M;E) .
\end{equation}
Now, $\ell_{j0}-\ell_{j0}|_{s=0}$ are matrices with coefficients in $s\mathcal{C}^\infty(M)$. (We often regard functions $f$ on $\mathcal{I}^+$ as functions on $M$ via $f(s,\omega):= f(\omega)$.) As $s\partial_i$ are in $s\mathcal{V}_b$, we see directly that $L-I(L)\in s\operatorname{Diff}_b^m(M;E).$ The indicial family is then locally
\begin{equation}
I(L, \lambda)=\sum_{j \leq m} \ell_{j 0}(0,\omega) \lambda^j.
\end{equation}
This is a matrix with polynomial entries in $\lambda$ and the roots of the determinant are called \textit{indicial roots}. They do in general depend on $\omega$, but all indicial roots in this paper will be constant along $\mathcal{I}^+$.

\subsection{Weighted Sobolev spaces on the domain}
We fix on $M$ and $\mathcal{I}^{+}$ the densities $\left|\frac{\mathrm{d} s}{s}\mathrm{d}\cancel{g} \right|$ and $\left|\mathrm{d}\cancel{g}\right|$, respectively. For $s_0$ fixed, we work in the coordinate systems described above for $^\uparrow\Omega$ and $^\downarrow\Omega$. 
\\

Then for $ \beta \in \mathbb{R}$ and $k \in \mathbb{N}_0$, we define $s^\beta H_b^k(\Omega)$ to be the space of elements of $L_{\mathrm{loc}}^2\left(\Omega^{\circ}\right)$ with finite norm
\begin{equation}
\begin{aligned}
    \|u\|_{s^\beta H_b^k(\Omega)}:=\sum_{j+|\gamma| \leq k}\left\|s^{-\beta}\left(s \partial_s\right)^j \partial_{x}^\gamma ({}^\uparrow\varphi u)\right\|_{L^2(^\uparrow\Omega)}\\
    +\left\|s^{-\beta}\left(s \partial_s\right)^j \partial_{x}^\gamma ({}^\downarrow\varphi u)\right\|_{L^2(^\downarrow\Omega)}
    \end{aligned}
\end{equation}
As above, $\partial_x^\gamma=\partial_{x_1}^{\gamma_1}\ldots\partial_{x_n}^{\gamma_n}$ for the derivative along the coordinate directions $\partial_{x_i}$ of the respective coordinate systems. We define analogously the space $s^\beta H_0^k(\Omega)$ but using $\left(s \partial_s\right)^j (s\partial_{x})^\gamma$ to test. We similarly define $H^k(\mathcal{I}^+)$ with norm
\begin{equation}
    \|u\|_{H^k(\mathcal{I}^+)}:= \sum_{|\gamma| \leq k} \left\|\partial_{x}^\gamma ({}^\uparrow\varphi u)\right\|_{L^2(^\uparrow\mathcal{I}^+)}+\left\|\partial_{x}^\gamma ({}^\downarrow\varphi u)\right\|_{L^2(^\downarrow\mathcal{I}^+)}.
\end{equation}
The space $s^\beta \mathcal{C}_b^k(\Omega)$ is defined using the $\mathcal{C}_b^0(\Omega)$-norm which is defined as the sup norm on the space $\mathcal{C}_b^0(\Omega)$ of all bounded continuous functions on $\Omega \backslash \mathcal{I}^{+}$. 
\\

For $v$ a section of the bundles ${}^0TM$, ${}^0 T^*M$, $s^{-2}S^2T^*\mathbb{S}^n$ and $S^2\ {}^0 T^*M$ over $M, \Sigma$ or $\mathcal{I}^+$, we define its norm by writing $v={}^\uparrow\varphi v+{}^\downarrow\varphi v$. We trivialize the bundles in the frames $e_\mu$, $e^\mu$, $e^j\otimes_s e^k$ or $e^\mu\otimes_s e^\nu$ of the coordinate charts. The norm of $v$ is then the sum of the norms of the coefficients of $v$ in these frames. This works for all norms in the section above.
\\

For operator classes with non-smooth coefficients, we define, given a function space $\mathscr{F}$ like the ones above, the class $\mathscr{F}\operatorname{Diff}_0^k(M)$ as finite sums of operators $aL$ with $a\in \mathscr{F}$ and $L\in \operatorname{Diff}_0^k(M)$. Then for operators acting on the bundle $S^2\ {}^0 T^*M$, $\mathscr{F}\operatorname{Diff}_0^k(M;S^2\ {}^0T^*M)$ is defined as operators that in each local trivialization are matrices with coefficients in $\mathscr{F}\operatorname{Diff}_0^k(M)$. In order to define norms on these operator spaces, we write $L={}^\uparrow\varphi L+{}^\downarrow\varphi L$ for $L \in \mathscr{F}\operatorname{Diff}_0^k(M;S^2\ {}^0T^*M)$. By trivializing $S^2\ {}^0T^*M$ on the respective chart domain using $e^\mu\otimes_s e^\nu$, each part can be written uniquely as
\begin{equation}
    ^\uparrow\varphi L = \sum_{j+|\gamma| \leq m} 
  {}^\uparrow\ell_{j \gamma}\left(s \partial_x\right)^\gamma\left(s \partial_s\right)^j,
\end{equation}
with $^\uparrow\ell_{j \gamma}$ matrices with coefficients $(^\uparrow\ell_{j \gamma})_{ab}$ in $\mathscr{F}$. The same is true for $\restr{L}{^\downarrow\Omega}$. We then define the norm of $L$ as
\begin{equation}
    \|L\|_{ \mathscr{F}\operatorname{Diff}_{0}^m(\Omega)}:=\sum_{j+|\gamma| \leq m}\sum_{ab}\left(\left\|(^\uparrow\ell_{j \gamma})_{ab}\right\|_{\mathscr{F}}+\left\|(^\downarrow\ell_{j \gamma})_{ab}\right\|_{\mathscr{F}} \right).
\end{equation}
Next, we discuss the algebra properties of the b-Sobolev spaces defined above. We write the following Lemmas for functions $u$ on $\Omega$ or $\mathcal{I}^+$. The obvious generalizations for sections of the bundles ${}^0TM$, ${}^0 T^*M$, $s^{-2}S^2T^*\mathbb{S}^n$ and $S^2\ {}^0 T^*M$ can be proved in the same manner. We fix a finite set $\mathscr{V}$ of spanning vector fields of $\mathcal{V}(\mathbb{S}^n)$, e.g. rotation vector fields.
\begin{lemma}[Sobolev embedding, restriction and algebra properties]
\label{le:2.9}
    See \cite[Lemma~2.9]{hintz2024stability}. cWrite $d_m:=\lceil\frac{m+1}{2}\rceil$ for the smallest integer larger than $\frac{m}{2}$. We write $D_b^p$ for any $p$-fold composition of $s\partial_s$ and elements of $\mathscr{V}$, or only $\mathscr{V}$ acting on functions $u$ on $\Omega$ or $\mathcal{I}^+$ respectively.\\
    \begin{enumerate}
    \item \textnormal{(Sobolev embedding.)} For every $k\in \mathbb{N}_0$ there exist constants $C_{\Omega,k}^S,\ C_{\mathcal{I}^+,k}^S$ such that
    \begin{equation}
    \label{eq:sobolev embedding}
        \|u\|_{\mathcal{C}^k_b(\Omega)}\leq C_{\Omega,k}^S\|u\|_{H^{d_{n+1}+k}_b(\Omega)},\quad  \|u\|_{\mathcal{C}^k(\mathcal{I}^+)}\leq C_{\mathcal{I}^+,k}^S\|u\|_{H^{d_{n}+k}(\mathcal{I}^+)}.
    \end{equation}
    \item \textnormal{(Restrictions).} Let $s_1\in(0,s_0]$ and write $\Sigma_{s_1}:=\Omega\cap\{s=s_1\}$. Then the restriction map $u \mapsto u\left|_{\Sigma_{s_1}}\right.$ defines a bounded linear map
    \begin{equation}
        H^k_b(\Omega) \rightarrow H^{k-1}(\Sigma),\quad k\in\mathbb{N}.
    \end{equation}
    \item \textnormal{(Estimates for products.)} For every $k\in \mathbb{N}_0$ we have constants $C_{k,\Omega}$, $C_{k,\mathcal{I}^+}$ and $C_{k,\Omega,\mathcal{I}^+}$ such that for all functions $u_1,\ u_2$ on $\Omega$ and $v_1,\ v_2 $ on $\mathcal{I}^+$ and $a,b\in\mathbb{N}_0$ with $a+b=k$ we have
\begin{equation}
    \begin{aligned}
    \label{eq:productest}
\left\|\left(D_{\mathrm{b}}^a u_1\right)\left(D_{\mathrm{b}}^b u_2\right)\right\|_{L^2(\Omega)} \leq C_{k, \Omega}&\left(\left\|u_1\right\|_{H_{\mathrm{b}}^{ d_{n+1}}(\Omega)}\left\|u_2\right\|_{H_{\mathrm{b}}^k(\Omega)}\right.\\
 &\left.+\left\|u_1\right\|_{H_{\mathrm{b}}^{k+ d_{n+1}}(\Omega)}\left\|u_2\right\|_{L^2(\Omega)}\right) \\
\left\|\left(D_{\mathrm{b}}^a v_1\right)\left(D_{\mathrm{b}}^b v_2\right)\right\|_{L^2( \mathcal{I}^+)} \leq C_{k,  \mathcal{I}^+}&\left(\left\|v_1\right\|_{H^{ d_{n}}( \mathcal{I}^+)}\left\|v_2\right\|_{H^k( \mathcal{I}^+)}\right.\\
&\left.+\left\|v_1\right\|_{H^{k+ d_{n}}( \mathcal{I}^+)}\left\|v_2\right\|_{L^2( \mathcal{I}^+)}\right) \\
\left\|\left(D_{\mathrm{b}}^a v_1\right)\left(D_{\mathrm{b}}^b u_2\right)\right\|_{L^2(\Omega)} \leq C_{k, \Omega,  \mathcal{I}^+}&\left(\left\|v_1\right\|_{H^{ d_{n}}( \mathcal{I}^+)}\left\|u_2\right\|_{H_{\mathrm{b}}^k(\Omega)}\right.\\
&\left.+\left\|v_1\right\|_{H^{k+ d_{n}}(\mathcal{I}^+)}\left\|u_2\right\|_{L^2(\Omega)}\right)
\end{aligned}
\end{equation}
\item \textnormal{(Estimate for nonlinear expressions)} Let $\delta>0$ and $F:[-\delta,\delta]\rightarrow \mathbb{R}$ be a smooth function with $F(0)=0$. Then for all $u\in H^k_b(\Omega)$ with $\|u\|_{H^{d_{n+1}}_b(\Omega)}\leq \frac{\delta}{C^S_{d_{n+1},\Omega}}$ we have $F(u)\in H^k_b(\Omega)$ with 
\begin{equation}
\label{eq:nonlinest}
    \|F(u)\|_{H^k_b(\Omega)}\leq C_{F,\Omega,k}\|u\|_{H^{k+d_{n+1}}_b(\Omega)}.
\end{equation}
Similarly, if $v\in H^k_b(\mathcal{I}^+)$ with $\|v\|_{ H^{d_n}_b(\mathcal{I}^+)}\leq \frac{\delta}{C^S_{d_n,\mathcal{I}^+}}$, then $f(v)\in  H^k_b(\mathcal{I}^+)$ with norm bounded by a constant times $\|v\|_{ H^{k+d_n}_b(\mathcal{I}^+)}$
\end{enumerate}
\end{lemma}
One could also prove sharper estimates, where in statements (3) and (4) of the above Lemma, $d_{n+1}$ and $d_n$ could be set to 0 using methods such as in \cite{Taylor}. We keep to these simpler versions so that the proof of several propositions in the paper proceed more similarly to \cite{hintz2024stability}. Also, the simpler statements have shorter, self-contained proofs that are analogous to the proofs in \cite{hintz2024stability}.
These estimates can also be easily generalized to weighted spaces. For example the Sobolev embedding estimate would become $$\|u\|_{s^\beta \mathcal{C}^k_b(\Omega)}\leq C_{\Omega,k,\beta}^S\|u\|_{s^\beta H^{k+d_{n+1}}_b(\Omega)}.$$ \\

Furthermore, we can let the nonlinearity also depend on $p\in\Omega$ explicitly (in a $\mathcal{C}^\infty$ way), as long as $F(p,u=0)=0$ for all $p\in \Omega$. This works because $\Omega$ is compact; otherwise one would need boundedness in this variable of all derivatives. We will use these generalizations frequently below, but do not state them here explicitly.\\

Next, we discuss smoothing operators.
\begin{lemma}[Smoothing Operators] Let $\beta\in\mathbb{R}$. There exist continuous linear maps 
$$
S_\theta:s^\beta L^2(\Omega)\rightarrow s^\beta H^\infty_b(\Omega),\quad \theta>1,
$$
so that 
$$
\begin{aligned}
k \leq k^{\prime} &\Longrightarrow\left\|S_\theta u-u\right\|_{ s^\beta H_{\mathrm{b}}^k} \leq C_{k, k^{\prime}} \theta^{k-k^{\prime}}\|u\|_{ s^\beta H_{\mathrm{b}}^{k^{\prime}}} \\
k \geq k^{\prime} &\Longrightarrow\left\|S_\theta u\right\|_{s^\beta H_{\mathrm{b}}^k}  \leq C_{k, k^{\prime}} \theta^{k-k^{\prime}}\|u\|_{ s^\beta H_{\mathrm{b}}^{k^{\prime}}}.
\end{aligned}
$$
There exists continuous linear maps $L^2(\mathcal{I}^+)\rightarrow H^\infty_b(\mathcal{I}^+)$ with analogous properties.
\label{le:smoothingops}
\end{lemma}
The proof is analogous to \cite{hintz2024stability}. For the Nash-Moser iteration, we need smoothing operators on sections of the bundles ${}^0 T^*M$, $s^{-2}S^2T^*\mathbb{S}^n$, and $S^2\ {}^0 T^*M$. But, since the construction is local, using a partition of unity, we can reduce these cases to smoothing operators on functions, the coefficients of sections in the local frames $e_\mu$. So the obvious generalization of this Lemma to the required bundles holds as well.

\section{Gauge-fixed Einstein equation and stability of de Sitter}\label{se:stability}
We study perturbations to the de Sitter space metric $g_{\mathrm{dS}}$ on the domain $\Omega$ with initial data posed at $\Sigma$. We wish to find a solution $g$ to the Einstein field equations.
\\

We work in a generalized harmonic gauge $\Upsilon\left(g ; g_0\right)+E_{g_0}\left(g-g_0\right)=0$. Here, $g_0$ is a dynamically chosen `background metric' and
\begin{align} 
\Upsilon\left(g ; g_0\right) &:=g\left(g_0\right)^{-1} \delta_g \mathrm{G}_g g_0 \\ E_{g_0} h  &:=\chi e^0\left(-2 \operatorname{tr}_{g_0} h-h\left(e_0, e_0\right)\right), \quad e_0:=s \partial_s, \quad e^0:=\frac{\mathrm{d} s}{s} .
\end{align}
Here, $(\delta_gh)_\mu=-g^{\kappa\lambda}h_{\mu\kappa;\lambda}$ and $G_gh=h-\tfrac{1}{2}g\operatorname{tr}_gh$. In coordinates, the expression for $\Upsilon$ reads $ \Upsilon\left(g ; g_0\right)_\mu=g_{\mu \nu} g^{\kappa \lambda}\left(\Gamma(g)_{\kappa \lambda}^\nu-\Gamma\left(g_0\right)_{\kappa \lambda}^\nu\right)$. We fix two cutoffs 
\begin{align}
\chi&=\chi(s) \in \mathcal{C}_{\mathrm{c}}^{\infty}([0, \tfrac{1}{2} s_0)),\left.\quad \chi\right|_{\left[0, \tfrac{1}{4} s_0\right]}=1, \\
\tilde{\chi}&=\tilde{\chi}(s) \in \mathcal{C}_{\mathrm{c}}^{\infty}\left(\left[0, s_0\right)\right),\left.\quad \tilde{\chi}\right|_{\left[0, \tfrac{1}{2} s_0\right]}=1,
\end{align}
Let $\delta_g^*$ be the symmetric gradient defined as  $(\delta_g^*\omega)_{\mu\nu}=\tfrac{1}{2}(\omega_{\mu;\nu}+\omega_{\nu;\nu})$. We define the modification $\tilde{E}$ and modify the symmetric gradient via
\begin{equation}
\label{eq:symmetricgradient}
\tilde{\delta}_g^*=\delta_g^*+\tilde{E}, \quad \tilde{E} \omega:=\chi\left(2 \omega\left(e_0\right) e^0 \otimes e^0-4 e^0 \otimes_s \omega\right).
\end{equation}
With that, we can now define the gauge-fixed Einstein operator $P$. It takes as input a section $h_0$ of $s^{-2}S^2T^*\mathbb{S}^n$, a section $\tilde{h}$ of $S^2\ ^0T^*M$ and a gauge one form $\theta$. The full metric we solve for is then $g=g_{\mathrm{dS}} + \chi h_0 + \tilde
h$, with background metric $g_0=g_{\mathrm{dS}}+\chi h_0$. The section $h_0$ captures the leading order change to $g_{\mathrm{dS}}$ on $\mathcal{I}^+$, whereas $\tilde{h}$ captures further corrections to $g_{\mathrm{dS}}$ and makes sure the full metric $g$ satisfies the initial conditions on $\Sigma$. The full gauge-fixed Einstein operator reads
\begin{equation}
\begin{aligned}
P(h_0, \tilde{h}, \theta&):=2\left( \operatorname{Ric}(g_{\mathrm{dS}}+\chi h_0+\tilde{h})-\Lambda(g_{\mathrm{dS}}+\chi h_0+\tilde{h})\right. \\
& \left.-\tilde{\delta}_{g_{\mathrm{dS}}+\chi h_0+\tilde{h}}^*(\Upsilon(g_{b}+\chi h_0+\tilde{h} ; g_{\mathrm{dS}}+\chi h_0)+E_{g_{b}+\chi h_0} \tilde{h}-\tilde{\chi} \theta)\right).
\label{eq:gaugefixedein}
\end{aligned}
\end{equation}
The choice of gauge is the $n$ dimensional generalization of the gauge used in \cite[Section~3]{hintz2024stability}. The requirement is that all indicial roots of the linearization of $P$ in $\tilde{h}$ are non-negative; we verify this for the choice \eqref{eq:symmetricgradient} of modified symmetric gradient in \ref{le:indicialroots}. The motivation for each term in \eqref{eq:gaugefixedein} is explained in \cite[Sections~1.3, 3.1.1]{hintz2024stability}. \\

We can now formulate the main Theorem of this section.
\begin{theorem}
    Let $d\in \mathbb{N}$ and $\delta_0>0$. Then there exist $D\in\mathbb{N}$ and $\epsilon>0$ such that the following holds. Let 
    $$
    \underline{h}_0, \underline{h}_1 \in H_{b}^{\infty}\left(\Sigma ; S^{2}\ ^0T^* M\right)
    $$
    and suppose that $\|\underline{h_{j}}\|_{H_{b}^D}<\epsilon, j=0,1$. Let $\beta \in(0,1)$. Then there exist

    \begin{equation}
        \begin{aligned}
            h_0 & \in H_{b}^{\infty}\left(\mathcal{I}^{+} ; s^{-2} S^2 T^* \mathbb{S}^n\right) \\
            \tilde{h} & \in s^\beta H_{b}^{\infty}\left(\Omega ; S^{2}\   ^0T^* M\right), \\
            \theta & \in s^\beta H_{b}^{\infty}\left(\Omega ;^0T^*    M\right),
        \end{aligned}
    \end{equation}
    with weighted $H_{b}^d$-norms less than $\delta_0$, so that $P\left(h_0, \tilde{h}, \theta\right)=0$, and so that $\tilde{h}$ (and thus also $\chi h_0+\tilde{h}$ ) satisfies the initial conditions $$\left.\tilde{h}\right|_{\Sigma}=\underline{h}_0,\left.\left(\mathcal{L}_{-s \partial_s} \tilde{h}\right)\right|_{\Sigma}=\underline{h}_1.$$
    \label{thm:nonlinsol}
\end{theorem}
In order to prove this Theorem using the main result of \cite{nashmoser}, we need to control the solution of the linearized problem. To wit, we calculate the linearized Einstein operator. 
\begin{equation}
\begin{gathered}
L_{h_0, \tilde{h}}:=\left.D_2 P\right|_{\tilde{h}}\left(h_0, \cdot\right)=\square_g-2 \Lambda+2 \tilde{E} \delta_g \mathrm{G}_g+2 \mathscr{R}_g+2 \tilde{\delta}_g^* \circ\left(\mathscr{E}_{g ; g_0}-E_{g_0}\right), \\
+\left(D_g \tilde{\delta}_{.}^*\right)\left(\Upsilon\left(g ; g_0\right)+E_{g_0} \tilde{h}\right) \\
\left(\mathscr{R}_g u\right)_{\mu \nu}=\operatorname{Riem}(g)_{\kappa \mu \nu \lambda} u^{\kappa \lambda}+\frac{1}{2}\left(\operatorname{Ric}(g)_{\mu \lambda} u_\nu{ }^\lambda+\operatorname{Ric}(g)_{\nu \lambda} u_\mu{ }^\lambda\right), \\
\left(\mathscr{E}_{g ; g_0} u\right)_\mu=\left(\Gamma(g)_{\kappa \nu}^\lambda-\Gamma\left(g_0\right)_{\kappa \nu}^\lambda\right)\left(g_{\mu \lambda} u^{\kappa \nu}-u_{\mu \lambda} g^{\kappa \nu}\right).
\end{gathered}
\end{equation}
Here again, $g=g_{\mathrm{dS}}+\chi h_0 +\tilde{h}$ and $g_0=g_{\mathrm{dS}}+\chi h_0$, we raise and lower indices using $g$ and $(\square_gu)_{\mu\nu}=-g^{\kappa\lambda}u_{\mu\nu;\kappa\lambda}$ is the tensor wave operator. Moreover, $\left(D_g \tilde{\delta}^*_\cdot\right) \eta=\left(D_g \delta_\cdot^*\right) \eta$, for a fixed 1-form $\eta$, maps a symmetric 2 -tensor $h$ to $\restr{\frac{\mathrm{d}}{\mathrm{d} t}\left(\delta_{g+t h}^* \eta\right)}{t=0}$.
\\

We now investigate the structure of this operator and its indicial operator.
\begin{proposition}
\label{prop:indicialcalc}
 Let $k \geq 2$ and
$$
h_0 \in H_{b}^k\left(\mathcal{I}^{+} ;s^{-2} S^2 T^* \mathrm{S}^n\right)), \quad \tilde{h} \in   s^\beta H_{b}^k\left(\Omega ; S^{2}\ ^0T^* M\right)
$$
Suppose that $\left\|h_0\right\|_{H_{b}^{ d_n+2}}<\delta_0$ and $\|\tilde{h}\|_{  s^\beta H_{b}^{ d_{n+1}+2}}<\delta_0$ for some small $\delta_0>0$ (independently of $k$). Set $g:=g_{b}+\chi h_0+\tilde{h}$ and $g_0:=g_{b}+\chi h_0$.
\begin{enumerate}
\item (Structure.) As differential operators on $\Omega$ acting on sections of $S^{2}\ ^0T^* M$, we can write
$$
\begin{aligned}
& L_{h_0, \tilde{h}}=L_{0,0}+L_{(0), h_0}+\tilde{L}_{h_0, \tilde{h}} \\
& \quad L_{0,0} \in \operatorname{Diff}_{0}^2, \quad L_{(0), h_0} \in H_{b}^{k-2}\left(\mathcal{I}^{+}\right) \operatorname{Diff}_{0}^2, \quad \tilde{L}_{h_0, \tilde{h}} \in  s^\beta H_{b}^{k-2}\left(\Omega\right) \operatorname{Diff}_{0}^2,
\end{aligned}
$$
where $L_{(0), h_0}$ and $\tilde{L}_{h_0, \tilde{h}}$ satisfy the tame estimates
\begin{equation}
\label{eq:Lweakesti}
\begin{aligned}
&\left\|L_{(0), h_0}\right\|_{H_{b}^{k-2} \operatorname{Diff}_{\mathrm{0}}^2} \leq C_k\left\|h_0\right\|_{H_{b}^{k+ d_n}}, \\
&\left\|\tilde{L}_{h_0, \tilde{h}}\right\|_{  s^\beta H_{b}^{k-2} \operatorname{Diff}_{0}^2} \leq C_k\left(\left\|h_0\right\|_{H_{b}^{k+ d_n}}+\|\tilde{h}\|_{  s^\beta H_{b}^{k+ d_{n+1}}}\right)
\end{aligned}
\end{equation}
\item (Indicial operator.) Write
\begin{equation}
\label{eq:gonbound}
 g_{(0)}:=\cancel{g}+h_{(0)}, \quad h_{(0)}:=s^2 h_0  
\end{equation}
for the (rescaled) $s^{-2} S^2 T^* \mathbb{S}^n$ part of $\left.g\right|_{\mathcal{I}^{+}}$. In the splitting \eqref{eq:splitting1}, the indicial operator of $L_{h_0, \tilde{h}}$, recalling $\Lambda=n$, is given by
\begin{equation}
\begin{aligned}
I_{g_{(0)}}&\left(s \partial_s\right):=
\left(s \partial_s\right)^2-3 s \partial_s\\
\\
+&\left(\begin{array}{ccc}
-4 s \partial_s+2(n+2) & 0 & 2\left(s \partial_s-2\right) \operatorname{tr}_{g_{(0)}} \\
0 & -4 s \partial_s+3(n+1) & 0 \\
-2g_{(0)} & 0 & 2g_{(0)} \operatorname{tr}_{g_{(0)}}
\end{array}\right)
\end{aligned}
\end{equation}
With this, we mean 
\begin{equation}
\begin{aligned}
& L_{h_0, \tilde{h}}-I_{g_{(0)}}\left(s \partial_s\right)=R_0+\tilde{R}_{h_0, \tilde{h}}, \\
& \quad R_0 \in s \operatorname{Diff}_{b}^2, \quad \tilde{R}_{h_0, \tilde{h}} \in s^\beta H_{b}^{k-2} \operatorname{Diff}_{b}^2,
\end{aligned}\\
\end{equation}
where we have the following estimate
\begin{equation}
\left\|\tilde{R}_{h_0, \tilde{h}}\right\|_{s^\beta H_{}^{k-2} \operatorname{Diff}_{b}^2} \leq C_k\left(\left\|h_0\right\|_{H_{b}^{k+d_n}}+\|\tilde{h}\|_{ s^\beta H_{b}^{k+d_{n+1}}}\right) .
\end{equation}
\end{enumerate}
\end{proposition}
\begin{proof}
    We split $L = {}^\uparrow\varphi L+ {}^\downarrow\varphi L$. Both terms can then be examined separately in the respective coordinate systems and frames. We show the calculations for $^\uparrow\Omega$, the lower part is completely analogous. For a function $f$ that is only defined on $^\uparrow\Omega$, when we say that it is in $s^\alpha H^l_b$ we mean that $^\uparrow\varphi f$ is in $s^\alpha H^l_b(\Omega)$, and similarly for $\mathcal{I}^+$. 
    \\
    
    We therefore have the coordinates $(s,x)$ and work in the frame
    \begin{equation}
    e^0=\frac{\mathrm{d} s}{s}, e^i=\frac{\mathrm{d} x^i}{s}, \quad e_0=s \partial_s, e_i=s \partial_{x^i}.
    \end{equation}
    We use Greek letters for space-time indices $\mu= 0,1,...,n$ and Latin letters for space indices $i=1,...,n$. The components of the metric are labeled 
    $$
    g_{\mu\nu}=g(e_\mu,e_\nu)=(g_{\mathrm{dS}})_{\mu\nu} + \chi (h_0)_{\mu\nu} + \tilde{h}_{\mu\nu},
    $$
    where $(h_0)_{\mu\nu}$ and $\tilde{h}_{\mu\nu}$ are real valued functions of class $H^k$ and $s^\beta H^k_b$ respectively, and $(h_0)_{\mu\nu}=0$ unless both indices are nonzero. In the splitting \eqref{eq:splitting1} and with dividing $g_{\mathrm{dS}}$ up like in \eqref{eq:gunderline} we have 
    \begin{equation}
        g_{\mu\nu} = \left(\underline{g}_{\mu\nu}+\chi(h_0)_{\mu\nu}\right) + \overline{g}_{\mu\nu} + \tilde{h}_{\mu\nu} \in (-1,0,g_{(0)}) + s^2\mathcal{C}^\infty + s^\beta H^k_b.
    \end{equation}
    Now, let $g^{\mu\nu}$ be the components of the inverse. They are given by Cramer's rule, so are rational functions of the $g_{\mu\nu}$ with denominator being the determinant of $g_{\mu\nu}$. Because of Sobolev embedding, $(h_0)_{\mu\nu}$ and $\tilde{h}_{\mu\nu}$ are small in $L^\infty$ on $^\uparrow\Omega$ and therefore also on the support of $^\uparrow\varphi$. Therefore, the inverse is well defined on the support of $^\uparrow\varphi$, because the determinant of $(g_{\mathrm{dS}})_{\mu\nu}$ is uniformly bounded away from 0. We can write $g^{-1}= g_0^{-1} - g_0^{-1}\tilde{h}g^{-1}$, and can use our estimate for nonlinear expression \eqref{eq:nonlinest} to see 
    \begin{equation}
    \left(g^{-1}\right)^{\mu \nu}-\left(g_0^{-1}\right)^{\mu \nu} \in s^\beta H_b^k
    \end{equation}
    where their norms are bounded by $C_k\|\tilde{h}\|_{H^{k+d_{n+1}}_b}$. Similarly, $g_0^{-1}= g_{\mathrm{dS}}^{-1}-g_{\mathrm{dS}}^{-1}h_0g_0^{-1}$ and therefore
    \begin{equation}
        \left(g_0^{-1}\right)^{\mu \nu}-\left(g_{\mathrm{dS}}^{-1}\right)^{\mu \nu} \in H_b^k
    \end{equation}
    with norm bounded by  $C_k\|h_0\|_{H^{k+d_n}_b}$. Finally, we can split the contribution from $g_{\mathrm{dS}} +h_0$ into a leading order  contribution from $\underline{g} + h_0$ and decaying terms of order $s^2\mathcal{C}^\infty$.\\
    
    With these estimates for the inverse in hand, we can write the operator $\operatorname{G}_g$ as 
    \begin{equation}
    \label{eq:indG}
\mathrm{G}_g \equiv\left(\begin{array}{ccc}
\frac{1}{2} & 0 & \frac{1}{2} \operatorname{tr}_{g_{(0)}} \\
0 & I & 0 \\
\frac{1}{2} g_{(0)} & 0 & I-\frac{1}{2} g_{(0)} \operatorname{tr}_{g_{(0)}}
\end{array}\right) \bmod s^2 \mathcal{C}^{\infty}+ s^\beta H_{  b}^k\left(\Omega\right)
\end{equation}
in the splitting of \eqref{eq:splitting1}. The first term is the indicial operator of $\operatorname{G}_g$, whereas the correction lying in $s^2\mathcal{C}^\infty$ is independent of $h_0,\tilde{h}$ and comes from $\overline{g}$. The last correction term is bounded by a constant times $\|h_0\|_{H_b^{k+d_n}} + \|\tilde h\|_{s^\beta H_b^{k+d_{n+1}}}$.
\\

In order to prove the structural properties of the operator $L_{h_0,\tilde{h}}$, we examine each of its components separately. Let us start by computing the connection coefficients
\begin{equation}\label{eq:connectiondef}
\begin{aligned}
\Gamma(g)_{\lambda \mu \nu}  =&g\left(\nabla_{e_\mu}^g e_\nu, e_\lambda\right)\\ 
=&\frac{1}{2}\left(e_\mu g_{\nu \lambda}+e_\nu g_{\mu \lambda}-e_\lambda g_{\mu \nu}\right.\\
&\left.-g\left(e_\mu,\left[e_\nu, e_\lambda\right]\right)-g\left(e_\nu,\left[e_\mu, e_\lambda\right]\right)+g\left(e_\lambda,\left[e_\mu, e_\nu\right]\right)\right).
\end{aligned}
\end{equation}
In our frame, $[e_0,e_0]=[e_i,e_j]=0$ and $[e_0,e_i]=e_i=-[e_i,e_0]$ and therefore we get
\begin{equation}
\label{eq:connectiondown}
\begin{aligned}
\Gamma(g)_{000} & =\frac{1}{2} e_0 g_{00},  \\
\Gamma(g)_{0 i 0} & =\frac{1}{2} e_i g_{00}, \\
\Gamma(g)_{00 j} & =\frac{1}{2} e_j g_{00}+g_{0 j},\\
\Gamma(g)_{0 i j} & =\frac{1}{2}\left(e_i g_{0 j}+e_j g_{0 i}-e_0 g_{i j}\right) +g_{i j}, \\
\Gamma(g)_{\ell 00} & =\left(e_0-1\right) g_{0 \ell}-\frac{1}{2} e_{\ell} g_{00},\\
\Gamma(g)_{\ell i 0} & =\frac{1}{2}\left(e_i g_{0 \ell}+e_0 g_{i \ell}-e_{\ell} g_{0 i}\right) -g_{i \ell},\\
\Gamma(g)_{\ell 0 j} & =\frac{1}{2}\left(e_0 g_{j \ell}+e_j g_{0 \ell}-e_{\ell} g_{0 j}\right),\\
 \Gamma(g)_{\ell i j} & =\frac{1}{2}\left(e_i g_{j \ell}+e_j g_{i \ell}-e_{\ell} g_{i j}\right), \\
\end{aligned}
\end{equation}
Now, using the formula for nonlinear estimates of Lemma \ref{le:2.9}, we get an estimate for the difference of the coefficients,
$$
\begin{aligned}
 \Gamma(g)_{\lambda\mu\nu}-\Gamma(g_0)_{\lambda\mu\nu} &\in s^\beta H_b^{k-1}\\
\Gamma(g)_{\lambda\mu\nu}-\Gamma(g_{\mathrm{dS}})_{\lambda\mu\nu} &\in H_b^{k-1} + s^\beta H_b^{k-1}.
\end{aligned}
$$
Raising the indices with $g^{-1}$, we get the same memberships
$$
\begin{aligned}
 \Gamma(g)_{\mu\nu}^\lambda-\Gamma(g_0)_{\mu\nu}^\lambda &\in s^\beta H_b^{k-1}\\
\Gamma(g)_{\mu\nu}^\lambda-\Gamma(g_{\mathrm{dS}})_{\mu\nu}^\lambda &\in H_b^{k-1} + s^\beta H_b^{k-1},
\end{aligned}
$$
again by the nonlinear estimates and using the considerations for the inverse above.
\\

Inspecting the operator $L_{h_0,\tilde{h}}$ in this frame, we see that each term is a 0-differential operator of order up to two, with coefficients being sums and products of $\Gamma^\lambda_{\mu\nu}$ and $e_\alpha \Gamma^\lambda_{\mu\nu}$. The only term where this needs further checking is the one involving $D_g \tilde{\delta}^*_{\cdot}$, but we can write it as
$$
\left(\left(D_g \tilde{\delta}^*_{\cdot}\right) \eta\right)_{\mu \nu}: h \mapsto-\frac{1}{2}\left(h_\mu{ }^\kappa{ }_{; \nu}+h_\nu{ }^\kappa{ }_{; \mu}-h_{\mu \nu}{ }^{; \kappa}\right) \eta_\kappa,
$$
for the one form $\eta:=\Upsilon\left(g ; g_0\right)+E_{g_0} h \in s^\beta H_b^{k-1}\left(\Omega\right)$. We then deal with each term in the same way. For example, writing the components of the Riemann tensor $\operatorname{Riem(g)}$ as $\operatorname{Riem}(g)_{\kappa \mu \nu \lambda}=g\left(e_\kappa,\left(\left[\nabla_{e_\nu}, \nabla_{e_\lambda}\right]-\nabla_{\left[e_\nu, e_\lambda\right]}\right) e_\mu\right)$, we see that it is of the desired form. Writing
$$
\operatorname{Riem}(g)=\operatorname{Riem}(g_{\mathrm{dS}})+\left(\operatorname{Riem}(g)-\operatorname{Riem}(g_0)\right)+\left(\operatorname{Riem}(g_0)-\operatorname{Riem}(g_{\mathrm{dS}})\right)
$$
we see that we can apply the nonlinear estimate on the second two terms when estimating in $H^{k-2}_b$. We get the same memberships for the differences of $\operatorname{Riem}(g)_{\kappa \mu \nu}^\lambda$ and $\operatorname{Riem}(g_0)_{\kappa \mu \nu}^\lambda$ or $\operatorname{Riem}(g_{\mathrm{dS}})_{\kappa \mu \nu}^\lambda$ as for the Connection coefficients, only with $k-1$ replaced by $k-2$. Again, those differences are bounded by $C_k\left( \|\tilde{h}\|_{s^\beta H^{k+d_{n+1}}_b(\Omega)} +\|h_0\|_{H^{k+d_n}_b(\mathcal{I}^+)}\right)$. The same reasoning can be applied to all the other terms. Then, repeating the proof for the chart on $^\downarrow\Omega$, we get the structural properties of the Proposition.
\\

Now we calculate the indicial operator. We can discard all terms involving $e_i$. Also, we can discard all contributions arising from $\overline{g}$ and $\tilde{h}$, as they will be of class $s^2\mathcal{C}^\infty$ and $s^\beta H^{k-2}_b$ respectively. Therefore, we can work with $\underline{g}+h_0$ and we can drop the terms arising from $\mathscr{E}$ and $D_g\tilde{\delta}_\cdot^*$.\\

As $e_0(h_0)_{\mu\nu}=0$ and $e_0\underline{g}_{\mu\nu}=0$, we have for the connection coefficients $\Gamma(g)_{\lambda\mu\nu}\equiv 0$ mod $s^2\mathcal{C}^\infty+s^\beta H^{k-1}_b(\Omega)$ for all $\lambda \mu \nu$ except for 
$$
\Gamma(g)_{li0}\equiv -g_{li}\quad \Gamma(g)_{0ij}\equiv g_{ij} 
$$
and therefore $\Gamma(g)^\lambda_{\mu\nu} \equiv 0$ except for
$$
\Gamma(g)^{i}_{l0} \equiv -\delta^{i}_l \quad \Gamma(g)^0_{ij} \equiv -g_{ij}$$
Notice that the coefficients are not symmetric in their lower indices, as $\Gamma(g)^{i}_{0l} \equiv 0$. 
\\

Using $R_{\kappa \mu \nu \lambda}=g\left(e_\kappa,\left(\left[\nabla_{e_\nu}, \nabla_{e_\lambda}\right]-\nabla_{\left[e_\nu, e_\lambda\right]}\right) e_\mu\right)$, we get that the only non-zero contributions to the indicial operator are $R_{0m0l}\equiv -g_{ml}$ and $R_{kmnl}\equiv g_{kn}g_{ml}-g_{kl}g_{mn}$ and those obtained from these two using the symmetries of the Riemann tensor. Calculating the trace of the spatial part, $g^{ij}g_{ij} = n$ with $n$ the number of spatial dimensions, this then yields for the Ricci tensor $\operatorname{Ric}(g)_{00} \equiv -n$, $\operatorname{Ric}(g)_{0i} \equiv 0$ and $\operatorname{Ric}(g)_{ij}\equiv n g_{ij}$, therefore $\operatorname{Ric}(g) \equiv ng$. In the splitting \eqref{eq:splitting1}, this gives for the operator $\mathscr{R}_g$
\begin{equation}
\label{eq:indR}
\mathscr{R}_g \equiv\left(\begin{array}{ccc}
n I & 0 & \operatorname{tr}_{g_{(0)}} \\
0 & (n+1) I & 0 \\
g_{(0)} & 0 & (n+1) I-g_{(0)} \operatorname{tr}_{g_{(0)}}
\end{array}\right).
\end{equation}
\\

Now we calculate the indicial operator of $\square_g$. As $g^{00} \equiv -1$ and $g^{0i}$ modulo terms which do not contribute to the indicial operator, we get $(\square_g u)_{\mu\nu} \equiv u_{\mu\nu;00}- g^{kl}u_{\mu\nu;kl}$, where $u$ is a section of $S^2\ ^0T^*M$. Now, using $u_{\mu\nu;\kappa}=e_\kappa u_{\mu\nu} - \Gamma^\rho_{\kappa\mu}u_{\rho\nu}-\Gamma^\rho_{\kappa\nu}u_{\mu\rho}$, and $u_{\mu\nu;\kappa\lambda}= e_\lambda u_{\mu\nu;\kappa} - \Gamma^\rho_{\lambda\kappa}u_{\mu\nu;\rho}- \Gamma^\rho_{\lambda\mu}u_{\rho\nu;\kappa}- \Gamma^\rho_{\lambda\nu}u_{\mu\rho;\kappa}$, we get first
$$
\begin{aligned}
u_{00 ; 0} & \equiv e_0 u_{00}, & & u_{00 ; k} \equiv 2 u_{0 k}, \\
u_{0 j ; 0} & \equiv e_0 u_{0 j}, & & u_{0 j ; k} \equiv u_{j k}+g_{j k} u_{00} \\
u_{i j ; 0} & \equiv e_0 u_{i j}, & & u_{i j ; k} \equiv u_{0 j} g_{i k}+u_{i 0} g_{j k},
\end{aligned}
$$
and then
$$
\begin{aligned}
u_{00 ; 00} & \equiv e_0 e_0 u_{00}, & & u_{00 ; k \ell} \equiv 2 u_{0 \ell ; k}+g_{k \ell} u_{00 ; 0}\\
u_{0 j ; 00} &\equiv e_0 e_0 u_{0 j}, & & u_{0 j ; k \ell} \equiv u_{\ell j ; k}+u_{00 ; k} g_{j \ell}+u_{0 j ; 0} g_{k \ell} \\
u_{i j ; 00} &\equiv e_0 e_0 u_{i j}, & & u_{i j ; k \ell} \equiv g_{i \ell} u_{0 j ; k}+g_{j \ell} u_{i 0 ; k}+g_{k \ell} u_{i j ; 0}
\end{aligned}
$$
This then gives for the indicial operator
\begin{equation}
\label{eq:indwave}
\square_g \equiv e_0 e_0-n e_0+\left(\begin{array}{ccc}
-2n & 0 & -2 \operatorname{tr}_{g_{(0)}} \\
0 & -(n+3) & 0 \\
-2 g_{(0)} & 0 & -2
\end{array}\right).
\end{equation}
Now, using the calculations above, we can also calculate the indicial operators for $\delta_g$ and the gauge modifications $E_{g_0}$ and $\tilde{E}$.
\begin{equation}
\label{eq:inddelta}
\delta_g \equiv\left(\begin{array}{ccc}
e_0-n & 0 & -\operatorname{tr}_{g_{(0)}} \\
0 & e_0-(n+1) & 0
\end{array}\right) ,
\end{equation}
\begin{equation}
E_{g_0} \equiv\left(\begin{array}{ccc}
1 & 0 & -2 \operatorname{tr}_{g_{(0)}} \\
0 & 0 & 0
\end{array}\right), \quad \tilde{E} \equiv\left(\begin{array}{cc}
-2 & 0 \\
0 & -2 \\
0 & 0
\end{array}\right).
\end{equation}
For the indicial operator of $\delta^*_g$ we use for the covariant derivative of a one form $\omega_{\mu;\nu} = e_\nu \omega_\mu - \Gamma^\rho_{\nu\mu}\omega_\rho$. This gives $\omega_{0;0} \equiv e_0 \omega_0$, $\omega_{0,j} \equiv \omega_j$, $\omega_{i;0}\equiv e_0\omega_i$, and $\omega_{i;j} = g_{ij}\omega_0$. This then gives
\begin{equation}
\label{eq:inddeltastar}
\delta_g^* \equiv\left(\begin{array}{cc}
e_0 & 0 \\
0 & \frac{1}{2}\left(e_0+1\right) \\
g_{(0)} & 0
\end{array}\right), \quad \tilde{\delta}_g^* \equiv\left(\begin{array}{cc}
e_0-2 & 0 \\
0 & \frac{1}{2}\left(e_0-3\right) \\
g_{(0)} & 0
\end{array}\right).
\end{equation}
Now, with the indicial operator of $L_{h_0,\tilde{h}}$ being equal to that of $\square_g-2 \Lambda+2 \tilde{E} \delta_g G_g+2 \mathscr{R}_g-2 \tilde{\delta}_g^* \circ E$, and remembering that $\Lambda = n$, we get the desired expression.
\end{proof}

\begin{lemma} [Mapping properties of the gauge-fixed Einstein operator.] Let $h_0$, $\tilde{h}$ and $\theta$ be as in Theorem \ref{thm:nonlinsol}. Then $P(h_0,\tilde{h},\theta)$ lies in $s^\beta H^\infty_b(\Omega;S^2\ ^0T^*M)$ with the tame estimate 
$$\|P(h_0,\tilde{h},\theta)\|_{s^\beta H^k_b(\Omega)}\leq C_k\left(\|h_0\|_{H^{k+2}(\mathcal{I}^+)}+\|\tilde{h}\|_{s^\beta H^{k+2}_b(\Omega)}+\|\theta\|_{s^\beta H^{k+1}_b(\Omega)} \right)$$
for all $k\in \mathbb{N}_0$.
\label{le:Pesti}
\end{lemma}
\begin{proof}
    The calculations in the proof of the above Proposition imply that $\operatorname{Ric}(g)-\Lambda g$ vanishes modulo terms that decay as $s^\beta$. The gauge part of $P$ is in coordinates given by terms that do not change the order of decay acting on the difference $\Gamma(g)-\Gamma(g_0)$, $\tilde{h}$ and $\theta$, all of which decay as $s^\beta$. The tame estimates then follow easily from Lemma \ref{le:2.9}.
\end{proof}
Now we determine the indicial roots of this operator in the following Lemma.
\begin{lemma}
\label{le:indicialroots}
$I_{g_{(0)}}(\lambda):=I(L_{h_0,\tilde{h}},\lambda)$ has roots at $\lambda =0,2,3,n,n+1$. The space of indicial solutions corresponding to $\lambda=0$ is given by the bundle $s^{-2}\operatorname{ker}_{g_{(0)}}$ as a subbundle of $S^2\ ^0T^*M$ over $\mathcal{I}^+$, and $I_{g_{(0)}}(\lambda)^{-1}$ has a simple pole at $\lambda=0$.
\end{lemma}
\begin{proof}
We write the indicial family $I_{g_{(0)}}(\lambda)$ from Proposition \ref{prop:indicialcalc} using the refined splitting \eqref{eq:splitting3} and get
\begin{equation}
\begin{aligned}
I_{g_{(0)}}&(\lambda)=\\
&\lambda^2-n \lambda+\left(\begin{array}{cccc}
-4 \lambda+2(n+2) & 0 & 2n(\lambda-2) & 0 \\
0 & -4 \lambda+3(n+1) & 0 & 0 \\
-2 & 0 & 2n & 0 \\
0 & 0 & 0 & 0
\end{array}\right).
\end{aligned}
\end{equation}
We see that the determinant of $I_{g_{(0)}}(\lambda)$ factors as $\lambda(\lambda-2)^2(\lambda-3)(\lambda-n)^3(\lambda-(n+1))$. It has a simple zero at $\lambda=0$ with corresponding kernel $s^{-2}\operatorname{ker}_{g_{(0)}}$.
\end{proof}
For later use, we compute the indicial family of $(-\delta_gG_g+E_{g_0})\delta^*_g$ in the splitting $\eqref{eq:splitting2}$
\begin{equation}
\label{eq:gaugeind}
    I((-\delta_gG_g+E_{g_0})\delta^*_g,\lambda)=-\frac{1}{2}\left(\begin{array}{cc}
        (\lambda-2)(\lambda-n) &0  \\
        0 & (\lambda+1)(\lambda-(n+1)) 
    \end{array}\right),
\end{equation}
using the expressions in the proof of propostion \ref{prop:indicialcalc}, and note that it is invertible for $\lambda\notin\{-1,2,n,n+1\}$.

\subsection{Energy estimates}
\label{se:energy}
The goal of this section is to prove tame estimates for the solution of the linearized gauge-fixed Einstein Equation 
\begin{equation}
L_{h_0, \tilde{h}} v=f,\left.\quad\left(v, \mathcal{L}_{-s \partial_s} v\right)\right|_{\Sigma}=\left(v_0, v_1\right).
\end{equation}
It closely follows the approach of \cite{hintz2024stability}. We will show in detail the proof of the first-order estimate and how to gain an estimate of one order higher, then leave the further arguments for tame estimates of higher orders to the interested reader, as they proceed completely analogously to \cite{hintz2024stability}.
\\

In order to declutter the notation, we introduce two norms. Firstly, for control on the initial data and homogeneity:
$$
\|(f,v_0,v_1)\|_{k,\alpha}:=\|f\|_{s^\alpha H^k_b}+\|v_0\|_{H^{k+1}}+\|v_1\|_{H^k}
$$
We wish to bound the solution $v$ using this norm. The norm on the solution in question will control one 0-derivative, and $k$ b-derivatives. We write this as
$$
\|v\|_{s^{\alpha}H_{0;b}^{1;k}}:=\|v\|_{s^\alpha H^k_b}+\|s\partial_s  v\|_{s^\alpha H^k_b} + \sum_{j=1}^n\|s\partial_{x^j}v\|_{s^\alpha H^k_b}.
$$
(Here, we use the $x^j$ on the respective coordinate systems when calculating the norm.)
\\

We will in this section always assume
\begin{equation}
\label{eq:smallness}
\begin{aligned}
    \|h_0\|_{H^{2d_n+4}(\mathcal{I}^+;s^{-2}S^2\ T^*\mathbb{S}^n)}\leq \delta_0\\
    \|\tilde{h}\|_{s^\beta H_b^{2d_{n+1}+4}(\Omega;S^2\ ^0T^*M)}\leq \delta_0.
\end{aligned}
\end{equation}
By Sobolev embedding \eqref{eq:sobolev embedding}, we therefore have $\mathcal{C}^0$ bounds on the coefficients of $h_0$ and $\tilde{h}$ in the frames $e^\mu\otimes_se^\nu$ in the respective charts.

\subsubsection{Tame bounds on growing spaces}
First, we prove tame estimates on spaces that allow growth at the conformal boundary.
\begin{proposition}
\label{prop:boundsgrow}
    There exists an $N>0$ such that the following holds. Let
    $$
v_0, v_1 \in  H_{b}^{\infty}\left(\Sigma ; S^2\ ^0T^* M\right)), \quad f \in  s^{-N} H_{b}^{\infty}\left(\Omega ; S^{2}\ ^0T^* M\right).
$$
Then the unique solution $v\in s^{-N}H^\infty_b(\Omega;S\ ^0T^*M)$ of
\begin{equation}
    L_{h_0,\tilde{h}}v=f,\quad \restr{(v,\mathcal{L}_{-s\partial_s})}{\Sigma}=(v_0,v_1)
\end{equation}
satisfies the tame estimate
\begin{equation}
\begin{aligned}
\|v\|_{ s^{-N} H_{b}^k} \leq &C_k(\| \left(f, v_0, v_1\right) \|_{{k,-N}} \\
& \left.+\left(\left\|h_0\right\|_{ H^{k+2 d_n+2}}+\|\tilde{h}\|_{ s^\beta H_{b}^{k+2 d_{n+1}+2}}\right)\left\|\left(f, v_0, v_1\right)\right\|_{{0,-N}}\right).
\end{aligned}
\end{equation}
\end{proposition}
\begin{proof}
The proof is based on a system of second-order equations with a particular lower triangular structure. The control over the solutions to this system of wave equations are obtained using a simple first-order estimate of a scalar wave equation, which we study first.

\subsubsection{Domain for the estimates}
\label{se:domainforesti}
We prove an energy estimate by applying Stokes' Theorem on a domain that lies completely in one chart of $\Omega$. A global estimate can then be obtained by adding the two estimates. For the estimate, we need a domain with spacelike boundaries for $g=g_{\mathrm{dS}} + \chi h_0 + \tilde{h}$, where the boundary includes the upper hemisphere for all $s\in[s_1,s_{0}]$ for an arbitrary $s_1\in [0,s_0)$. Furthermore, the outward pointing normal must be future directed, except for the part of the boundary with $s=s_0$. For concreteness sake, we define our domain
$$
U_{s_1}:=\{(s,\omega)\in{}^\uparrow\Omega:|{}^\uparrow p(\omega)|\leq \frac{11}{10}+\sqrt{s},s\in[s_1,s_0]\}.
$$
The boundary $\partial U$ is given by
$$
\begin{aligned}
&\{(s_0,\omega):|{}^\uparrow p(\omega)|\leq 1.1+\sqrt{s_0} \}\\
&\cup \{(s_1,\omega):|{}^\uparrow p(\omega)|\leq 1.1+\sqrt{s_1} \}\\
&\cup \{(s,\omega):|{}^\uparrow p(\omega)|=\, 1.1+\sqrt{s},\ s\in[s_1,s_0]\}\\
&:=\Sigma_{s_0}\cup\Sigma_{s_1}\cup \Sigma_a.
\end{aligned}
$$
A schematic of the domain and its boundary is shown in Figure \ref{fig:Domain}.\\

We check that the boundary is spacelike for $g_{\mathrm{dS}}$. The metric in our coordinates $^\uparrow P(s,\omega)$ on $^\uparrow \Omega$ takes the form 
$$
\frac{1}{\sin^2(s)}\mathrm{diag}\{-1,\frac{4}{(1+|{}^\uparrow p(\omega)|^2)^2}\operatorname{id}_{n\times n}\}.
$$
Here, $|{}^\uparrow p(\omega)|$ is the standard Euclidian norm on $\mathbb{R}^n$. Immediately, this shows that $\Sigma_{s_0}$ and $\Sigma_{s_1}$ are spacelike. For a curve $\gamma(t)=(s(t),\omega(t))\in \Sigma_a$ (with $\dot{s}=0$, else $\gamma $ is constant), we get that $\partial_t|{}^\uparrow p(\omega(t))|=\frac{\dot{s}}{2\sqrt{s}}$ and therefore 
$$
\frac{\sin^2(s)}{\dot{s}^2}|\dot{\gamma}|_{g_{\mathrm{dS}}} = -1 + \frac{1}{s\left(1+(1.1+\sqrt{s})^2\right)^2},
$$
which is a positive function for $0\leq s\leq s_0=0.1$. Therefore, $\partial U$ is spacelike for $g_{g_{\mathrm{dS}}}$ and, by extension, also for $g$, provided $\delta_0$ is small enough. Also, the outward pointing normal for $\Sigma_{s_1}$ and $\Sigma_a$ have a component $-c\partial_s$ with $c>0$, so are future oriented in our convention.\\
\begin{figure}[h!]
    \centering
    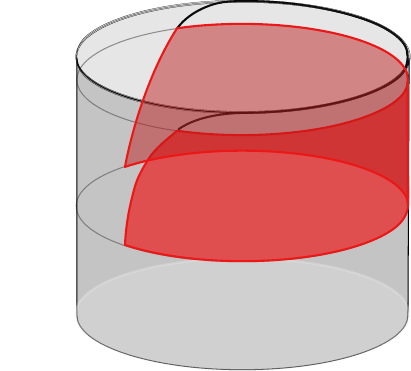
    \caption{Schematic of the Domain. The red shaded area corresponds to $U_{s_1}$. The top and bottom part of its border correspond to $\Sigma_{s_1}$ and $\Sigma_{s_0}$ respectively. An the curved parts correspond $\Sigma_a$}
    \label{fig:Domain}
\end{figure}

In order to apply this result to de Sitter type metrics (obtained from scattering data) far from exact de Sitter space, as mentioned in \ref{re:generalX}, one would of course need to adapt this concrete choice domain. It is clear, however, that such adaptions can be done, so no further details will be given here.
\subsubsection{Estimate for the scalar wave equation}
We consider the scalar wave equation $\Box_g v=f$ on $\Omega$ with initial data $(v_0,v_1)$. Define the energy-momentum tensor $T=T[v]$ as 
$$T(V,W) := V(v)W(v)-\frac{1}{2}g(V,W)g(\nabla v,\nabla v)$$
for two vector fields $V,W$ and for a given vector field $V$ the associated $J$-current 
$$^{(V)} J=T(V, \cdot).$$
Calculating the divergence of this one-form gives
\begin{equation}
    \operatorname{div}_g\left(^{(V)} J\right)=-\left(\square_g v\right) V v+{ }^{(V)} K, \quad^{(V)} K=\frac{1}{2}T \cdot \mathcal{L}_V g,
\end{equation}
where $\cdot$ denotes tensor contraction (using $g$).\footnote{Recall the convention $(\Box_gv)_{\mu\nu}=-g^{\kappa\lambda}v_{\mu\nu;\kappa\lambda}$}
\\

Now we get the following equation by applying Stokes' Theorem on $U_{s_1}$.
$$\begin{aligned} \int_{\Sigma_{s_1}\cup\Sigma_a}T(V,\nu) \mathrm{d} \sigma & +\int_{U_{s_1}}{ }^{(V)} K \mathrm{~d} g \\ & =\int_{\Sigma_{s_0}}T(V,\nu) \mathrm{d} \sigma+\int_{U_{s_1}}\left(\square_g v\right) V v \mathrm{~d} g\end{aligned}$$
Here, $\nu$ is the future pointing unit normal at the respective boundary. (Recall that since $-\partial_s$ is future oriented for $g_{\mathrm{dS}}$ and $h_0$ and $\tilde{h}$ are small in $\mathcal{C}^0$ as sections of $S^2\ ^0T^*M$ we have that the future unit normal is $-s\partial_s$ modulo small correction terms).\\

Now, by choosing a clever vector field $V$, we can exploit the positive definiteness of the energy momentum tensor to control a weighted norm of the solution. As long as we choose something that is future-time-like, the first term on the left-hand side is non-negative and will be dropped. Furthermore, the integral over $\Sigma_{s_0}$ is bounded by a constant times $\|v_0\|_{H^1}+\|v_1\|_{L^2}$, as long as the vector field is bounded on $\Sigma$. So, a first choice could be $V = -s\partial_s$. After taking the limit of $s_1$ going to 0, we would get
$$
\int_{U_0} { }^{(V)}K \mathrm{d}g \leq C (\|v_0\|_{H^1}^2+\|v_1\|_{L^2}^2) +\int_{U_0} (\square_g v) Vv\mathrm{d}g
$$
But in order to be able to control the solution with the left-hand side of the equation, we will need a modification. So we define
\begin{equation}
    V_0:=-s\partial_s, \quad w=w(s):=s^{N'}e^{Fs^{2\beta}/2\beta},\quad V:=w^2V_0,
\end{equation}
and see that $V$ is still future time like. The constants $N'$ and $F$ will be fixed later. In fact, to get a first-order energy estimate, a weight of $s^{N'}$ would suffice, but the exponential term allows us to absorb terms arising from higher-order estimates. The exponential term is uniformly bounded away from 0 and infinity, so we can easily estimate it. \\

Now we calculate
\begin{align}
    ^{(V)}K={}^{(w^2V_0)}K= w^2 {}^{(V_0)}K + 2wT(\nabla w,V_0)\\
    T(\nabla w,V_0)=sw'(s)T\left(\left(\frac{\mathrm{d}s}{s}\right)^\sharp,-s\partial_s\right).
\end{align}
Both $\left(\frac{\mathrm{d}s}{s}\right)^\sharp$ and $-s\partial_s$ are uniformly time-like on $\Omega$ and therefore $T\left(\left(\frac{\mathrm{d}s}{s}\right)^\sharp,-s\partial_s\right)$ controls a norm of derivatives along 0-vector fields of $v$. To explicitly see this, we work in the frame $e_\mu$ and compute $T_{\mu\nu}= (e_{\mu}v) (e_{\nu}v)-\frac{1}{2}g_{\mu\nu}g^{\kappa\lambda}(e_{\kappa}v) (e_{\lambda}v)$. By writing $\left(\frac{\mathrm{d}s}{s}\right)^\sharp$ and $-s\partial_s$ in this frame, we see that there exists a constant $c_0$ such that $T\left(\left(\frac{\mathrm{d}s}{s}\right)^\sharp,-s\partial_s\right)\geq c_0\sum_{\mu=0}^n \left|e_\mu v\right|^2$. In addition, there exists a constant $C_0$ such that $\left|^{(V_0)}K\right|\leq C_0\sum_{\mu=0}^n \left|e_\mu v\right|^2$ on $^\uparrow \Omega$. Together, we get the estimate
\begin{equation}
    \begin{aligned}
    \int_{U_0} (2sw'c_0-wC_0&)w \sum_{\mu=0}^n \left|e_\mu v\right|^2 \mathrm{d}g\\
    &\leq  C (\|v_0\|_{H^1}^2+\|v_1\|_{L^2}^2) +\int_{U_0} (\square_g v) Vv\mathrm{d}g.
    \end{aligned}
\end{equation}
Now we calculate
\begin{equation}
    sw'=(N'+Fs^{2\beta})w
\end{equation}
This shows that for $N'$ large enough, the left-hand side of the estimate controls a weighted norm of the 0-derivatives of $v$. \\

To also include the (weighted) $L^2$-norm of $v$ we again use Stokes' Theorem, this time for the vector field $v^2W$, where we define $W:=w^2W_0$ and $W_0=s\partial_s$. Again using the domain $U_{s_1}$ from above, we get
$$
\int_{\Sigma_{s_1}\cup\Sigma_a}\left\langle v^2 W, \nu\right\rangle \mathrm{d} \sigma+\int_{U_{s_1}} \operatorname{div}_g\left(v^2 W\right) \mathrm{d} g=\int_{\Sigma_{s_0}}\left\langle v^2 W, \nu\right\rangle \mathrm{d} \sigma.
$$
We again drop the integral over $\Sigma_{s_1}\cup\Sigma_a$ as it is nonnegative ($s\partial_s$ is past- and $\nu$ is future time-like). Furthermore, $\langle W,\nu \rangle$ is bounded from above and below by a positive multiple of $s^{2N'}$. Therefore, the right-hand side is bounded by a constant times $\|v_0\|_{L^2}$. Next, we calculate the divergence
\begin{equation}
    \begin{aligned}
        \operatorname{div}_g(v^2W)&= 2vWv + v^2\operatorname{div}_g(W) \\
        &=2vWv + v^2(2wW_0w + w^2 \operatorname{div}_g(W_0))\\
        &=2vWv + v^2w^2(2\left(N'+F s^{2 \beta}\right)+\operatorname{div}_g(W_0)).
    \end{aligned}
\end{equation}
The divergence $\operatorname{div}_g(W_0)$ equals $\Gamma(g)^\mu_{\mu0}$. With the same calculation as in the proof of proposition \ref{prop:indicialcalc}, we see that this gives $-n$ plus two correction terms, one from $\overline{g}$ and one from $h_0$ and $\tilde{h}$, the latter being small in $L^\infty$ by Sobolev embedding and the smallness assumption \eqref{eq:smallness}. Therefore we can estimate $\operatorname{div}_g(W_0)\leq C_1$ (independent of $h_0,\tilde{h}$). The first term we estimate as $|2vWv|\leq w^2v^2+w^2(e_0v)^2$. This gives for the integrand
\begin{equation}
\begin{aligned}
    \operatorname{div}_g(v^2W)&\geq v^2w^2(2\left(N'+F s^{2 \beta}\right)) -w^2v^2(1+C_1) -w^2(e_0v)^2\\
    &\geq v^2w^2(N'+Fs^{s\beta})-w^2(e_0v)^2,
\end{aligned}
\end{equation}
for $N'$ large enough. After letting again $s_1$ go to 0 we can add both integral inequalities, and by increasing $C_0$ by 1, we get the estimate
\begin{equation}
\label{eq:integralestimate}
    \begin{aligned}
        \int_{U_0}( v^2w^2(N'+Fs^{2\beta}) &+( 2c_0 s w'-C_0 w) w \sum_{\mu=0}^n|e_\mu v|^2 ) \mathrm{d}g\\
        &\leq C\left(\left\|v_0\right\|_{H^1}^2+\left\|v_1\right\|_{ L^2}^2\right)+\int_{U_0} w^2\left|\square_g v \| s \partial_s v\right| \mathrm{d} g .
    \end{aligned}
\end{equation}
This is the main estimate with which we will continue to work. We will fix the constants $N'$ and $F$ later. This is because the constants $C_0$ and $C$ will increase to absorb error terms.
\\

We will not need it here, but by adding the corresponding estimate on the lower hemisphere using $^\downarrow\Omega$, we can extract an energy estimate for the scalar wave equation from this, using the same arguments as below.
\subsubsection{First order estimate for the linearized Einstein equation}
\label{se:firstorder}
Now, let $v$ be the unique solution to the full tensor equation 
$$ 
L_{h_0, \tilde{h}} v=f,\left.\quad\left(v, \mathcal{L}_{-s \partial_s} v\right)\right|_{\Sigma}=\left(v_0, v_1\right).
$$
On $^\uparrow\Omega$, we can write $L_{h_0,\tilde{h}}$ as 
$$
\begin{aligned}
& L_{h_0, \tilde{h}}=\square_g \otimes \operatorname{Id}_{\mathbb{R}^{(n+1)(n+2)/2}}+Q, \quad Q=Q_0+Q_{(0)}+\tilde{Q}, \\
& \quad Q_0 \in \operatorname{Diff}_{0}^1, \quad Q_{(0)} \in \mathcal{C}^0\left(\mathcal{I}^{+}\right) \operatorname{Diff}_{0}^1, \quad \tilde{Q} \in s^\beta \mathcal{C}^0\left(\Omega\right) \operatorname{Diff}_{0}^1
\end{aligned}
$$
acting component wise on the $(n+1)(n+2)/2$ component vector $v$ ($(n+1)(n+2)/2$ is the rank of $S^2\ ^0T^*M$). $Q_0$ is independent of the metric perturbations $h_0 $ and $\tilde{h}$, while $Q_{(0)}$ and $\tilde{Q}$ have small norms in their respective spaces due to the smallness assumptions \eqref{eq:smallness}. Now we apply estimate \eqref{eq:integralestimate} to each component $v^{(i)}$ separately. The change from $\square_g v^{(i)}$ to $(L_{h_0, \tilde{h}}v)^{(i)}$ gives error terms of the form $\int_{\Omega}|(Qv)^{(i)}||s\partial_sv^{(i)}|\mathrm{d}g$. As $Q$ is a 0-differential operator of order 1 we can estimate the terms arising from it by $C_{Q}\int_{U_0} w^2\sum_{(j)}\left( (v^{(j)})^2+\sum_{\mu=0}^n(e_\mu v^{(j)})^2\right)$. The constant $C_Q$ can be taken to be independent of $h_0$ and $\tilde{h}$ because of the smallness assumption \eqref{eq:smallness}. Next, we sum the estimates obtained thus far, for all components $v^{(i)}$. We put the term arising from $Q$ on the left-hand side and absorb the derivative part of it again into a larger $C_0$ and the other part by exchanging $N'+Fs^{2\beta}$ to $N'/2+Fs^{2\beta}$ if $N'$ is sufficiently large. We use $|ab|\leq \frac{|a^2}{2}+ \frac{|b|^2}{2}$ on the terms $\int_{U_0} |(L_{h_0,\tilde{h}}v)^{(i)}||s\partial_sv^{(i)}|\mathrm{d}g$ and absorb the $e_0$ part into yet a larger $C_0$. 
\\

Now finally, we can fix $N'$ such that $2N'c_0-C_0 \geq 1$ and $N'/2 \geq 1$ to get the following estimate
\begin{equation}
\begin{aligned}
& \int_{U_0}  w^2\left(1+F s^{2 \beta}\right)\left|\partial^{\leq 1} v\right|^2 \mathrm{~d} g \\
& \quad \leq C\int_{U_0} w^2 |L_{h_0, \tilde{h}} v|^2\mathrm{d}g+C_{N', F}\left(\left\|v_0\right\|_{H^1}^2+\left\|v_1\right\|_{ L^2}^2\right),
\end{aligned}
\label{eq:integralesti2}
\end{equation}
where we write $|\partial_{\leq1}v|^2:=\sum_{(i)}\left( (v^{(i)})^2+\sum_{\mu=0}^n(e_\mu v^{(i)})^2\right)$. 
\\

With the same reasoning, we can get the identical estimate on $^\downarrow \Omega$. Summing the two, the left-hand side controls the $s^{-N'-1/2} H^1_0(\Omega;S^2\ ^0T^*M)$ norm of $v$ as the respective $U_0$ cover their respective hemispheres. We get the first order estimate
$$
    \|v\|_{s^{-N'+\frac{1-n}{4}}H^1_0(\Omega)}\leq C\|(f,v_0,v_1)\|_{(0,-N'+\frac{1-n}{4})}.
$$
The $\frac{1-n}{4}$ term arises because our choice of density for the Sobolev spaces on $\Omega$ differs from the metric density $\mathrm{d}g$ by a factor of $s^{\frac{n-1}{2}}$. After defining $N:=N'+\frac{1-n}{4}$ to improve the cosmetics of the equation this gives 
\begin{equation}
\label{eq:firstorderesti}
    \|v\|_{s^{-N}H^{1;0}_{0;b}(\Omega)}\leq C\|(f,v_0,v_1)\|_{0,-N}.
\end{equation}
which implies the estimate in Proposition \ref{prop:boundsgrow} for $k=0$.\\

In order to control up to $k$ b-derivatives of the solution $v$ we look at the wave equation for $Vv$, where $V$ stands for an up to $k$ fold composition of b-vector fields. For it, we have equation $L_{h_0,\tilde{h}}Vv = Vf+[L_{h_0,\tilde{h}},V]v$. If we write a vector with a basis for all such $V$, we get a system of equations where the error terms arising from the commutator have a particular lower triangular structure discussed below and because of that can be absorbed into the left-hand side of this estimate provided $F$ is fixed accordingly. So first, we prove an energy estimate for a system of wave equations with the sort of error terms that will arise later.
\subsubsection{System of wave equations}
Let $K\in \mathbb{N}$ and consider the first order differential operator $A$, where its components are operators $A_{I J}, 1 \leq I, J \leq K$ acting on sections of $S^2\ ^0T^*M$ of the form $A_{I J}=A_{0, I J}+A_{(0), I J}+\tilde{A}_{I J}$, where
$$
A_{0, I J} \in \operatorname{Diff}_{0}^1, \quad A_{(0), I J} \in  \mathcal{C}_b^0\left(\mathcal{I}^{+}\right) \operatorname{Diff}_{0}^1, \quad \tilde{A}_{I J} \in s^\beta \mathcal{C}_{b}^0\left( \Omega \right) \operatorname{Diff}_{0}^1
$$
Furthermore, these operators have the following lower triangular structure
\begin{equation}
\label{eq:lowtriang}
I \leq J \Longrightarrow A_{0, I J} \in s \operatorname{Diff}_{0}^1, \quad A_{(0), I J}=0.
\end{equation}
Consider the wave equation $\mathcal{L}v=f$, with $\mathcal{L}= L_{h_0,\tilde{h}}\otimes \operatorname{id}_{\mathbb{R}^K} +A$ and let $v$ be the unique solution with initial data $v_0,v_1$ at $s=s_0$. Investigating this equation on $U_0$ like above, we can apply our estimate \eqref{eq:integralesti2} to each $(n+1)(n+2)/2$ dimensional component $v^{I}$ of $v$. On the right-hand side, we get inside the integral $L_{h_0, \tilde{h}} v^I=f^I-\sum_{J=1}^K A_{I J} v^J$. We divide the sum $\sum_{J=1}^KA_{IJ}v^J$ into the parts $J<I$ and $J\geq I$. For the terms with $J\geq I$ we can estimate $\sum_{J\geq I}|A_{IJ}v^{J}|\leq \sum_{J\geq I}C_A s^\beta \left|\partial^{\leq 1} v^J\right|$ because of \eqref{eq:lowtriang}. For the terms $J<I$ we have no additional decay and we estimate them by $\sum_{J< I}|A_{IJ}v^{J}|\leq \sum_{J< I}C_A \left|\partial^{\leq 1} v^J\right|$. After some basic estimates to separate the $f^{I}$ from the terms containing $A_{IJ}$, we get
$$
\begin{aligned}
 \int_{U_0}  &w^2\left(1+F s^{2 \beta}\right)\left|\partial^{\leq 1} v^I\right|^2 \mathrm{~d} g \\
& \leq C_K\left(\left\| f^I\right\|_{s^{-N}L^2}^2\ +\|v_0^{I}\|^2_{H^1}+\|v_1^{I}\|^2_{L^2}+C_{A} \sum_{J<I} \int_{U_0}  w^2\left|\partial^{\leq 1} v^J\right|^2 \mathrm{~d} g \right. \\
 & \left. +C_A \sum_{J \geq I} \int_{U_0}  s^{2 \beta} w^2\left|\partial^{\leq 1} v^J\right|^2 \mathrm{~d} g \right)
\end{aligned}
$$
Call this estimate $(*)^{I}$. We then take the weighted sum of these estimates $\sum_{I=1}^K\epsilon^{I}(*)^{I}$ for a positive $\epsilon$ to be fixed later. The last term on the right-hand side is unproblematic. Because of the additional $s^{2\beta}$ term they can be absorbed into the left-hand side of the estimate by choosing $F$ sufficiently large. For the second term, we estimate for $\epsilon < 1$
$$
C_KC_A\sum_{I=2}^K \sum_{J=1}^{I-1} \epsilon^{I} \left|\partial^{\leq 1} v^J\right|^2 \leq C_KC_AK\epsilon\sum_{I=J}^{K}\epsilon^{J}\left|\partial^{\leq 1} v^J\right|^2.
$$
Now we fix $\epsilon$ such that $C_KC_AK\epsilon<1/2$ and absorb it into the left-hand side. \\

Summing with the corresponding estimate for $^\downarrow\Omega$ and taking the square root, the left-hand side now controls the $s^{-N}H^1_0(\Omega)$ norm of $v$ and therefore, we have proven the estimate,
\begin{equation}
\label{eq:systemwaveesti}
    \|v\|_{s^{-N}H^{1;0}_{0;b}(\Omega)}\leq C\|(f,v_0,v_1)\|_{0,-N}
\end{equation}
for this system of equations. (Importantly, we have the same $N$ as in the case studied in section \ref{se:firstorder}).
\subsubsection{Tame estimates, arbitrary order}
In order to use a Nash-Moser iteration later we need tame estimates of arbitrary order on the solution of the linearized Einstein equation. We claim the estimate 
\begin{equation}
\label{eq:weirdweakesti}
\begin{aligned}
& \|v\|_{ s^{-N} H_{0 ; b}^{1 ; k}} \\
& \quad \leq C_k\left(\left\|\left(f, v_0, v_1\right)\right\|_{{k,-N}}\right. \\
& \left.\quad+\left(\left\|h_0\right\|_{ H^{k+2+2 d_n}}+\|\tilde{h}\|_{ s^\beta H_{b}^{k+2+2 d_{n+1}}}\right)\|v\|_{ s^{-N} H_{0}^1}+\|v\|_{ s^{-N} H_{0 ; b}^{1 ; k-1}}\right) .
\end{aligned}
\end{equation}
Crucially, we prove this for all $k$ for a \emph{fixed} value of $N$ (independent of $k$). Applying this iteratively, and in the end using the first order estimate \eqref{eq:firstorderesti}, we get the desired tame estimates of the Proposition \ref{prop:boundsgrow}.\\

We prove only the case $k=1$. The structural reason why it works is Lemma \ref{le:ideal}. The general case then proceeds completely analogously and works for the same reason but introduces considerable notational complexity. \\

Fix $\mathscr{V}=\{V_1,...,V_{n+1}\}$ the rotation vector fields around the coordinate axes in $\mathbb{S}^n$. We now wish to show that the vector $v':= (s\partial_s v,V_1v,...,V_{n+1}v)$ fulfills a wave equation of the type discussed above. Write, for example, for the first component 
$$
L_{h_0,\tilde{h}}(s\partial_s v)=s\partial_sL_{h_0,\tilde{h}}v + [L_{h_0,\tilde{h}},s\partial_s]v
$$
and analogously for the other components. Our aim is to show that we can split the commutator terms into two terms. The first are differential operators of order 1 on $v'$ with the lower triangular structure discussed before. They will play the role of the $A$ in the system of wave equations discussed above. The other, we will be able combine with the first term on the right-hand side into an $f'$, where we will have tame estimates for this $f'$ estimated in $L^2$.
\\

We use the partition of unity to write $L_{h_0,\tilde{h}} = {}^\uparrow L_{h_0,\tilde{h}}+{}^\downarrow L_{h_0,\tilde{h}}:= {}^\uparrow\varphi L_{h_0,\tilde{h}}+{}^\downarrow\varphi L_{h_0,\tilde{h}}$. We split $[L_{h_0,\tilde{h}},s\partial_s]v$ into $[{}^\uparrow L_{h_0,\tilde{h}},s\partial_s]v + [{}^\downarrow L_{h_0,\tilde{h}},s\partial_s]v$ and analyze each term in the respective coordinate system. We show the calculations in $^\uparrow\Omega$ for ${}^\uparrow L_{h_0,\tilde{h}}$, in the other chart it works the same. In a slight abuse of notation in order to increase legibility, we drop the $\uparrow$ from the notation and just write $L_{h_0,\tilde{h}}$ as we will keep working in this chart.
\\

Now, trivializing $L_{h_0,\tilde{h}}$ in the frame $e_\mu$ and writing $(s\mathscr{V})^\gamma=(sV_1)^{\gamma_1}\ldots (sV_{n+1})^{\gamma_{n+1}}$, we can write it as
$$
L_{h_0,\tilde{h}}= \sum_{|\gamma|+i\leq2} \ell_{\gamma i}(s,x)(s\mathscr{V})^\gamma (s\partial_s)^{i},
$$ 
where $\ell_{\gamma i}(s,x)= \ell_{0,\gamma i}(s,x) +\ell_{(0),\gamma i}(x) + \tilde{\ell}_{\gamma i}(s,x)$ are $(n+1)(n+2)/2\times(n+1)(n+2)/2$ matrices of class
\begin{equation}
\ell_{0,\gamma i} \in \mathcal{C}^{\infty}\left(\Omega \right), \quad \ell_{(0),  \gamma i} \in H_b^{\infty}\left(\mathcal{I}^{+}\right), \quad \tilde{\ell}_{ \gamma i} \in  s^\beta H_{{b}}^{\infty}\left(\Omega\right).
\end{equation}
In fact, they obey tame estimates in terms of $h_0$ and $\tilde{h}$ because of proposition \ref{prop:indicialcalc}. In addition, they have compact support in $^\uparrow\Omega$. Now, we can expand the commutators
\begin{align}
    [L_{h_0,\tilde{h}},s\partial_s]v&=-\sum_{|\gamma|+i\leq 2} \left((s\partial_s \ell_{\gamma i})(s\mathscr{V})^\gamma (s\partial_s)^{i}v - |\gamma|\ell_{\gamma i}(s\mathscr{V})^\gamma (s\partial_s)^{i}\right)v \label{eq:commu1},\\
    [L_{h_0,\tilde{h}},V_a]v&=-\sum_{|\gamma|+i\leq 2} \left((V_a \ell_{\gamma i})(s\mathscr{V})^\gamma (s\partial_s)^{i}v - \ell_{\gamma i}[V_a,(s\mathscr{V})^\gamma] (s\partial_s)^{i}\right)v\label{eq:commu2}.
\end{align}
Consider first the terms with $|\gamma| + i\leq1$. We estimate them in $s^{-N}L^2$ by a constant times $\sum_{\gamma,i}\|\ell_{\gamma i}\|_{\mathcal{C}_b^1}\|v\|_{s^{-N}H^{1,0}_{0,b}(\Omega)}$ and then  
$$
\|\ell_{\gamma i}\|_{\mathcal{C}^1_b} \leq C(1+\|h_0\|_{H^{2d_n+3}(\mathcal{I}^+)}+\|\tilde{h}\|_{s^\beta H_b^{2d_{n+1}+3}})
$$
using Sobolev embedding and \eqref{eq:Lweakesti} with $k-2 = d_n+1$ and $k-2= d_{n+1}+1$ respectively\footnote{The partition of unity may increase the constant, but not destroy the estimate}. These terms we put into $f'$.\\

Now, we turn to the terms with $\gamma + i = 2$. Firstly, consider the first term on the right-hand side of the first equation \eqref{eq:commu1}. As $s\partial_s\ell_{\gamma i}\in s\mathcal{C}^\infty + s^\beta H^\infty$, we can write this term as $(s\operatorname{Diff}_0^1+s^\beta \mathcal{C}^0\operatorname{Diff}_0^1)$ acting on one component of $v'$, e.g. on $s\partial_sv$ if $i\geq1$ and to an $V_av$ otherwise. Therefore, we can count it as a contribution to $A_{11}$ or $A_{1a}$ respectively. Next, we consider the second term on the right-hand side of both equations, \eqref{eq:commu1}\eqref{eq:commu2}. For them to be non-zero, we need $|\gamma| \geq 1$. Therefore, we have at least one factor of $sV_b$ (this uses that $[V_a,V_b]=g^c_{ab}V_c$ for smooth functions $g^c_{ab}$). As $V_b$ and $s\partial_s$ commute, we can write this as 
\begin{equation}
\label{eq:umschribe}
    s\cdot q (V_a\ell_{\gamma i})(sV_c)^{\gamma'} (s\partial_s)^{i}\circ V_b v,
\end{equation}
for a smooth function $q$ and a multiindex $\gamma'$ with $|\gamma'|=|\gamma|-1$, and consider this as $s$ times a first order 0-differential operator acting on $V_b v$ and therefore a trivial contribution to the appropriate $A_{IJ}$. Lastly, we consider the first term of the second equation \eqref{eq:commu2}. In the case that $|\gamma|\geq1$ we can again write it like \eqref{eq:umschribe} and absorb it into an appropriate $A_{IJ}$. If $|\gamma| = 0$, then we have $i=2$ and we write this as a first order 0-differential operator acting on $s\partial_s v=v'^1$. Therefore, it is a contribution to $A_{a1}$, so a strictly lower triangular term. It does not decay toward $s=0$ and this is precisely why we allowed for such lower triangular terms in our estimate.\\

So now, to summarize, we have a system of wave equations of the form $\mathcal{L}v'=f'$ with $\mathcal{L}$ with the lower triangular structure described above, where $v'=(s\partial_sv,V_1v,...,V_nv)$ and $f'$ given by $(s\partial_sf,V_1f,...,V_nf)$ plus the terms from the right-hand side with $|\gamma|+i \leq1$ both from ${}^\uparrow L_{h_0,\tilde{h}}$ and ${}^\downarrow L_{h_0,\tilde{h}}$. The $A_{IJ}$ come from the terms with $|\gamma| +i =2$, again from both ${}^\uparrow L_{h_0,\tilde{h}}$ and ${}^\downarrow L_{h_0,\tilde{h}}$. We have control over $f'$ in $s^{-N}L^2$ with 
$$
\begin{aligned}
&\|f'\|_{s^{-N}L^2}\\
&\leq C\left(\|f\|_{s^{-N}L^2} + (1+\|h_0\|_{H^{2d_n+3}(\mathcal{I}^+)}+\|\tilde{h}\|_{H_b^{2d_{n+1}+3}})\|v\|_{s^{-N}H^{1,0}_{0,b}(\Omega)}\right).
\end{aligned}
$$
What is missing is control of the initial data of $v'$ in terms of $v_0$ and $v_1$. As the $V_a$ are tangent to $\Sigma$, the only part we cannot directly write in terms of $v_0$ and $v_1$ is $\restr{(s\partial_s)^2v}{\Sigma}$. Working in each coordinate chart, we notice that $\ell_{02}$, the coefficient of $(s\partial_s)^2$ in $L_{h_0,\tilde{h}}$, is diagonal with entries 
$$
-g^{-1}(\frac{\mathrm{d}s}{s},\frac{\mathrm{d}s}{s})=-g_{\mathrm{dS}}^{-1}(\frac{\mathrm{d}s}{s},\frac{\mathrm{d}s}{s})\bmod H_b^\infty(\mathcal{I}^+) + s^\beta H_b^\infty(\Omega)= -g_{\mathrm{dS}}^{-1}(\frac{\mathrm{d}s}{s},\frac{\mathrm{d}s}{s})
$$
plus small corrections in view of \eqref{eq:smallness}. Therefore, they are bounded away from 0 on the support of the respective $\varphi$, and $\ell_{02}$ can be inverted on this domain with bounded coefficients. To carefully deal with that term, we use the product estimate and the nonlinear expression estimate. We write
\begin{equation}
\left(s \partial_s\right)^2 v=\ell_{02}^{-1}\left(L_{h_0, \tilde{h}} v-L_0v-L_1(s\partial_sv)\right),
\end{equation}
where $L_0$ and $L_1$ are differential operators built purely from derivatives that are tangent to $\{s=\mathrm{const}\}$. Therefore, the restriction of that part to $\Sigma$ can be computed purely in terms of $v_0,v_1$. Also, the coefficients can be estimated in $L^\infty$ and then, using Sobolev embedding, we see that they obey tame estimates. The restriction of $L_{h_0, \tilde{h}} v= f$ to $\Sigma$ if bounded in $L^2(\Sigma)$ by the $s^{-N}H^1_b(\Omega)$ norm of $f$ by Lemma \ref{le:2.9} (2). Now, finally, we use the estimate \eqref{eq:systemwaveesti} to conclude the weak estimate \eqref{eq:weirdweakesti} for $k=1$. 
\\

The general case works for the same reasons. The exact estimates are analogous as in \cite{hintz2024stability}, remembering that we change $d_3$ and $d_4$ to $d_n$ and $d_{n+1}$ respectively, and we always split $L_{h_0, \tilde{h}} ={}^\uparrow L_{h_0, \tilde{h}} + {}^\downarrow L_{h_0, \tilde{h}} $.
\\

This finally proves the Proposition \ref{prop:boundsgrow}.
\end{proof}

\subsubsection{Tame bounds on decaying spaces}
Now we aim to use that the roots of the indicial family $I_{g_{(0)}}(\lambda)$ are nonnegative to iteratively control the solution to the linearized equation in spaces with slower decay. We need the following Lemma.
\begin{lemma}[Inversion of the indicial operator]
    Let $s_1\in(0,s_0)$ and $\eta_1 <\eta_2<1$ with $\eta_1, \eta_2 \neq 0$, and $k \in \mathbb{N}_0$. Recall the operator $I_{g_{(0)}}$ from proposition \ref{prop:indicialcalc}, where $g_{(0)}=\cancel{g}+h_{(0)}$ is defined as in \eqref{eq:gonbound}. Suppose $v \in s^{\eta_1} H_{b}^k\left(\Omega\right)$ vanishes for $s \geq s_1>0$, and $I_{g_{(0)}}\left(s \partial_s\right) v \in$ $ s^{\eta_2} H_{b}^k\left(\Omega\right)$.
\begin{enumerate}
    \item (Improving the weight) If $\eta_1<\eta_2<0$ or $0<\eta_1<\eta_2<1$, then $v \in  s^{\eta_2} H_{b}^k\left(\Omega\right)$ and
\begin{equation}
\|v\|_{ s^{\eta_2} H_{  b}^k} \leq C_k\left(\left\|I_{g_{(0)}} v\right\|_{s^{\eta_2} H_{b}^k}+\left\|h_0\right\|_{ H^{k+ d_{n}}}\left\|I_{g_{(0)}} v\right\|_{s^{\eta_2} L^2}\right)
\label{eq:indinv}
\end{equation}
\item (Extracting asymptotics) If $\eta_1<0<\eta_2$, then there exist tensors $\tilde{v}^{\prime} \in s^{\eta_2} H_b^k\left(\Omega\right)$ and $v_0 \in  H^k\left(\mathcal{I}^{+} ; s^{-2} \operatorname{ker}\operatorname{tr}_{g_{(0)}}\right)$ so that
\begin{equation}
v\left(s, \omega\right)=\chi(s)v_0(\omega)+\tilde{v}^{\prime}(s, \omega)
\label{eq:indasy}
\end{equation}

and $\left\|v_0\right\|_{ H^k}+\left\|\tilde{v}^{\prime}\right\|_{s^{\eta_2} H_{b}^k}$ is bounded by the right hand side of \eqref{eq:indinv}.
\label{item2}
\end{enumerate}
\end{lemma}

\begin{proof}
The proof is identical to \cite{hintz2024stability}. It uses a contour-shifting argument on the Mellin transform side. Because the indicial family $I_{g_{(0)}}(\lambda)$ has a simple pole at $\lambda = 0$, we get the zero-order contribution $v_0$ in part \ref{item2}.
\end{proof}
Now we use this Lemma to iteratively upgrade our estimate to decaying spaces.
\begin{proposition}[Tame bounds on decaying spaces]
\label{prop:boundsdecay}
 There exists $d \in \mathbb{N}$ so that the following holds. Whenever $\left\|h_0\right\|_{   H_{  b}^d},\|\tilde{h}\|_{ s^\beta H_{  b}^d}<\delta_0$, the unique solution $v$ of the initial value problem
\begin{align}
L_{h_0, \tilde{h}} v & =f \quad \in  s^\beta H_{  b}^{\infty}\left(\Omega   \right) \\
\left.\left(v, \mathcal{L}_{-   s\partial_s} v\right)\right|_{\Sigma   } & =\left(v_0, v_1\right) \in    H^{\infty}\left(\Sigma   ;S^2T^*\Sigma\right) \oplus    H^{\infty}\left(\Sigma  ;S^2T^*\Sigma \right)
\label{eq:initialconds}
\end{align}
can be written as
$$
v=\chi v_0+\tilde{v}
$$
where $v_0 \in    H^{\infty}\left(\mathcal{I}^{+};s^{-2}S^2T^* \mathbb{S}^n\right)$ and $\tilde{v} \in  s^\beta H_{b}^{\infty}\left(\Omega ;S^2\ ^0T^*M \right)$ satisfy for all $k \in \mathbb{N}_0$ a tame estimate
$$
\begin{aligned}
\left\|v_0\right\|_{   H^k}+\|\tilde{v}\|_{ s^\beta H_{  b}^k} \leq C_k(\| & \left(f, v_0, v_1\right) \|_{k+d, \beta} \\
& \left.+\left(\left\|h_0\right\|_{   H^{k+d}}+\|\tilde{h}\|_{ s^\beta H_{  b}^{k+d}}\right)\left\|\left(f, v_0, v_1\right)\right\|_{0,\beta}\right)
\end{aligned}
$$ 
\end{proposition}
\begin{proof}
    Again, the proof is identical to \cite{hintz2024stability}. It inductively uses the previous Lemma to upgrade the estimate from Proposition \ref{prop:boundsgrow} to spaces with less growth. When we cross zero, we get the 0-order contribution plus a decaying remainder.
\end{proof}
\subsection{Solution of the gauge-fixed Einstein equations: proof of Theorem \ref{thm:nonlinsol} and stability of de Sitter space}
We wish to show existence of global solutions (on $\Omega$) to the gauge-fixed Einstein operator \eqref{eq:gaugefixedein} 
$$
\begin{aligned}
& P\left(h_0, \tilde{h}, \theta\right)=2\left(\operatorname{Ric}(g)-\Lambda g-\tilde{\delta}_g^*\left(\Upsilon\left(g ; g_0\right)+E_{g_0}\left(g-g_0\right)-\tilde{\chi} \theta\right)\right) \\
& \quad \hspace{18mm} g:=g_{\mathrm{b}}+\chi h_0+\tilde{h}, \quad g_0:=g_{\mathrm{b}}+\chi h_0.
\end{aligned}
$$
We write $\restr{D_1}{h_0}P(v_0,\tilde{h},\theta)=\restr{\frac{\mathrm{d}}{\mathrm{d}t}P(h_0+tv_0,\tilde{h},\theta)}{t=0}$ for the linearization in the first argument, evaluated at $h_0$ acting on $v_0$ and similarly $\restr{D_{1,2}}{h_0,\tilde{h}}P(v_0,\tilde{v},\theta)=\restr{\frac{\mathrm{d}}{\mathrm{d}t}P(h_0+tv_0,\tilde{h}+t\tilde{v},\theta)}{t=0} =\restr{D_1}{h_0}P(v_0,\tilde{h},\theta)+ \restr{D_2}{\tilde{h}}P(h_0,\tilde{v},\theta) $ for the linearization of the first two arguments at $h_0$ and $\tilde{h}$ acting on $v_0$ and $\tilde{v}$ respectively. For the third argument, we we use analogous notation.\\

What we have shown in Proposition \ref{prop:boundsdecay} is that the solution $v$ of 
\begin{equation}
\restr{D_2}{\tilde{h}}P(h_0,v,\theta)=f 
\label{eq:d2}
\end{equation}
with $f\in  s^\beta H^\infty_b(\Omega,S^2\ {}^0T^*M)$ can be written as $v= \chi v_0+\tilde{v}$, where $v_0\in H^\infty_b(\mathcal{I}^+;s^{-2}\operatorname{ker}\operatorname{tr_{g_{(0)}}})$ and $\tilde{v}\in s^\beta H^\infty_b(\Omega,S^2\ {}^0T^*M)$ and both obey tame estimates in their respective spaces. But because we require $\tilde{h}$ to decay at $\mathcal{I}^+$, this is not a suitable upgrade for $\tilde{h}$, which would be required when solving the equation using the Nash-Moser iteration scheme. So we rewrite it in such a way that the arguments of the linearizations are in the same spaces as $h_0,\tilde{h}$ and $\theta$. To wit, we compute
$$
\begin{aligned}
    \frac{1}{2}\restr{D_1}{h_0}&P(v_0,\tilde{h},\theta)=D_g \operatorname{Ric}\left(\chi v_0\right)-\Lambda \chi v_0\\
    &-\left(D_g \tilde{\delta}_{.}^*\right)\left(\chi v_0\right)\left(\Upsilon\left(g ; g_0\right)+E_{g_0}\left(g-g_0\right)-\tilde{\chi} \theta\right)\\
    &-\tilde{\delta}_g^*\left(\left.D_1\right|_g \Upsilon\left(\chi v_0 ; g_0\right)+\left.D_2\right|_{g_0} \Upsilon\left(g ; \chi v_0\right)+\left(D_{g_0} E .\right)\left(\chi v_0\right)\left(g-g_0\right)\right)\\
    \frac{1}{2}\restr{D_2}{\tilde{h}}&P(h_0,\tilde{v},\theta)= D_g \operatorname{Ric}(\tilde{v})-\Lambda \tilde{v}\\
    &-\left(D_g \tilde{\delta}_{.}^*\right)(\tilde{v})\left(\Upsilon\left(g ; g_0\right)+E_{g_0}\left(g-g_0\right)-\tilde{\chi} \theta\right)\\
    &-\tilde{\delta}_g^*\left(\left.D_1\right|_g \Upsilon\left(\tilde{v} ; g_0\right)+E_{g_0} \tilde{v}\right)\\
    \frac{1}{2}\restr{D_3}{\theta}&P(h_0,\tilde{h},\tilde{\theta})=-\tilde{\delta}_g^*(\tilde{\chi} \tilde{\theta}).
\end{aligned}
$$
We see that \eqref{eq:d2} is equivalent to 
\begin{equation}
\label{eq:fulllin}
    \restr{D_{1,2,3}}{h_0,\tilde{h},\theta}P(v_0,\tilde{v},\tilde{\theta})=f,
\end{equation}
if we define the gauge one-form update $\tilde{\theta}$ to be 
\begin{equation}
\label{eq:gaugeupdate}
    \tilde{\theta}:= -\left.D_2\right|_{g_0} \Upsilon\left(g ; \chi v_0\right)+E_{g_0}\left(\chi v_0\right)-\left(D_{g_0} E .\right)\left(\chi v_0\right) \tilde{h}.
\end{equation}
(Note that $\tilde{\theta} =\tilde{\chi}\tilde{\theta}$ due to the support condition on $\chi$). So, what Proposition \ref{prop:boundsdecay} shows is that the linear equation \eqref{eq:fulllin} with initial conditions for $\tilde{h}$ as in \eqref{eq:initialconds} has a unique solution $(v_0,\tilde{v},\tilde{\theta})$ where we have tame estimates for $h_0$ and $\tilde{h}$. So what is still missing is to show that the gauge update $\tilde{\theta}$ lies $s^\beta H^\infty_b(\Omega,{}^0T^*M)$ and obeys tame estimates. This is the content of the following Lemma.
\begin{lemma}[Tame estimates for the gauge modification]
     Suppose that we have tensors $h_0\in H^\infty(\mathcal{I}^+,s^{-2}S^2T^*\mathbb{S}^n)$ and $\tilde{h}\in s^\beta H^\infty_b(\Omega,S^2\ {}^0T^*M)$ with respective norms $\|h_0\|_{H^{d_n+2}}<\delta_0$ and $\|\tilde{h}\|_{s^\beta H^{d_{n+1}+2}_b}<\delta_0$ for some small $\delta_0$. Define $g_{(0)}$ and $h_{(0)}$ by \eqref{eq:gonbound} and let 
    $$
    v_0\in H^\infty(\mathcal{I}^+;s^{-2}\operatorname{ker}\operatorname{tr_{g_{(0)}}}).
    $$
    Define $\tilde{\theta}$ by \eqref{eq:gaugeupdate}. Then $\tilde{\theta}\in s^\beta H_b^\infty(\Omega;{}^0T^*M)$ with the tame estimate
    \begin{equation}
        \|\tilde{\theta}\|_{ s^\beta H_{b}^k} \leq C_k\left(\left\|v_0\right\|_{ H^{k+1}}+\left(\left\|h_0\right\|_{ H^{k+1+d_n}}+\|\tilde{h}\|_{ s^\beta H_{{b}}^{k+1+d_{n+1}}}\right)\left\|v_0\right\|_{ L^2}\right) .
    \end{equation}
\end{lemma}
\begin{proof}
    Again, the proof works identically as in \cite{hintz2024stability}.
\end{proof}
Now we can finish the proof of Theorem \ref{thm:nonlinsol}
\begin{proof}[Proof of Theorem \ref{thm:nonlinsol}]
    Let $d$ be as in Proposition \ref{prop:boundsdecay}. For $k\in \mathbb{N}_0$ define the spaces 
    $$
    \begin{aligned}
    \mathbf{B}^k&:=H^k(\mathcal{I}^+;s^{-2}T^*\mathbb{S}^n)\oplus s^\beta H^k_b(\Omega;S^2\ ^0T^*M)\oplus s^\beta H^k_b(\Omega;{}^0T^*M)\\
    B^k&:=s^\beta H^k_b(\Omega;S^2\ ^0T^*M)\oplus H^k(\Sigma;S^2\ ^0T^*M)\oplus  H^k(\Sigma;S^2\ ^0T^*M)
    \end{aligned}
    $$
    For $k\geq d$ and $(h_0,\tilde{h},\theta)\in \mathbf{B}^{\infty}$ with $\|(h_0,\tilde{h},\theta)\|_{\mathbf{B}^{3d}}\leq \delta_0$ for $\delta_0>0$ sufficiently small set 
    $$
    \Phi(h_0,\tilde{h},\theta)=(P(h_0,\tilde{h},\theta),\tilde{h}\left|_\Sigma\right., \mathcal{L}_{-s\partial_s}\tilde{h}\left|_\Sigma\right.)-(0,\underline{h}_0,\underline{h}_1).
    $$
    This is a map from a subset of $\mathbf{B}^\infty$ to $B^\infty$. We apply the main result of \cite{nashmoser} to solve $\Phi(h_0,\tilde{h},\theta)=0$. The tame estimate on the operator $\Phi$ follows from Lemma \ref{le:Pesti} and the low regularity estimates of the first and second derivatives of $\Phi$ follow from the algebra properties of $H^d_b$. \\
    
    Considering the right inverse of the linearization $D_{1,2,3}\left|_{h_0,\tilde{h},\theta}\right.\Phi$, we use Proposition \ref{prop:boundsdecay}. For a given $f\in s^\beta H^\infty_b(\Omega;S^2\ ^0T^*M) $ and $\underline{v}_0',\underline{v}_1'\in H^\infty(\Sigma;S^2\ ^0T^*M)$, this gives us the sections $v_0\in H^\infty(\mathcal{I}^+,s^{-2}\operatorname{ker}\operatorname{tr}_{g_{(0)}})$ and $\tilde{v}\in s^\beta H^\infty_b(\Omega;S^2\ ^0T^*M)$ for $g_{(0)}$ defined as in \eqref{eq:gonbound}. Defining $\tilde{\theta}$ like in \eqref{eq:gaugeupdate}, we have that 
    $$
    D_{1,2,3}\left|_{h_0,\tilde{h},\theta}\right.\Phi(v_0,\tilde{v},\tilde{\theta})=\left(D_{1,2,3}\left|_{h_0,\tilde{h},\theta}\right.P(v_0,\tilde{v},\tilde{\theta}),\tilde{v}\left|_{\Sigma}\right.,\mathcal{L}_{-s\partial_s}\tilde{v}\left|_{\Sigma}\right.\right)=(f,\underline{v}_0',\underline{v}_1'),
    $$
    with $(v_0,\tilde{v},\tilde{\theta})$ satisfying tame estimates.\\
    
    Using the smoothing operators of Lemma \ref{le:smoothingops} we conclude the proof.
\end{proof}
We finish this section with a standard argument to show that we can construct a (future) solution to the Einstein vacuum equation with prescribed first and second fundamental forms on $\Sigma$ from the solution to $P(h_0,\tilde{h},\theta)=0$. By observing that de Sitter space is symmetric under the change of variables $s\rightarrow\pi-s$, we can immediately upgrade this to a past solution as well. Finally, we can easily glue them together to a global solution.\\

We call the first and second fundamental forms of $\Sigma$ induced by the de Sitter metric $\gamma_{\mathrm{dS}}$ and $k_{\mathrm{dS}}$ respectively. Also, call $\Omega^-:=[\pi-s_0,\pi]\times\mathbb{S}^n$ the neighborhood of $\mathcal{I}^-$, the analogon to $\Omega$ around $\mathcal{I}^+$
\begin{theorem}[Global stability of de Sitter space]
    \label{thm:solution}
    Let $d_{n+1}+2\leq d\in \mathbb{N}$. There exists a $D\in\mathbb{N}$ such that the following holds. For all $\beta\in(0,1)$ and $\delta_0>0$, there exists an $\epsilon>0$ such that if 
    $$
    \gamma=\gamma_{\mathrm{dS}}+\tilde{\gamma},\quad k=k_{\mathrm{dS}}+\tilde{k},\quad \gamma,k \in H^\infty(\Sigma,S^2T^*\Sigma),
    $$
    with $\|\tilde{\gamma}\|_{H^D}\leq \epsilon$ and $\|\tilde{k}\|_{H^D}\leq \epsilon$ and $(\gamma,k)$ satisfy the constraint equations, the maximally extended globally hyperbolic developments of these initial data isometric to 
    $$
    (M_{\mathrm{dS}},g),\quad g\left|_{\Omega}\right.= g_{\mathrm{dS}} + \chi h_0 + \tilde{h}, 
    $$
    with $h_0\in H^\infty(\mathcal{I}^+;s^{-2}S^2T^*\mathbb
    S^n)$ and $\tilde{h}\in s^\beta H^\infty_b(\Omega,S^2\ {}^0T^*M)$. Furthermore, the metric corrections are small in the low regularity norms; $\|h_0\|_{H^d_b}\leq\delta_0$ and $\|\tilde{h}\|_{s^\beta H^d_b}\leq\delta_0$. The metric near past conformal infinity $\mathcal{I}^-$ is of the same form
    $$
     g\left|_{\Omega^-}\right.=g_{\mathrm{dS}} + \chi^- h_0^- + \tilde{h}^-,
    $$
    where $\chi^-(s)=\chi(\pi-s)$ and $h_0^-,\tilde{h}^-$ lie in the same function spaces and fulfill the same smallness estimates if $s$ is replaced by $\pi-s$. Lastly, the constructed spacetime ist past and future geodesically complete.
\end{theorem}
The fact that that the metric corrections are small in the low regularity norms imply that the geometry of $g$ at $\mathcal{I}^+$ and $\mathcal{I}^-$ is qualitatively the same as that of $g_{\mathrm{dS}}$ if we choose $d\geq d_{n+1}$. This constitutes a new proof of the stability of de Sitter space in $n+1$ spacetime dimensions.\\

Furthermore, metrics of this form are future and past geodesically complete. This follows from the form of the Cristoffel symbols for exact de~Sitter space.
\begin{proof}
First, we construct a future solution. We wish to use Theorem \ref{thm:nonlinsol} and then show that we can arrange it so that the gauge condition is satisfied on $\Omega$. This is a standard argument; see, e.g. \cite{ringstromcauchy}. 
\\

\underline{Step 1. Construction of initial data}\\
We construct initial data $\underline{h}_0,\underline{h}_1$ as in the notation of Theorem \ref{thm:nonlinsol} from $\gamma, k$. Because of the support conditions on $\chi$ and $\tilde{\chi}$, at the initial surface $\Sigma$ our gauge is just the standard deTurck gauge\cite{DeTurck}. Concretely, define the product neighborhood
$$
\tilde{s}= (s-s_0),\quad (-\epsilon_0,\epsilon_0)_{\tilde{s}}\times T\Sigma
$$
of $\Sigma$ for an $\epsilon_0>0$ small enough. In the product coordinates, $g_{\mathrm{dS}}=-\frac{1}{\sin^2(\tilde{s}+s_0)}\mathrm{d}\tilde{s}^2 + \gamma_{\mathrm{dS}}$. Define $\underline{h}_0=\tilde{\gamma}$, where we consider it as an element of $ H^\infty(\Sigma;S^2\ {}^0T^*M)$ by inclusion. Now $g_{\Sigma}:=g_{\mathrm{dS}}+\underline{h}_0=-\frac{1}{\sin^2(\tilde{s}+s_0)}\mathrm{d}\tilde{s}^2 + \gamma_{\mathrm{dS}}+\tilde{\gamma}$ has indeed first fundamental form $\gamma$. Furthermore, the $H^D$-norm of $\underline{h}_0$ can be made arbitrarily small by choosing $\epsilon$ small enough. The future unit normal to $\Sigma$ is $\nu:= -\sin(s_0)\partial_s$ for both $g_{\Sigma}$ and $g_{\mathrm{dS}}$. We define the total metric $g:=g_{\Sigma}+\tilde{s}\underline{h}_1$ with $\underline{h}_1$ to be determined by matching the second fundamental form of $g$ with $k$ and fixing the gauge condition to be 0 on $\Sigma$. So we require for all sections $X,Y\in C^\infty(\Sigma;T\Sigma)$ the equality
$$
\begin{aligned}
    k(X,Y)&\stackrel{!}{=} g(\nabla^g_X Y,\nu)\\
    &=g((\nabla^g_X-\nabla^{g_{\mathrm{dS}}}_X)Y,\nu) + g_{\mathrm{dS}}(\nabla^{g_{\mathrm{dS}}}_XY,\nu) + \underline{h}_0(\nabla^{g_{\mathrm{dS}}}_XY,\nu)\\
        &=g_{\Sigma}((\nabla^g_X-\nabla^{g_{\mathrm{dS}}}_X)Y,\nu)  + k_{\mathrm{dS}}(X,Y) 
\end{aligned}
$$
This is equivalent to 
\begin{equation}
\label{eq:secondff}
        g_{\Sigma}((\nabla^{g_{\Sigma} + \tilde{s}\underline{h}_1}_X-\nabla^{g_{\Sigma}}_X)Y,\nu)\left.\right|_{\tilde{s}=0} \stackrel{!}{=} \tilde{k}(X, Y) - g_{\Sigma}((\nabla^{g_{\Sigma}}_X-\nabla^{g_{\mathrm{dS}}}_X)Y,\nu) .
\end{equation}
The right hand side is the evaluation of an element of $H^\infty(\Sigma;S^2T^*\Sigma)$ on $(X,Y)$, which is small in $H^{D-1}$ because of the smallness of $\tilde{k}$ and $\underline{h}_0$. This uniquely determines the tangential part of $\underline{h}_1$. To see this explicitly, we write equation \eqref{eq:secondff} using the local frame $e_\mu$. Using $e_\mu \tilde{s}=0$ unless $\mu=0$ this gives at $\tilde{s}=0$ for $(X,Y) = (e_i,e_j)$
$$
 \frac{\sin(s_0)}{2}(\underline{h}_1)_{ij}=\tilde{k}_{ij}-\frac{\sin(s_0)}{s_0}(\underline{h}_0)_{ij}=\tilde{k}_{ij}-\frac{\sin(s_0)}{s_0}\tilde{\gamma}_{ij}.
$$
To fix the rest of the components, we turn to the gauge condition
\begin{equation}
    \operatorname{tr}_{g_{\Sigma}}(\nabla^{g_{\Sigma+\tilde{s}\underline{h}_1}}-\nabla^{g_{\mathrm{dS}}})=0
\end{equation}
In the frame this gives for the $i$- and 0-components of this equation
$$
\begin{aligned}
\frac{\sin^2(s_0)}{s_0}(\underline{h}_1)_{0i}&=\left(\operatorname{tr}_{g_{\Sigma}}(\nabla^{g_\Sigma}-\nabla^{g_{\mathrm{dS}}})\right)_i 
\\
\frac{\sin^2(s_0)}{2s_0}(\underline{h}_1)_{00}&=\left(\operatorname{tr}_{g_{\Sigma}}(\nabla^{g_\Sigma}-\nabla^{g_{\mathrm{dS}}})\right)_0 + \frac{s_0}{2}(g_\Sigma)^{ij}(\underline{h}_1)_{ij},
\end{aligned}
$$
where the index on the right hand side was lowered using $g_\Sigma$. Now, after application of Lemma \ref{le:2.9}, we see that the $H^{D-d_{n+1}-1}_b(\Sigma,S^2\ {}^0T^*M)$ norm of $\underline{h}_1$ is small.\\

\underline{Step 2: Finding the solution}\\
Let $D'\in\mathbb{N}$ be the large order required in Theorem \ref{thm:nonlinsol}, then we can apply said Theorem if we choose $D=D'+d_{n+1}+1$. This gives us $\tilde{h},h_0, \theta$ that fulfill $P(h_0,\tilde{h},\theta)=0$ and $(\tilde{h}\left|_\Sigma\right.,(\mathcal{L}_{-s\partial_s}\tilde{h})\left|_\Sigma\right.)=(\underline{h}_0,\underline{h}_1)$. Writing
\begin{equation}
    g=g_{\mathrm{dS}}+\chi h_0 + \tilde{h},\quad g_0=g_{\mathrm{dS}}+\chi h_0,
\end{equation}
we have 
\begin{equation}
\operatorname{Ric}(g)-\Lambda g-\tilde{\delta}_g^* \eta=0, \quad \eta:=\Upsilon\left(g ; g_0\right)+E_{g_0}\left(g-g_0\right)-\tilde{\chi} \theta.
\end{equation}
By construction, and the support conditions on $E$ and $\tilde{\chi}$, we have $\eta\left|_\Sigma\right.=0$. But because $\gamma,k$ fulfill the constraint equations, a standard argument such as, for example, in \cite{ringstromcauchy} shows that also $\mathcal{L}_{-s\partial_s} \eta$ vanishes at $\Sigma$. Because of the second Bianchi identity, $\eta$ fulfills the linear wave type equation $
\delta_g \mathrm{G}_g \tilde{\delta}_g^* \eta=0$ on $\Omega$. We conclude that $\eta$ vanishes on the entirety of $\Omega$. Therefore, also
$$
\operatorname{Ric}g=\Lambda g
$$
holds on $\Omega$. Lastly, we see that $g$ (so far only defined on $\Omega$) is exactly of the form of the Theorem.\\

\underline{Step 3. Global solution}\\
As the gauge condition at on $[s_0,s_0/2]\times\mathbb{S}^n$ is just $\Upsilon(g;g_{\mathrm{dS}})=0$, we can apply a standard finite in time evolution result \cite{ChoquetBruhatLocalEinstein}. Using the constructed initial data at $\Sigma$, we get a solution $g'$ to $\operatorname{Ric}g'-\Lambda g'=0$ on $[s_0/2,\pi-s_0/2]\times\mathbb{S}^n$ that fulfills $\Upsilon(g';g_{\mathrm{dS}})=0$ and the initial conditions at $\Sigma$.\\

Working in the coordinate $s':=\pi-s$, we see that $g'$ induces initial data at $\Sigma^-:=\{s'=s_0\}\times\mathbb{S}^n$ of the form $$
(\underline{h}_0^-,\underline{h}_1^-):=\left((g'-g_{\mathrm{dS}})\left|_{\Sigma^-}\right.,(-\mathcal{L}_{s'\partial_{s'}}(g'-g_{\mathrm{dS}}))\left|_{\Sigma^-}\right.\right).
$$
By the finite in time result, they have small $H^{D'}_b$ norm. This is initial data of the type required to apply Theorem \ref{thm:nonlinsol} on the domain $\Omega^-$. Therefore, performing step 2 of the above proof, gives us a metric $g''$ defined on $\Omega^-$ that is of the form claimed in the Theorem.\\

What remains to show, is that $g$ and $g'$ agree on $[s_0/2,s_0]\times\mathbb{S}^n$ and $g'$ and $g''$ agree on $[\pi-s_0,\pi-s_0/s]\times\mathbb{S}^n$. But as on those two domains the respective metrics are solutions to the Einstein equations in the gauge $\Upsilon(\cdot;g_{\mathrm{dS}})=0$ for the same initial data, by a standard argument they are the same: We write the difference of the gauged Einstein equations as a linear wave equation for the difference of the metrics. Uniqueness of solutions to linear wave equations gives the result. See \cite{ringstromcauchy} for arguments of this type.\\

The past and future geodesic completeness follows from the properties of the Christoffel symbols of exact de~Sitter space and the smallness of the metric corrections of our solution. 
\end{proof}

\section{Expansion at the conformal boundary}\label{se:expansion}
In this section we will use the diffeomorphism invariance of the Einstein equations to pull back the solution obtained in Theorem \ref{thm:solution} to improve control over the solution at $\mathcal{I}^+$. Of course the exact same proof also works near $\mathcal{I}^-$. First, we show that we can pull back the solution to be log-smooth at $\mathcal{I}^+$ and then we prove an explicit asymptotic expansion. Only the gauge one-form $\theta$ impedes log smoothness (see the proof of log smoothness \ref{le:logsmooth}). So, as a first step, we wish to show that we can construct a pullback such that the gauge form vanishes to infinite order at $\mathcal{I}^+$.
\subsection{Construction of pullbacks}\label{se:pullbacks}
Consider a 0-vector field $V$ on $M$ that vanishes near $\Sigma$. We define the diffeomorphism $\phi^t_V$ as the flow along $V$ for time $t$ and then 
\begin{equation}
\label{eq:diffeo}
\phi_V:=\phi_V^{t=1} 
\end{equation}
as the flow along $V$ for $t=1$.
\begin{lemma}
\label{le:flowclose}
    Let $X\in s^\beta H^\infty_b(\Omega;{}^0TM)$ for $\beta>0$, and assume that $X$ vanishes near $\Sigma$. Let $\epsilon>0$. Then there exists $\delta>0$ such that if $\|X\|_{C^0_b(\Omega;{}^0TM)}<\delta$, then for all $(s,\omega)\in\Omega$ we have 
    $$
    \phi _X^t (s,\omega) = (s(1+a), \omega'), 
    $$
    with $|a(t,s,\omega)|<\epsilon$ for all $t\leq2$ (independently of $(s,x)$). Recalling the notation of section \ref{se:manifold}, in a chart around $\omega$, say without loss of generality $\omega\in{}^\uparrow\mathbb{S}^n$, we have that $\frac{|\omega^\mu-\omega'^\mu|}{s}\leq \epsilon$.
\end{lemma}
\begin{proof}
    Fix $p=(s,\omega)\in\Omega $ and assume without loss of generality that it lies in $^\uparrow\Omega$. Call $^\uparrow P(p)=p^\mu=(p^0,p^{i})$ its coordinates in the chart $^\uparrow P$. We have $\phi_X^t(p)=\gamma(t)$ where $\gamma$ is the integral curve of $X$ starting at $p$, namely $\gamma$ fulfills
    \begin{align}
    \label{eq:flowequation}
        \gamma'(t)&=X(\gamma(t))\\
        \gamma(0)&=p.
    \end{align}
    We can assume that the curve still lies in the domain, where the stereographic projection $^\uparrow p$ is still valid, as this curve will be contained in a $2\delta$ neighborhood of the original point and does not cross $\mathcal{I}^+$ in finite time. We call $U_\delta$ this $2\delta$ neighborhood of $^\uparrow\Omega$.
    \\
    
    We work in the coordinate system on $^\uparrow\Omega$ and define $\tilde{\gamma}^\mu$ via $\gamma^\mu(t)=p^\mu + s\tilde{\gamma}^\mu(t)$. It is the solution to the system
    \begin{align}
        (\tilde{\gamma}^\mu) '(t)&= \frac{1}{s}(X(p^\nu+s\tilde{\gamma}^\nu(t)))^\mu\\
        \tilde{\gamma}^\mu(0)&=0.
    \end{align}
    Now, integrating this ODE, we get
    $$
        \tilde{\gamma}^\mu=\int_0^t\frac{1}{s}(X(p^\nu+s\tilde{\gamma}^\nu(\tau)))^\mu\mathrm{d}\tau.
    $$
    We write $X(q)=X^\mu(q) e_\mu=q^0X^\mu(q)\partial_\mu$ in the frame $e_\mu$. This leads to the notation $X(q)^\mu=q^0X^\mu(q)$. (The $\mu$ component of $X$ in the coordinates on $^\uparrow\Omega$, $X(q)^\mu$, is given by $q^0$ times the $X^\mu(p)$, the $\mu$ component of $X$ in the frame $e_\mu$). In the integral, we therefore get for the $\mu$ component of $\tilde{\gamma}$
    $$
    \begin{aligned}
    \tilde{\gamma}^\mu(t)&=\int_0^t \frac{1}{s}\gamma^0(\tau) X^\mu(p^\nu+s\tilde{\gamma}^\nu(\tau))\mathrm{d}\tau\\
    &=\int_0^t \frac{1}{s}\left(s+s\tilde{\gamma}^0(\tau) \right)X^\mu(p^\nu+s\tilde{\gamma}^\nu(\tau))\mathrm{d}\tau
    \end{aligned}
    $$
    We estimate this as 
    
    \begin{equation}
    \begin{aligned}
    \label{eq:flowesti}
    |\tilde{\gamma}^\mu(t)|&\leq \int_0^t (\operatorname{sup}_{U_\delta}|X^\mu|)(1+|\tilde{\gamma}^0(\tau)|)\mathrm{d}\tau\\
    &\leq \int_0^t \|X\|_{\mathcal{C}^0_b}(1+|\tilde{\gamma}^0(\tau)|)\mathrm{d}\tau\\
    &\leq\int_0^t \delta(1+|\tilde{\gamma}^0(\tau)|)\mathrm{d}\tau.
    \end{aligned}
    \end{equation}
    
    First, look at the equation for $\mu=0$. We apply Grönwall's Lemma to get 
    $$
        |\tilde{\gamma}^0(t)|\leq \delta te^{\delta t}
    $$
    For $t<2$ we can estimate the right-hand side by $C_\delta$. Now, inserting this estimate into equation \eqref{eq:flowesti} for the $i$ components, we get for $t<2$, $|\tilde{\gamma}^\mu|\leq C_\delta'$. Importantly, this constant is independent of $p=(s,\omega)$, and can therefore be made smaller than $\epsilon$ uniformly, for all $p\in\Omega$. This proves the Lemma. 
\end{proof}
The purpose of this Lemma is to show that for functions $f$ on $M$ we have that $f\circ\phi_X$ has the same decay behavior as $f$ toward $\mathcal{I}^+$, as the $s$  coordinate of the input of $f$ changes by a small amount relative to the size of $s$. With it, we prove a Lemma that shows that the flow along a b-vector field which decays toward $\mathcal{I}^+$, only changes the coordinates in a way where the difference of coordinates and its b-derivatives also decay in the same manner.

\begin{lemma}
\label{le:flowsobolev}
    Let $X\in s^\beta H^\infty_b(\Omega;{}^0TM)$ for $\beta>0$ and assume that $X$ vanishes near $ \Sigma$. Let $\epsilon>0$. There exists $\delta>0$ such that if $\|X\|_{C^0_b(\Omega;{}^0TM)}<\delta$, then 
    $$
    ^\uparrow\varphi\cdot({}^\uparrow P\circ\phi_X-{}^\uparrow P)\in s\cdot s^\beta H^\infty_b(\Omega).
    $$
    The analogous statement for $^\downarrow\Omega$ holds as well.
\end{lemma}
\begin{proof}
    First, notice that the function is well defined because of the cutoff $^\uparrow\varphi$, and because for $p\in{}^\uparrow\Omega$, $\phi_X(p)$ still lies in a $\delta$ neighborhood of $^\uparrow\Omega$, where $^\uparrow P$ is still defined for $\delta$ small enough.\\

    Let $p=(s,\omega)\in {}^\uparrow\Omega$ and let $p^\mu=(p^0,p^{i})=(s,\omega^{i})$ be its coordinates. Writing $\gamma(t)$ for the flow starting at $p$ along $X$, working in coordinates, we split it as $\gamma^\mu(t)=p^\mu+\bar{\gamma}^\mu(t)$. Then the $\bar{\gamma}$ fulfills the system
    \begin{equation}
        \label{eq:bargammasyst}
        \begin{aligned}
            (\bar{\gamma}^\mu)'(t)&=X(\gamma(t))^\mu=X(p^\nu+\bar{\gamma}^\nu(t))^\mu\\
            \bar{\gamma}^\mu(0)&=0.
        \end{aligned}
    \end{equation}
    What we need to show is that arbitrary b-derivatives with respect to the starting point $p$ of $\bar{\gamma}^\mu(1)$ are (as functions of said starting point) integrable in $s^\beta L^2$. First, we show this for the case of no derivatives.\\

    Integrating the system, and again writing $X=X^\mu e_\mu$ and therefore $X(p)^\mu=p^0X^\mu(p)$, we get 
    \begin{equation}
    \label{eq:Xoorderesti}
    \begin{aligned}
        |\bar{\gamma}^\mu(t)|&\leq\int_0^t\left|(p^0+\bar{\gamma}^0(\tau))X^\mu(\gamma(\tau))\right|\mathrm{d}\tau\\
        &\leq\int_0^t\left| sX^\mu(\gamma(\tau))\right| + \left|\bar{\gamma}^0(\tau)\right|\|X\|_{\mathcal{C}_b^0}\mathrm{d}\tau
        \end{aligned}
    \end{equation}
    After applying Grönwalls Lemma to the $\mu=0$ component of this equation, we get 
    $$
    |\bar{\gamma}^0(t)|\leq e^{\delta t}\int_0^t\left| sX^\mu(\gamma(\tau))\right| \mathrm{d\tau}.
    $$
    Now, because of Lemma \ref{le:flowclose}, the evaluation of $X^\mu$ at $\gamma(\tau)$ for $0\leq\tau\leq2$ instead of $p$ does not change the integrability properties of $X^\mu$. Therefore, the right-hand side lies in $s\cdot s^\beta L^2$ for $t\leq1$ by definition of $X$.\footnote{In fact, this function is only defined in a neighborhood of $^\uparrow\Omega$, but as stated in the Lemma, we want to analyze functions after multiplication with the cutoff $^\uparrow\varphi$. So when we say that a function lies in a certain function space during this proof and the proof of Lemma \ref{le:flowdif}, we mean that the statement is true, after multiplication with the cutoff. } Plugging in this result for the $i$ components of equation \eqref{eq:Xoorderesti}, we get that also the $|\bar{\gamma}^i(t)|$ lie in $s\cdot s^\beta L^2$ (for $t\leq1$). All the $\bar{\gamma}^\mu$ also lie in $s\cdot \mathcal{C}^0_b$, which we will need later. This proves the claim in the case of 0 b-derivatives.\\
    
    In order to control over $k$ b-derivatives, we derive the system \eqref{eq:bargammasyst} with respect to the initial point $k$ times. We first consider one b-derivative in detail, after which the structure will become apparent.

    Differentiating the equation \eqref{eq:bargammasyst} with respect to $s\partial_s$ we get
    \begin{equation}
    \label{eq:firsdersyst}
        \begin{aligned}
            (s\partial_s\bar{\gamma}^\mu)'(t)&= s\partial_s(X(\gamma(t))^\mu)\\
            s\partial_s\bar{\gamma}^\mu(0)&=0
        \end{aligned}
    \end{equation}
    We evaluate the right-hand side as
    $$
    \begin{aligned}
        s\partial_s(X(\gamma(t))^\mu)&=s\partial_s\left((p^0+\bar{\gamma}^0(t))X^\mu(\gamma(t))\right)\\
        &=(s+s\partial_s\bar{\gamma}^0)X^\mu(\gamma(t)) + (s+\bar{\gamma}^0(t))(\partial_\nu X^\mu(\gamma(t)))s\partial_s\bar{\gamma}^\nu.\\
        &=sX^\mu(\gamma(t)) + [s\partial_s\bar{\gamma}^0\cdot X^\mu(\gamma(t))+s\partial_s\bar{\gamma}^\nu\cdot(s+\bar{\gamma}^0(t))(\partial_\nu X^\mu(\gamma(t)))],
    \end{aligned}
    $$
    where in the last line the equation was just rearranged to show that it has the schematic structure
    $$
        A + (s\partial_s \bar{\gamma})\cdot B,
    $$
    where $A\in s\cdot s^\beta L^2$, $A\in s\cdot \mathcal{C}^0_b$ and $B$ can be estimated by a constant. That $A$ is in said function spaces, uses Lemma \ref{le:flowclose} to show that the point of evaluation of $X^\mu$ does not change the integrability (here, we use $\beta>0$ for the boundedness). Lemma \ref{le:flowclose} will be used extensively like this from here on. That B can be estimated by a constant uses that $\bar{\gamma}^0$ lies in $s\cdot \mathcal{C}^0_b$, as the additional factor of $s$ is needed to be put together with the $\partial_\nu X^\mu$ term to produce $e_\nu X^\mu$ (otherwise in the case $\nu=0$ we would lose control of this term). Now, knowing this structure, we integrate the system \eqref{eq:firsdersyst} and sum over the components $\mu$ to get for $|\bar{\gamma}(t)|:=\sum_\mu|\bar{\gamma}^\mu(t)|$
    \begin{equation}
        |s\partial_s\bar{\gamma}(t)|\leq\int_0^t |A| + |(s\partial_s \bar{\gamma})|\cdot |B|\mathrm{d}\tau,
    \end{equation}
    for different $A$ and $B$ arising from the sum of all the component equations, but having the same structure. Grönwalls Lemma gives
    \begin{equation}
         |s\partial_s\bar{\gamma}(t)|\leq e^{Ct}\int_0^t|A|\mathrm{d}\tau = C'\tilde{A},
    \end{equation}
    where $\tilde{A}$ lies in the same function spaces as $A$. This proves $|s\partial_s\bar{\gamma}(t)|\in s\cdot s^\beta L^2$ and $\in s\cdot \mathcal{C}^0_b$ for $t\leq1$

    To study derivatives with respect to the spatial components $\partial_{p^{i}}=\partial_{\omega^{i}}$ we again take the derivative of equation \eqref{eq:bargammasyst}. We get
    \begin{equation}
\begin{aligned}
    (\partial_{\omega^{i}}\bar{\gamma}^\mu)'(t)&= \partial_{\omega^{i}}(X(\gamma(t))^\mu)\\
            \partial_{\omega^{i}}\bar{\gamma}^\mu(0)&=0
\end{aligned} 
    \end{equation}
    Computing the right-hand side, we get 
    $$
    \begin{aligned}
        \partial_{\omega^{i}}(X(\gamma(t))^\mu) &=\partial_{\omega^{i}}\left((p^0+\bar{\gamma}^0(t))X^\mu(\gamma(t))\right)\\
        &=\partial_{\omega^{i}} \bar{\gamma}^0(t)X^\mu(\gamma(t)) + (s+\bar{\gamma}^0(t))(\partial_\nu X^\mu(\gamma(t)))(\delta^\nu_i + \partial_{\omega_i}\bar{\gamma}^\nu(t))\\
        &=(s+\bar{\gamma}^0(t))\partial_i X^\mu(\gamma(t))\\
        &\hspace{5mm}+[\partial_{\omega^{i}} \bar{\gamma}^0(t)\cdot X^\mu(\gamma(t)) +\partial_{\omega_i}\bar{\gamma}^\nu(t)\cdot(s+\bar{\gamma}^0(t))(\partial_\nu X^\mu(\gamma(t)))].
    \end{aligned}
    $$
    This is again of the form
    $$
         A + (\partial_{\omega^{i}} \bar{\gamma})\cdot B,
    $$
    for $A$ and $B$ of the same structure as before (the term $\bar{\gamma}^0(t)$ in $A$ can be split into $s\cdot f$ with $f$ bounded by a constant because of the 0th order result). Indeed, the $B$ is the identical $B$ as in the $s\partial_s$ case. Therefore, we can perform the same proof as before and get $| \partial_{\omega^{i}}\bar{\gamma}|\in s\cdot s^\beta L^2$ and $\in s\cdot \mathcal{C}^0_b$ for $t\leq1$. This proves the statement of the Lemma for one b-derivative.\\

    For general order, the procedure goes by induction.  Writing $D^q_b$ for a $q$-fold composition of b-derivatives (so $s\partial_s$ and $\partial_{\omega^{i}}$), suppose for $q\leq k-1 $ we know $| D^{q}_b\bar{\gamma}(t)|\in s\cdot s^\beta L^2$ and $\in s\cdot \mathcal{C}^0_b$ for $t\leq1$ and that the right-hand side of equation \eqref{eq:bargammasyst} derived $k-1$ times is of the form 
    \begin{equation}
        A + (D^{k-1}_b \bar{\gamma})\cdot B
    \end{equation}
    The $A$ is a sum of terms of the form 
    $$
    s^{a}(\partial_i^l\partial_0^mX)\prod_{i=1}^k(D^{c_i}_b\bar{\gamma})^{d_i},
    $$
    where $\sum_ic_i <k-1$ and for $d:=\sum_id_i$, $m\leq a+d -1$. (The latter condition ensures that $A\in s\cdot s^\beta L^2$ and $\in s\cdot \mathcal{C}^0_b$. This is because the problem terms $\partial_o^mX$ can be combined with $s^{a}\prod_{i=1}^k(D^{c_i}_b\bar{\gamma})^{d_i}$, which in turn can be estimated as $|s^{a}\prod_{i=1}^k(D^{c_i}_b\bar{\gamma})^{d_i}|\leq C s^{a+d}$ and therefore $A\leq C s\cdot |\partial_\omega^ls^m\partial_0^mX|$.) The $B$ term is given by the exact same expression as in the first order case with $D_b\bar{\gamma}$ replaced by $D_b^{k-1}\bar{\gamma}$.

    Now, if we differentiate this expression with respect to $D_b$, if the derivative hits the $D_b^{k-1}\bar{\gamma}$ term, we get $D_b^k\bar{\gamma}\cdot B$. If the $D_b$ hits $B$, we get a new contribution to $A$, as can be seen by direct computation because of the explicit form of $B$. If it hits $A$, we get new terms of the same form as $A$ with $k-1$ replaced by $k$. The only term where this is not immediately obvious is when evaluating $D_b(\partial_i^l\partial_0^mX)$ we get an additional $\partial_0$ derivative acting on $X$. But by the chain rule, this also adds a $D_b\bar{\gamma}^0$ factor, which increases $d$ by one and therefore still, $a+(d+1)-1\geq m+1$. Therefore, we have 
    $$
    \begin{aligned}
        (D_b^k\bar{\gamma}^\mu)'(t)&=A + B\cdot D_b^k\bar{\gamma}^\mu(t)\\
        D_b^k\bar{\gamma}^\mu(0)&=0
    \end{aligned}
    $$
    and by the same argument as in the first order case using Grönwalls Lemma, we get $| D^{q}_b\bar{\gamma}(t)|\in s\cdot s^\beta L^2$ and $\in s\cdot \mathcal{C}^0_b$ for $t\leq1$ for $q\leq k$. This closes the induction and the proof.
\end{proof}

Now that we have under control how the flow along a b-vector field changes the coordinates, we investigate how the difference of two flows along two b-vector fields with different order of decay changes the coordinates.

\begin{lemma}
\label{le:flowdif}
    Let $V\in s^\beta H^\infty_b(\Omega;{}^0TM)$, $W\in s^\alpha H^\infty_b(\Omega;{}^0TM)$ with $\alpha\geq\beta >0$. Assume further that they vanish for $s>s_0/2$. There exists a $\delta>0$ such that if $\|V\|_{\mathcal{C}^0_b}<\delta$ and $\|W\|_{\mathcal{C}^0_b}<\delta$ we have that $|{}^\uparrow\varphi\cdot( {}^\uparrow P\circ\phi_{V+W}- {}^\uparrow P\circ\phi_V)|\in s\cdot s^{\alpha}H_b^\infty(M)$. (The analogous statement for $^\downarrow\Omega$ holds as well).
\end{lemma}
\begin{proof}
    This proof works in a very similar way to the proof of the previous Lemma, just some more care needs to be taken to distinguish which term decays at which order.

    Fix $p=(s,\omega)\in\Omega$, say in $^\uparrow\Omega$. The other case is analogous. We work in the coordinate system on $^\uparrow\Omega$ and consider the integral curves $\gamma_{V},\gamma_{V+W}$ of the 0-vector fields $V, V+W$ starting at $(s,\omega)$, and their difference $h^\mu:=\gamma_{V+W}^\mu-\gamma_V^\mu$. They are well defined for $t<2$, lie in a $4\delta$ neighborhood of $^\uparrow\Omega$ and do not cross $\mathcal{I}^+$ and $\Sigma$. Because of the definition of the flow, $h(t)^\mu$ fulfills the following initial value problem
    $$
    \begin{aligned}
        (h^\mu)'(t)&=V(\gamma_{V+W}(t))^\mu-V(\gamma_V(t))^\mu + W(\gamma_{V+W}(t))^\mu\\
        h^\mu(0)&=0
    \end{aligned}
    $$
    We write the difference of the $V$ terms on the left-hand side as
$$
\left(V\right.(\gamma_{V+W}(t))-V(\gamma_V(t)\left.)\right)^\mu
    =\int_0^1 \frac{\mathrm{d}}{\mathrm{d}\xi} V(\gamma_{V}(t)+\xi h(t))^\mu\mathrm{d}\xi.
$$
    Simplifying the notation by writing $\gamma_V(t) + \xi h(t) = f(t)$, we compute
    $$
    \begin{aligned}
    \left(V(\gamma_{V+W}(t))-V(\gamma_V(t))\right)^\mu&=\int_0^1 \frac{\mathrm{d}}{\mathrm{d}\xi} V(f(t))^\mu\mathrm{d}\xi\\
    &=\int_0^1 \frac{\mathrm{d}}{\mathrm{d}\xi}\left( f(t)^0V^\mu (f(t))\right)\mathrm{d}\xi\\
    &=\int_0^1 [f(t)^0(\partial_\nu V^\mu)(f(t))\cdot h^\nu(t)]  \\&\hspace{6mm}+[V^\mu(f(t))\cdot h^0(t)]\mathrm{d}\xi\\
    &=\int_0^1 [(e_\nu V^\mu)(f(t))\cdot h^\nu(t)] + [V^\mu(f(t))\cdot h^0(t) ]\mathrm{d}\xi
    \end{aligned}
    $$
    Inserting in the initial value problem, this gives
    \begin{equation}
        \label{eq:ivpVWdiff}
        \begin{aligned}
        (h^\mu)'(t)&=\int_0^1 [(e_\nu V^\mu)(f(t))\cdot h^\nu(t)] + [V^\mu(f(t))\cdot h^0(t) ]\mathrm{d}\xi\\
        &\hspace{5mm}+ W(\gamma_{V+W}(t))^\mu\\
        h^\mu(0)&=0
        \end{aligned}
    \end{equation}
    Now, the proof goes similarly to \ref{le:flowsobolev}. Applying Lemma \ref{le:flowclose} to $V, W$ and $V+W$, we see that the evaluation of the components $V^\mu,W^\mu$ at $\gamma_V(t),\gamma_{V+W}(t),f(t)$ instead of $p$ does not change their integrability behavior.

    First, we look at an estimate for $|h(t)|$. To wit, integrate equation \eqref{eq:ivpVWdiff} to get $h^\mu(t)=\int_0^t(h^\mu)'(\tau)\mathrm{d}\tau$. We see, that all terms proportional to a $h^\mu(\tau)$ have prefactors of the form $V^\mu, e_\nu V^\mu$ evaluated at $f(\tau)$. We estimate them by $\|V\|_{\mathcal{C}^1_b}<\infty$ using Sobolev embedding. The $W$ term, we write as 
    $$W
    (\gamma_{V+W}(\tau))^\mu = \gamma_{V+W}^0(\tau)W^\mu(\gamma_{V+W}(\tau)) = s\cdot \frac{\gamma_{V+W}^0(\tau)}{s}\cdot W^\mu(\gamma_{V+W}(\tau)).
    $$
    Lemma \ref{le:flowsobolev} applied to $V+W$ shows that $\frac{\gamma_{V+W}^0(\tau)}{s}$ lies in $\mathcal{C}^0_b$ and by Lemma \ref{le:flowclose} the change of evaluation point of $W^\mu$ does not change the fact that it lies in $s^\alpha L^2$. Together, this shows that $W(\gamma_{V+W}(\tau))^\mu$ lies in $ s\cdot s^\alpha L^2$ and $s\cdot \mathcal{C}^0_b$. So we get
    $$
    |h^\mu(t)|\leq \int_0^t C|h(\tau)|+s\cdot \frac{\gamma_{V+W}^0(\tau)}{s}\cdot \left|W^\mu(\gamma_{V+W}(\tau))\right|\mathrm{d}\tau.
    $$
    Summing these equations over all $\mu$ and applying Grönwalls Lemma then yields, for $t<2$
    $$
    |h(t)|\leq C'\int_0^ts\cdot \frac{\gamma_{V+W}^0(\tau)}{s}\cdot \sum_\mu |W^\mu(\gamma_{V+W}(\tau))|\mathrm{d}\tau
    $$
    and therefore $|h(t)|$ lies in $s\cdot s^\alpha L^2$ and $s\cdot \mathcal{C}^0_b$ for $t<2$. This proves $|{}^\uparrow\varphi\cdot( {}^\uparrow P\circ\phi_{V+W}- {}^\uparrow P\circ\phi_V)|$ lies in these same function spaces. \\

    Proving that estimates of this type also hold for arbitrary b-derivatives of $h$ (with respect to the starting point $p$), goes by induction. It is in essence analogous to the inductive proof of Lemma \ref{le:flowsobolev}. At each order $k$, we have that $D_b^{q}h(t)$ for $q\leq k-1$ is in the function spaces $s\cdot s^\alpha L^2$ and $s\cdot \mathcal{C}^0_b$ for $t<2$. Also, the derived equation is schematically of the form, 
    $$
    (D_b^{k-1}h)'(t)=A + B\cdot D_b^{k-1}h(t)
    $$
    The term $B\cdot D_b^{k-1}h(t)$ is now given by 
    $$
    \int_0^1 (e_\nu V^\mu)(f(t))\mathrm{d}\xi\cdot D_b^{k-1}h^\nu(t) + \int_0^1V^\mu(f(t)) \mathrm{d}\xi\cdot D_b^{k-1}h^0(t),
    $$
    and is bounded by $C|h(t)|$. The term $A$ lies in $s\cdot s^\alpha L^2$ and $s\cdot \mathcal{C}^0_b$ and is a sum, where each term is schematically one of two types. Either, it has the form
    $$
    (\partial_i^l \partial_0^mW^\mu) \prod_{i=1}^{k}(D_b^{c_i}\gamma_{V+W})^{d_i}
    $$
    with $l+m<k$, $\sum_{i}c_i<k$ and $d=\sum_id_i\geq m+1$, or
    $$
    D_b^{a}h\int_0^1(\partial_i^l\partial_0^mV^\mu)\prod_{i=1}^{k-1}(D_b^{c_i}f)^{d_i}\mathrm{d}\xi
    $$
    with $a<k-1$, $\sum_ic_i<k$, $l+m<k$ and $d=\sum_id_i\geq m$. 
    
    The first type comes from when the $D_b$ derivatives hit the term containing the $W$ in \eqref{eq:ivpVWdiff}, the second from when they hit the term containing the $V$. In terms of the first type, the $d-1$ $D_b^c\gamma_{V+W}$ factors are needed to counteract the $\partial_0^mW$ part. The remaining factor then gives the `free' $s$, and the $W$ part contributes the $s^\alpha L^2$ part in $s \cdot s^\alpha L$. 
    
    In terms of the second type, the $d$ $D_b^cf$ factors are potentially all needed to combine with the $\partial_0^mV$ terms in a way that the product can be estimated by a constant. The function space dependence $s\cdot s^\alpha L^2$ then comes from the $D_b^{a}h$ part, which we assume by induction.

    Now, for the induction step, we apply an additional $D_b$ derivative on the system. It is easy to see that $D_b$ acting on $A+B\cdot D_b^{k-1}h$ is of the same form, with $k-1$ replaced by $k$. This gives the system
$$ 
\begin{aligned}
(D_b^kh)'(t) &= A + B\cdot D_b^kh(t) \\
D_b^kh(0)&=0.
\end{aligned}
$$
After integrating the equation to $t=1$ and applying Grönwalls Lemma like before, we get the result at order $k$. This closes the induction and the proof.
\end{proof}
Now we can characterize the difference of pulling back the metric along two different such diffeomorphism. This is the result we will be using later on. After the previous work, it is essentially only a computation coupled with multiple first and second order Taylor expansions.
\begin{lemma}[Pullback of metric]
    \label{le:pullbackdiff}
    Let $V\in s^\beta H^\infty_b(\Omega;{}^0TM)$, $W\in s^\alpha H^\infty_b(\Omega;{}^0TM)$ with $\alpha\geq \beta >0 $ and $g = g_0 + \tilde{h}\in \mathcal{C}^\infty_b(\Omega;S^2\ {}^0T^*M)+ s^\beta H^\infty_b(\Omega;S^2\ {}^0T^*M)$. Furthermore, define $\omega_W:=2g_0(W,\cdot)\in s^\alpha H^\infty_b(\Omega;{}^0T^*M)$. There exists a $\delta>0$ such that for $\|V\|_{\mathcal{C}^1_b}<\delta$ and $\|W\|_{\mathcal{C}^0_b}<\delta$ we have
    \begin{equation}
    \label{eq:pullback}
        \phi_{V+W}^*g-\phi_V^*g\equiv\delta^*_{g_0}\omega_W\  \mathrm{mod}\ s^{\alpha+\beta}H^\infty_b(\Omega;S^2\ ^0T^*M).   \end{equation}
\end{lemma}
\begin{proof}
    Let $p\in \Omega$, say in $^\uparrow\Omega$, so we work in that coordinate system and write $p=(s,\omega)$. We write $V=V^\mu e_\mu=sV^\mu\partial_\mu$ and analogously for $W$. In this proof, we will write that a function $f$ is of order $s^\gamma $ if $^\uparrow\varphi f\in s^{\gamma}H^\infty_b(\Omega) $ for the partition of unity cutoff $^\uparrow\varphi$.\\
    
    Now, we expand the map $\phi_{V+W}$ using a second order Taylor expansion\footnote{To compute the second derivative, we differentiate the flow equation, the analogue to \eqref{eq:flowequation}  for $V+W$.}
    $$
    \begin{aligned}
    \phi_{V+W}(p)^\mu=&p^\mu + sV^\mu(p)+sW^\mu(p) \\
    &+ \int_0^1 \left(\frac{\mathrm{d}^2}{\mathrm{d}t^2}\phi_{V+W}^t(p)\left|_{t=\tau}\right.\right)^\mu\frac{(1-\tau)^2}{2!}\mathrm{d}\tau\\
    =&p^\mu + sV^\mu(p)+sW^\mu(p) \\
    &+ \int_0^1 \left[(\right.V^\mu + W^\mu)(sV^0+sW^0)+\\
    &\hspace{11mm}(s\partial_\nu(V^\mu+W^\mu))(sV^\nu+sW^\nu\left.)\right]\left|_{\phi^\tau_{V+W}(p)}\right.\frac{(1-\tau)^2}{2!}\mathrm{d}\tau.
    \end{aligned}
    $$
    The first term in the integral comes from when the $\partial_s$ derivative hits the first factor of $s$. We expand the products inside the integral and split them into the terms containing only factors of $V$, and ones with at least a factor of $W$.\\
    
    Consider first the terms in the integral with at least one $W$ factor. There, we use Lemma $\ref{le:flowclose}$ applied to the argument $\phi_{V+W}^\tau(p)$ to see that each of those terms of order $s^{\alpha+\beta}$ (recalling that $\alpha\geq\beta$, and therefore $2\alpha\geq\alpha+\beta$). \\
    
    We turn to the terms in the integral with only $V$ factors. They are evaluated at $\phi_{V+W}^\tau(p)$, to which we apply Lemma \ref{le:flowdif}. This shows $\phi_{V+W}^\tau(p)^\mu=\phi^\tau_V(p)^\mu + sf_W(s,\omega)$ with $f_W$ of order $ s^\alpha$. Let $u$ be one of the terms we wish to study. They are schematically of the form $u=s\cdot V (s\partial)V$ or $s\cdot V^2$. We write 
    \begin{equation}
    \label{eq:funcest}
    u(\phi_{V+W}^\tau(p))= u(\phi_{V}^\tau(p))+\int_0^1\partial_\sigma u(\phi_V^\tau(p)+\xi sf_W(s))(sf_W(s))\mathrm{d}\xi.
    \end{equation}
    The integrand consists of terms of the form $\left((s\partial)^{a} V (s\partial)^b V\right)\cdot s f_W$ with $a+b\leq2$, and therefore the integral is of order $s\cdot s^{2\beta+\alpha}$. This allows us to change the point of evaluation in the $V$-only terms from $\phi_{V+W}^\tau(p)$ to $\phi^\tau_V(p)$, committing an error of the form $s$ times order $s^{2\beta + \alpha}$. But these $V$-only terms, are exactly the second-order Taylor expansion terms of $\phi_V(p)$. So adding them together with $p^\mu+sV^\mu(p)$ we get $\phi_V(p)^\mu$ and therefore we have proven 
    \begin{equation}
    \label{eq:flowsplit}
    \phi_{V+W}(p)^\mu=\phi_V(p)^\mu + sW^\mu(p) + s\cdot \mathcal{O}(s^{\alpha + \beta}) 
    \end{equation}
    
    Now, we begin to compute $\phi_{V+W}^*g$:
    \begin{equation}
    \begin{aligned}
    \label{eq:gpullinter}
        (\phi_{V+W}^*g)_p(\partial_\mu,\partial_\nu)&=g_{\phi_{V+W}(p)}(\mathrm{d}\phi_{V+W}\partial_\mu,\mathrm{d}\phi_{V+W}\partial_\nu)\\
        &=g_{\phi_{V+W}(p)}(\partial_\mu (\phi_{V+W}(p))^\sigma \partial_\sigma,\partial_\nu (\phi_{V+W}(p))^\tau \partial_\tau).
    \end{aligned}
    \end{equation}
    We expand the terms inside the metric using \eqref{eq:flowsplit}
    $$
    \partial_\mu (\phi_{V+W}(p))^\sigma = \partial_\mu \phi_V(p)^\sigma + \partial_\mu (sW^\sigma(p)) + \mathcal{O}(s^{\alpha+\beta}).
    $$
    Notice that we lose one order on the error term because of the $\partial_s=\partial_0$ derivative. We use the bilinearity of $g$ to expand the expression \eqref{eq:gpullinter}. For the cross terms containing $\phi_V(p)$ we split it as $\phi_V(p)^\sigma = p^\sigma + s\cdot \mathcal{O}(s^\beta)$. This gives us 
    \begin{equation}
    \label{eq:gpulliter2}
    \begin{aligned}
        (\phi_{V+W}^*&g)_p(\partial_\mu,\partial_\nu) = g_{\phi_{V+W}(p)}(\partial_\mu (\phi_{V}(p))^\sigma \partial_\sigma,\partial_\nu (\phi_{V}(p))^\tau \partial_\tau)\\
        &+g_{\phi_{V+W}(p)}(\partial_\mu,\partial_\nu(sW^\tau(p))\partial_\tau)+g_{\phi_{V+W}(p)}(\partial_\mu(sW^\sigma(p))\partial_\sigma,\partial_\nu)\\
        &+\sum_{\sigma,\tau}g_{\phi_{V+W}(p)}(\partial_\sigma,\partial_\tau)\mathcal{O}(s^{\alpha+\beta}).
        \end{aligned}
    \end{equation}
    We investigate the second line of the right hand side. First, because those terms contain $W$ components, we can change the point of evaluation of $g$ in this line from $\phi_{V+W}(p)$ to $p$, by committing another error of the form of the third line. (Use Lemma \ref{le:flowdif} for $V'=0, W'=V+W$ and then the analogous argument to equation \eqref{eq:funcest} for $u=g_{\mu\nu}$).  We write 
    $$
    \partial_\mu(sW^\sigma(p)) = \nabla_\mu(sW^\sigma)(p) -\Gamma^\sigma_{\mu\lambda}(p)sW^\lambda(p).
    $$
    (Here, the connection coefficients are with respect to the coordinates and not to the frame $e_\mu$). This gives after a short computation for the second line in \eqref{eq:gpulliter2}
    
    \begin{equation}
    \begin{aligned}
       &(g_p)_{\mu\tau}\nabla_\nu(sW^\tau)(p) + (g_p)_{\sigma\nu}\nabla_\mu(sW^\sigma)(p) - (\Gamma_{\mu\nu\lambda}(p)+\Gamma_{\nu\mu\lambda}(p))sW^\lambda(p)\\
       &=(\delta^*_g 2g(W,\cdot))_{\mu\nu}(p) -sW^\lambda\partial_\lambda g_{\mu\nu}(p).
    \end{aligned}
    \end{equation}
    Now, we turn to the first line of \eqref{eq:gpulliter2}. We write $\phi_{V+W}(p)$ again via equation \eqref{eq:flowsplit} and expand the functions $u:= g_{\sigma\tau}(\phi_{V+W}(p))$ in the analogous fashion to \eqref{eq:funcest} to second order. Furthermore, splitting $\phi_V(p)^\mu = p^\mu + s\cdot \mathcal{O}(s^\beta)$ and expanding the products, we get for the first line of \eqref{eq:gpulliter2}
    \begin{equation}
        (\phi_V^*)_p(\partial_\mu,\partial_\nu) + sW^\lambda\partial_\mu g_{\mu\nu}(p) + \sum_{\sigma\tau}g_{\sigma\tau} \cdot \mathcal{O}(s^{\alpha + \beta}).
    \end{equation}
    Now, putting all together, we finally get 
    \begin{equation}
    \begin{aligned}
        (\phi_{V+W}^*g)_p(\partial_\mu,\partial_\nu) =&(\phi_{V}^*g)_p(\partial_\mu,\partial_\nu) +(\delta^*_g 2g(W,\cdot))_{\mu\nu}(p) \\
        &+\sum_{\sigma\tau}g_{\sigma\tau} (p)\cdot \mathcal{O}(s^{\alpha + \beta}).
        \end{aligned}
    \end{equation}
    Multiplying the equation by $s^2$ changes the error term to just $\mathcal{O}(s^{\alpha + \beta})$ (as $s^2 g_{\sigma\tau}$ are of order $s^0$) and we get the coefficients of the metrics in the frame $e^\mu\otimes_se^\nu$.
    \\
    
    Regarding the second term on the right-hand side, by changing $g $ to $g_0$ in the expression, we make another error of the form $\mathcal{O}(s^{\alpha+\beta})$ because $\tilde{h}\in s^\beta H^\infty_b$. The argument uses the same estimates for the difference of the connection coefficients as in the proof of Proposition \ref{prop:indicialcalc}.
    \\
    
    Now finally, write $\phi_{V+W}^*g = {}^\uparrow\varphi \phi_{V+W}^*g +{}^\downarrow\varphi \phi_{V+W}^*g$, and for each part we perform the calculation above. This closes the proof.
\end{proof}
\subsection{Log-smoothness of the metric}
Now we proceed to construct vector fields such that the metric pulled back along the diffeomorphisms we discussed is log-smooth at $\mathcal{I}^+$.
\begin{proposition}[Improving the gauge]
\label{prop:gaugeup}
Suppose $h_0,\tilde{h},\theta$ are as in Theorem \ref{thm:nonlinsol} with small respective weighted $H^d$- and $H^d_b$-norms. Let $g=g_{\mathrm{dS}}+\chi h_0+\tilde{h}$ and $g_0=g_{\mathrm{dS}}+\chi h_0$, and suppose $\operatorname{Ric}g-\Lambda g=0$ and $\Upsilon(g;g_0)+E_{g_0}(g-g_0) -\tilde{\chi}\theta=0$. Let $\beta^-\in(0,\beta)$ and $\epsilon>0$. Then there exists $V\in s^\beta H^\infty(\Omega;{}^0TM)$ with $s^{\beta^-}\mathcal{C}^1$ norm less than $\epsilon$ so that for $g':=\phi_V^*g$ we have
\begin{equation}
\label{eq:gaugeup}
    \theta':=\Upsilon(g';g_0)+E_{g_0}(g'-g_0)\in s^\infty H_b^\infty(\Omega;{}^0T^*M).
\end{equation} 
\end{proposition}
\begin{proof}
    We can assume that $\beta$ is irrational by reducing it by an arbitrarily small amount. This ensures that $k\beta\notin\mathbb{N}$ for all $k\in\mathbb{N}$ and thus avoids integer coincidences. Let $\delta:=\beta-\beta^-$. We will iteratively construct 
    $$
    \dot{V}_k\in s^{(k+1)\beta} H^\infty_b(\Omega;{}^0TM)
    $$
    for all $k\in \mathbb{N}_0$ with the following properties:
    \begin{enumerate}
    \item The $\dot{V}_k $ are supported in $s\leq s_0/2$.
    \item $\|\dot{V}_k\|_{s^{(k+1)\beta-\delta}H^{d_{n+1}+k}_b(\Omega;{}^0TM)}<\epsilon 2^{-k}$
    \item Setting $V_k:=\sum_{i=0}^{k-1}\dot{V}_i$ and $g_k' := \phi^*_{V_k}g$, we have
    \begin{equation}
        \Upsilon_E(g_k';g_0):=\Upsilon(g_k';g_0)+E_{g_0}(g_k'-g_0)\in s^{(k+1)\beta}H^\infty_b(\Omega;{}^0T^*M).
    \end{equation}
    \end{enumerate}
    Observe that the $V_k$ lie in $s^\beta H^\infty_b$. Suppose now that for $i=0,..,k-1$, the $\dot{V}_i$ have already been constructed. We need to find $\dot{V}_k$ such that 
    $$
    \Upsilon_E(\phi^*_{V_k+\dot{V}_k}g;g_0)\in s^{(k+2)\beta}H^\infty_b.
    $$
    We use equation \eqref{eq:pullback} with $V=V_k$, $W=\dot{V}_k$, with $\alpha = (k+1)\beta$ to rewrite this as
    $$
    \Upsilon_E(\phi^*_{V_k}g+\delta^*_{g_0}\omega_{\dot{V}_k} + h_k;g_0)\in s^{(k+2)\beta}H^\infty_b,
    $$
    where $\delta^*_{g_0}\omega_{\dot{V}_k} \in s^{(k+1)\beta}H^\infty_b$ and $h_k\in s^{(k+2)\beta}H^\infty_b$. By accepting errors of order at least $s^{(k+2)\beta}H^\infty_b$, we can expand $\Upsilon_E$ in the first argument and rewrite this as 
    $$
    \Upsilon_E(g_k';g_0)+\restr{D_1}{g_k'}\Upsilon_E(\delta^*_{g_0}\omega_{\dot{V}_k};g_0 )  \in s^{(k+2)\beta}H^\infty_b.
    $$
    Finally, we can replace the point of linearization $g'_k$ by $g_0$ and accrue another error in $s^{(k+2)\beta}H^\infty_b$. So in order to find our desired $\dot{V}_k$, we need to find a one form $\omega\in s^{(k+1)\beta}H^\infty_b(\Omega;{}^0T^*M)$ that satisfies 
    \begin{equation}
    \label{eq:omegafind}
        \restr{D_1}{g_0}\Upsilon_E(\delta_{g_0}^*\omega;g_0)\equiv-  \Upsilon_E(g_k';g_0)\ \mathrm{mod}\ s^{(k+2)\beta}H^\infty_b.
    \end{equation}
    Given this $\omega$, we can construct $\dot{V}_k$ with the required small norm using the following trick: The cut-off one-form $\chi(s/\epsilon_k)\omega$ fulfills the same equation because $\chi$ is constant near $s=0$. But when $\epsilon_k<1$ is sufficiently small, it will have an arbitrarily small norm in $s^{(k+1)\beta-\delta}H^{d_{n+1}+1+k}_b$. Then, we define $\dot{V}_k$ through $\omega=2g_0(\dot{V}_k,\cdot)$, whose norm is equivalent to that of $\omega$.\\
    
    To solve equation \eqref{eq:omegafind}, we can further replace $\restr{D_1}{g_0}\Upsilon_E(\delta_{g_0}^*\omega;g_0)$ by its indicial operator 
    $$I(\restr{D_1}{g_0}\Upsilon_E(\delta_{g_0}^*\omega;g_0),s\partial_s)=I(-\delta_{g_0}G_{g_0}+E_{g_0},s\partial_s)\circ I(\delta_{g_0}^*,s\partial_s).$$
    We have computed its indicial family, written here as $I(\lambda)$, in \eqref{eq:gaugeind}, which is invertible for $\lambda\neq-1,2,n,n+1$. Notice that $\Upsilon_E(g';g_0)$ is compactly supported in $[0,s_0)$ because the $V_k$ vanish near outside of the support of $\chi$ and therefore $g'=g$ there. So passing to the Mellin-transform side, using the convention and notation from \cite{hintz2024stability}, we set
    $$
\hat{\omega}(\lambda, x):=I(\lambda)^{-1}\left(\mathcal{M}( -\Upsilon_E(g_k';g_0))\right)(\lambda, x), \quad \operatorname{Re} \lambda=(k+1) \beta
$$ 
for $x\in \mathbb{S}^n$. Using the isomorphism induced by the Plancherel Theorem, we then have 
$$
\omega:=\mathcal{M}_{(k+1) \beta}^{-1} \hat{\omega} \in  s^{(k+1) \beta} H_b^{\infty}
$$
which is what we wanted to construct.
\\

Finally, we set $V:=\sum_{i=0}^\infty \dot{V}_i$. By construction, the sum converges in every $s^{\beta}H^N_b$ norm and defines an element of $s^{\beta^-}\mathcal{C}^1_b$ with small norm. Since $\Upsilon_E\left(\phi_{V}^* g ; g_0\right) \equiv \Upsilon_E\left(\phi_{V_k}^* g ; g_0\right)\bmod  s^{(k+1) \beta} H_b^{\infty}$ lies in $ s^{(k+1) \beta} H_{{b}}^{\infty}$ for all $k$, we obtain \eqref{eq:gaugeup}.
    \end{proof}

\begin{remark} The metric constructed in Theorem \ref{thm:solution} fulfills the conditions of this Lemma. So we get $g'$. We can write $g'$ as $g'=g_{\mathrm{dS}}+ \chi g_0 + \tilde{h}'$ with $\tilde{h}'$ small in the same weighted $H^d_b$ norm as $\tilde{h}$. We obviously still have $\operatorname{Ric}g'-\Lambda g'=0$ and therefore have $P(h_0,\tilde{h}',\theta')=\tilde{\delta}^*_{g'}((1-\tilde{\chi})\theta')$, which vanishes near $\mathcal{I}^+$. For metrics of this form, we can show log-smoothness. To avoid cumbersome notation, we will rename $g'=g$, $\tilde{h}'=\tilde{h}$, $\theta'=\theta$ in what follows.
\end{remark}
\begin{lemma}[Log-smoothness in the improved gauge]
\label{le:logsmooth}
    Suppose $h_0,\tilde{h}$ are as in \ref{thm:nonlinsol} with small weighted $H^d_b$-norms. Suppose that for $\theta\in s^\infty H^\infty_b(\Omega;{}^0T^*M)$
    \begin{equation}
        P(h_0,\tilde{h},\theta)=0 \bmod s^\infty H^\infty_b(\Omega;S^2\ ^0T^*M) .
    \end{equation}
    Then $\tilde{h}$ is log-smooth down to $\mathcal{I}^+$. This means that for every $i\in\mathbb{N}$ there exist $m_i\in\mathbb{N}$ and $\tilde{h}_{i,m}\in H^\infty_b(\mathcal{I}^+;S^2\ ^0T^*M)$ such that for all $N\in\mathbb{N}$
    \begin{equation}
        \tilde{h}(s,\omega)-\sum_{i=1}^N\sum_{m=0}^{m_i}s^{i}(\log(s))^m\tilde{h}_{i,m}\in s^N H^\infty_b(\Omega,S^2\ ^0T^*M).
    \end{equation}
    \begin{proof}
        The proof proceeds analogously to \cite{hintz2024stability}. It requires that for $g_0=g_{\mathrm{dS}} + \chi h_0$ we have $\operatorname{Ric}g_0-\Lambda g_0 \in s\mathcal{C}^\infty([0,s_0];H^\infty_b(\mathcal{I}^+))$, which follows from a quick computation. In fact, it is in $s^2\mathcal{C}^\infty$.
    \end{proof}
\end{lemma}
\subsection{Precise expansion at the conformal boundary}
Now that we know that $\tilde{h}$ is log-smooth at $s=0$, the gauge condition serves no further purpose. We now analyze the exact asymptotics using the ungauged Einstein equation. It proceeds similarly to \cite{hintz2024stability} with some added complications in even spatial dimensions.\\

Let $n\geq3$ be an integer. We work in the splitting of $S^2\ ^0TM^*$ obtained by combining the splittings \eqref{eq:splitting1}\eqref{eq:splitting2}
\begin{equation}
    \label{eq:splitting3}
    S^{2}\ ^0T^* M=\mathbb{R} \frac{\mathrm{d} s^2}{s^2} \oplus\left(2 \frac{\mathrm{~d} s}{s} \otimes_s s^{-1} T^* \mathbb{S}^n\right) \oplus \mathbb{R} s^{-2} g_{(0)} \oplus s^{-2} \operatorname{ker} \operatorname{tr}_{g_{(0)}}.
\end{equation}
We use the expressions of \eqref{eq:indG}\eqref{eq:indR}\eqref{eq:indwave}\eqref{eq:inddelta}\eqref{eq:inddeltastar} to derive
\begin{equation}
\label{eq:Idric}
\begin{aligned}
2 I(D \operatorname{Ric}-\Lambda, \lambda) & :=I\left(\square_{g_0}-2 \delta_{g_0}^* \delta_{g_0} \mathrm{G}_{g_0}+2 \mathscr{R}_{g_0}-2 \Lambda, \lambda\right) \\
& =\left(\begin{array}{cccc}
n (\lambda-2) & 0 & - n\lambda( \lambda-2) & 0 \\
0 & 0 & 0 & 0 \\
-\lambda+2n & 0 & \lambda(\lambda-2n) & 0 \\
0 & 0 & 0 & \lambda(\lambda-n )
\end{array}\right),
\end{aligned}
\end{equation}
remembering $\Lambda=n$. We study the relationship of its kernel with the range of 
\begin{equation}
    \label{eq:Id*}
    I\left(\delta^*, \lambda\right):=I\left(\delta_{g_0}^*, \lambda\right)=\left(\begin{array}{cc}
\lambda & 0 \\
0 & \frac{1}{2}(\lambda+1) \\
1 & 0\\
0&0
\end{array}\right).
\end{equation}
Further, we will investigate the relationship of the range of $I(D\operatorname{Ric}-\Lambda,\lambda)$ with the kernel of \begin{equation}
\begin{aligned}
2 I(\delta \mathrm{G}, \lambda):=&2 I\left(\delta_{g_0} \mathrm{G}_{g_0}, \lambda\right)\\
=&\left(\begin{array}{cccc}
\lambda-2n & 0 & n (\lambda-2) & 0 \\
0 &  2(\lambda-(n+1)) & 0 & 0
\end{array}\right).
\end{aligned}
\end{equation}
These operators depend on $\omega\in\mathcal{I}^+$ only through the splitting, i.e. through $g_{(0)}$. The diffeomorphism invariance of the Einstein equation gives for the indicial operators $I(D\operatorname{Ric}-\Lambda,\lambda)\circ I(\delta^*,\lambda)=0$. Furthermore, the linearized second Bianchi identity gives $I(\delta G,\lambda)\circ I(D\operatorname{Ric}-\Lambda,\lambda)$. Both facts can be explicitly checked by direct computation here.\\

We proceed to analyze these $\lambda$ dependent matrices as linear maps $\mathbb{R}^4\rightarrow\mathbb{R}^4$, $\mathbb{R}^2\rightarrow\mathbb{R}^4$ or $\mathbb{R}^4\rightarrow\mathbb{R}^2$. 

\begin{lemma} [Kernel of linearized Einstein modulo pure gauge]
    Let $\lambda >0$ and let $h$ lie in $\operatorname{ker}I(D\operatorname{Ric}-\Lambda,\lambda)$.\\
    (1) If $\lambda \neq n$, then $h\in \operatorname{ran}(I(\delta^*,\lambda))$.\\
    (2) If $\lambda=n$ and $\partial_\lambda I(D\operatorname{Ric}-\Lambda,\lambda)h\in\operatorname{ran}I(D\operatorname{Ric}-\Lambda,\lambda)$, then $h\in \operatorname{ran}(I(\delta^*,\lambda))$. And furthermore  
    \begin{equation}
    \label{eq:kermodrange}
    \operatorname{ker}(I(D\operatorname{Ric}-\Lambda,\lambda))/\operatorname{ran}(I(\delta^*,\lambda)) = \operatorname{span}\{(0,0,0,1)\}.
    \end{equation}.
    \label{le:410}
\end{lemma}
\begin{proof}
    For $\lambda>0$ and $\lambda\neq n$, we can compute that the kernel of $I(D\operatorname{Ric}-\Lambda,\lambda)$ is given by $\operatorname{span}\{(\lambda,0,1,0),(0,1,0,0)\}$. (Recall that we fixed $n\geq3$). This is exactly the range of $I(\delta^*,\lambda)$ for $\lambda\neq-1$.\\
    
    For $\lambda=n$, both the fourth row and fourth column of $I(D\operatorname{Ric}-\Lambda,\lambda)$ vanish. So we also find that $(0,0,0,1)$ lies in the kernel of $I(D\operatorname{Ric}-\Lambda,\lambda)$. Comparing to the range of $I(\delta^*,\lambda)$, we immediately get $\operatorname{ker}(I(D\operatorname{Ric}-\Lambda,\lambda))/\operatorname{ran}(I(\delta^*,\lambda)) = \operatorname{span}\{(0,0,0,1)\}$. We calculate
    \begin{equation}
    \label{eq:Iderivative}
        2 \partial_\lambda I(D \operatorname{Ric}-\Lambda, \lambda)=\left(\begin{array}{cccc}
n & 0 & -2n \lambda+2n & 0 \\
0 & 0 & 0 & 0 \\
-1 & 0 & 2 \lambda-n & 0 \\
0 & 0 & 0 & 2 \lambda-n
\end{array}\right) ,
    \end{equation}
    and see that $\partial_\lambda I(D \operatorname{Ric}-\Lambda, \lambda)(0,0,0,1)^T$ has non vanishing 4-component. Therefore, it cannot be in the range of $I(D\operatorname{Ric}-\Lambda,\lambda)$.
\end{proof}
Let $h\in\mathbb{R}^4$ and consider the function $s^\lambda h$. Then we have 
$$
I(D\operatorname{Ric}-\Lambda,s\partial_s)s^\lambda h=s^\lambda I(D\operatorname{Ric}-\Lambda,\lambda)h,$$
and thus
$$
    \begin{aligned}
    I(D\operatorname{Ric}-\Lambda,s\partial_s)&(s^\lambda \log(s)^mh)=\partial_\lambda^m(s^\lambda I(D\operatorname{Ric}-\Lambda,\lambda)h)\\
    &=\sum_{m'=0}^m \binom{m}{m'}s^\lambda \log(s)^{m'}\partial_\lambda^{m-m'}I(D\operatorname{Ric}-\Lambda,\lambda)h.
    \end{aligned}
$$
Comparing orders of $\log(s)$ we get the following implication for $k\geq 2$ and $f\in\mathbb{R}^4$
\begin{equation}
\label{eq:logind}
    \begin{aligned}
    I(D&\operatorname{Ric}-\Lambda,s\partial_s)\sum_{m=0}^k s^\lambda\log(s)^m h_m= s^\lambda f\\
    &\Rightarrow I(D\operatorname{Ric}-\Lambda,\lambda)h_k =0,\partial_\lambda I(D\operatorname{Ric}-\Lambda,\lambda)h_k\in\operatorname{ran}I(D\operatorname{Ric}-\Lambda,\lambda).
    \end{aligned}
\end{equation}
If $k=1$, then at least $I(D\operatorname{Ric}-\Lambda,\lambda)h_k =0$ still holds. If $f=0$, the full implication holds for $k=1$ and only the first part for $k=0$.\\

Now, using this implication, we can solve away the log terms of a certain order by gauge transformations if they lie in the kernel of $I(D\operatorname{Ric}-\Lambda,s\partial_s)$, with a complication for $\lambda=n$. This is the content of the following Lemma.
\begin{lemma}[Solving order $\lambda$ using gauge freedom]
\label{le:412}
    Let $\lambda>0$ and  $h_0,...,h_k\in\mathbb{R}^4$ for $k\in\mathbb{N}_0$. Define $h(s):=\sum_{m=0}^ks^\lambda \log(s)^mh_m$. Suppose $I(D\operatorname{Ric}-\Lambda,s\partial_s)h=0$. Then, there exist $\omega_0,...,\omega_k\in\mathbb{R}^2$ such that for $\omega(s):=\sum_{m=0}^ks^\lambda\log(s)^m\omega_m$ we have\\
    (1) In the case $\lambda \neq n$: $h+I(\delta^*,s\partial_s)\omega=0$\\
    (2) In the case $\lambda = n$: $h+I(\delta^*,s\partial_s)\omega= s^\lambda (0,0,0,\alpha)$ for $\alpha$ independent of $s$.
\end{lemma}
\begin{proof}
    By \eqref{eq:logind}, we have $h_k\in\operatorname{ker}I(D\operatorname{Ric}-\Lambda,\lambda)$. So, when $\lambda\neq n$, we can apply Lemma \ref{le:410} and get $\omega_k$ such that $h_k+I(\delta^*,\lambda)\omega_k$=0. Therefore we have
    \begin{equation}
    \begin{aligned}
        h + I(\delta^*,s\partial_s)(s^\lambda \log(s)^k\omega_k) =s^\lambda\log(s)^k(\underbrace{h_k+I(\delta^*,\lambda)\omega_k)}_{=0} \\
        +\sum_{m=0}^{k-1}s^\lambda \log(s)^mh_m + s^\lambda\log(s)^{k-1}\partial_\lambda I(\delta^*,\lambda)\omega_k,
    \end{aligned}
    \end{equation}
    and the right hand side is of the same form as $h$ but with the highest order of log reduced by one. An iterative argument finishes the proof in this case.\\
    
    Turning now to $\lambda=n$ we use the full implication \eqref{eq:logind} together with Lemma \ref{le:410} for $m\geq1$. Using the same argument as above, we get $\omega_1,...,\omega_k$ such that 
    $$
    h+I(\delta^*,s\partial_s)\sum_{m=1}^ks^\lambda\log(s)^m\omega_m =: s^\lambda f
    $$
    and $f$ is independent of $s$. Now, with only the first part of the implication \eqref{eq:logind} being true, we use \eqref{eq:kermodrange} to get $\omega_0$ such that $f = -I(\delta^*,\lambda)\omega_0 + (0,0,0,\alpha)$.
\end{proof}
When dealing with terms of order $n$ in the even-dimensional case, we will need a slightly more complicated version of this Lemma.
\begin{lemma}[Solving order $n$ even using gauge freedom]
    \label{le:412even}
    Let $h_0,...,h_k\in \mathbb{R}^4$ for $k\in \mathbb{N}_0$. Let $f\in \mathbb{R}^4$ with 4-component $=\alpha$. Suppose $I(D\operatorname{Ric}-\Lambda,s\partial_s)h=s^nf$. Then there exist $\omega_0,...,\omega_k\in\mathbb{R}^2$ and a $\gamma$ independent of $s$, such that for $\omega(s):=\sum_{m=0}^ks^n\log(s)^m\omega_m$ we have
    \begin{equation}
    \label{eq:ordernevengauge}
        h+I(\delta^*,s\partial_s)\omega = s^n\left(\log(s)(0,0,0,\frac{2\alpha}{n}) +(0,0,0,\gamma)\right).
    \end{equation}
\end{lemma}
\begin{proof}
    For $k\geq 2$ we use the same argument as in the previous proof. The implication \eqref{eq:logind} and Lemma \ref{le:410} give $\omega_2,...,\omega_k$ such that $$
    h+I(\delta^*,s\partial_s)\sum_{m=2}^ks^n\log(s)^m\omega_m= s^n(\log(s)h_1'+h_0')
    $$ with $h_1',h_0'\in \mathbb{R}^4$.
    Now we have
    $$
    I(D\operatorname{Ric}-\Lambda,s\partial_s)\left(s^n(\log(s)h_1'+h_0')\right) = s^nf.
    $$
    For the log term of this equation, we still get the first half of the implication \eqref{eq:logind}, namely $I(D\operatorname{Ric}-\Lambda,n)h_1=0$. By Lemma \ref{le:410}, we get $\omega_1$ and $\eta\in \mathbb{R}$ such that $h_1+I(\delta^*,n)\omega_1 =(0,0,0,\eta)$. Updating the gauge one form by this, we get
    \begin{equation}
    \begin{aligned}
     s^n(\log(s)h_1'+h_0') + &I(\delta^*,s\partial_s)\left(s^n\log(s)w_1\right) \\
     =&s^n\log(s)(0,0,0,\eta) + s^n\underbrace{\left(\partial_\lambda I(\delta^*,\lambda)\left|_{\lambda=n}\right.\omega_1+h_0'\right)}_{:=h_0''}.
     \end{aligned}
    \end{equation}
    Applying $I(D\operatorname{Ric}-\Lambda,s\partial_s)$ to this term gives
    \begin{equation}
    \label{eq:lastupdate}
    \begin{aligned}
        I(D\operatorname{Ric}&-\Lambda,s\partial_s)\left(s^n\log(s)(0,0,0,\eta)+s^nh_0''\right)\stackrel{!}{=}s^nf\\
        &=s^n\partial_\lambda I(D\operatorname{Ric}-\Lambda,\lambda)\left|_{\lambda=n}\right.(0,0,0,\eta) + s^n(I(D\operatorname{Ric}-\Lambda,n)h_0''.
    \end{aligned}
    \end{equation}
    Now, since the last row of $I(D\operatorname{Ric}-\Lambda,n)$ vanishes identically, and because the last column of $\partial_\lambda I(D\operatorname{Ric}-\Lambda,\lambda)\left|_{\lambda=n}\right.$ is equal to $(0,0,0,n/2)$, we get $\eta = 2\alpha/n$. From \eqref{eq:lastupdate}, we get that $h_0'' -(f-(0,0,0,\alpha))\in \operatorname{ker}I(D\operatorname{Ric}-\Lambda,n)$, and therefore equal to $-I(\delta^*,n)\omega_0 + (0,0,0,\gamma)$ for a $\omega_0\in \mathbb{R}^2$ by Lemma \ref{le:410}. Putting everything together, we get \eqref{eq:ordernevengauge}.
\end{proof}
Now we turn to the range of $I(D\operatorname{Ric}-\Lambda,\lambda)$, which is necessarily contained in $\operatorname{ker}I(\delta G,\lambda)$. We consider the cases $\lambda>0$ and $n>2$.
\begin{lemma}[Solvability of linearized Einstein]
\label{le:413}
    Let $\lambda>0$ and $f\in\mathbb{R}^4$. Suppose $I(\delta G,\lambda)f=0$.\\
    (1) If $\lambda\neq n,n+1$ then there exists $h=(0,0,h_3,h_4)$ with $f=I(D\operatorname{Ric}-\Lambda,\lambda)h$.\\
    (2) If $\lambda=n+1$ and $f$ is of the form $(f_1,0,f_3,f_4)$, then there exists $h=(0,0,h_3,h_4)$ with $f=I(D\operatorname{Ric}-\Lambda,\lambda)h$.
\end{lemma}
\begin{proof}
    The kernel of $I(\delta G,\lambda)$ is given by $\operatorname{span}\{(n(\lambda-2),0,-\lambda+2n,0),(0,0,0,1)\}$ for $\lambda\neq n+1$. If $\lambda =n+1$, it is given by the direct sum of the same two vectors and $\operatorname{span}(0,1,0,0)$. But since we excluded vectors with a 2-component in the case of $\lambda=n+1$, the Lemma follows after noticing that the basis of the kernel is spanned by the last two columns of $I(D\operatorname{Ric}-\Lambda,\lambda)$ for $\lambda\neq n$.
\end{proof}
Note that the Lemmas proved in this section remain true if we put back the parameter dependence on $x\in\mathbb{S}^n$; e.g, in Lemma \ref{le:410} if $h$ depends on $x$ in a $\mathcal{C}^\infty$ fashion, we find $\omega$ with the same parameter dependence that satisfies $I(\delta^*,\lambda)\omega=h$.\\

We record the following Lemma about the structure of the Einstein operator.
\begin{lemma}[Log order and parity of the Einstein operator]
\label{le:paritylog}
Let $l\in\mathbb{N}_0$ and let a tensor $g\in H^\infty_b(\Omega;S^2\ ^0T^*M)$ be of the form $g= \sum_{m=0}^l \log(s)^mg_m$, where $g_0\in -\frac{\mathrm{d}s^2}{s^2}+H^\infty(\mathcal{I}^+;s^{-2}S^2T^*\mathbb{S}^n)$ and the spatial part is positive definite. Furthermore, assume that for $m>0$, the $g_m$ are finite sums of elements of $s^{nm+r}\cdot H^\infty(\mathcal{I}^+;s^{-2}S^2T^*\mathbb{S}^n)$ for (potentially different) $r\in\mathbb{N}_0$. We regard $g$ as a function from $[0,s_0]$ to $H^\infty_b(\mathcal{I}^+;s^{-2}T^*\mathbb{S}^n)$ and consider the series expansion around $s=0$. We write $\tilde{\mathcal{O}}(s^\alpha)$ for log-smooth terms with almost $s^\alpha$ decay at $s=0$. Suppose
\begin{equation}
    \operatorname{Ric}g-\Lambda g= 0 + \tilde{\mathcal{O}}( s^{nl +k}) .
\end{equation}
for a $k\in \{0,..,n-1\}$. Then we get
\begin{equation}
    \label{eq:ferror}
    \operatorname{Ric}g-\Lambda g= s^{nl+k}\sum_{q=0}^l f^{(q)}\log(s)^q+\tilde{\mathcal{O}}(s^{nl+k+1})
\end{equation}
with $f^{(q)}\in H^\infty_b(\mathcal{I}^+;s^{-2}S^2T^*\mathbb{S}^n)$. Furthermore, if all $g_l$ contain only even powers of $s$ (i.e. $r$ is even in all terms) we have the additional implication. If $nl+k$ is odd, then $f^{(q)}=(0,f^{(q)}_2,0,0)$, and if $nl+k$ is even, $f^{(q)}=(f^{(q)}_1,0,f^{(q)}_3,f^{(q)}_4)$ in the splitting \eqref{eq:splitting3}, for all $q$.
\end{lemma}
\begin{proof}
    Using a partition of unity, we can work in a coordinate system on $\Omega$, say $^\uparrow\Omega$, where we use the frame $e_\mu$ from \eqref{eq:framedefinition}. In this proof, we write that a term is \textit{almost-even} or \textit{almost-odd} in $s$ if it is a sum of even or odd terms multiplied by powers of $\log(s)$. By construction, $g_{0i}$ are 0, which we will consider as almost-odd. If additionally, the $g_l$ are all even, we see that $g_{00}$ and $g_{ij}$ are almost-even. So they fulfill what we will call the parity rule: If the number of indices $>0$ in an expression is even, then the expression is almost-even. If it is instead odd, the expression is almost-odd. Notice, how $e_0$ preserves parity in $s$, whereas the $e_i$ reverse parity. So by using \eqref{eq:connectiondef} we see that the $\Gamma_{\lambda\mu\nu}$ satisfy the parity rule if it is fulfilled by $g$. Furthermore, the lowest order at which $\log(s)^{m+1}$ appears is $s^{n(m+1)}$. Next, by writing the series expansion of $g^{-1}$ around $g_0^{-1}$, we see that because of the block diagonal structure, the $g^{\mu\nu}$ satisfy the parity rule and therefore the $\Gamma^\lambda_{\mu\nu}$ do as well, if $g$ fulfills it. Again, the lowest order with $\log(s)^{m+1}$ terms is $s^{n(m+1)}$. Finally, writing 
\begin{equation}
\label{eq:riccicoords}
\begin{aligned}
\operatorname{Ric}_{00}&=e_\lambda \Gamma_{00}^\lambda-e_0 \Gamma_{\lambda 0}^\lambda+\Gamma_{\lambda \rho}^\lambda \Gamma_{00}^\rho-\Gamma_{0 \rho}^\lambda \Gamma_{\lambda 0}^\rho +\Gamma^l_{l0}\\
\operatorname{Ric}_{0j}&=e_\lambda \Gamma_{j0}^\lambda-e_j \Gamma_{\lambda 0}^\lambda+\Gamma_{\lambda \rho}^\lambda \Gamma_{j0}^\rho-\Gamma_{j \rho}^\lambda \Gamma_{\lambda 0}^\rho -\Gamma^0_{j0}\\
\operatorname{Ric}_{ij}&=e_\lambda \Gamma_{ji}^\lambda-e_j \Gamma_{\lambda i}^\lambda+\Gamma_{\lambda \rho}^\lambda \Gamma_{ji}^\rho-\Gamma_{j\rho}^\lambda \Gamma_{\lambda i}^\rho -\Gamma^0_{ji},
\end{aligned}
\end{equation}
we see that the parity rule holds also for $\operatorname{Ric}g-\Lambda g$ if it holds for $g$ (The last terms on the right hand sides arise because $e_0$ and $e_i$ do not commute). And again, the lowest order, where $\log(s)^{m+1}$ appears, is $s^{n(m+1)}$. This proves \eqref{eq:ferror}. Looking at $\operatorname{Ric}g-\Lambda g$ in the splitting \eqref{eq:splitting3}, we see that because of the parity rule, $(\operatorname{Ric}g-\Lambda g)_a$ is almost-even for $a=1,3,4$ and almost-odd for $a=2$. As the leading order term on the right hand side of \eqref{eq:ferror} has parity determined by whether $nm+k$ is even or odd, we get that the required components of the error terms $f^{(q)}$ vanish.
\end{proof}
As a final preparation for the expansion, we record the following
\begin{lemma}
    \label{le:einstein0order}
    Let $h_0\in H^\infty(\mathcal{I}^+;s^{-2}S^2T^*\mathbb{S}^n)$ with small $\mathcal{C}^2_b$ norm. Define the metric $g_0':=\underline{g} + \chi h_0$. Then in a series expansion around $s=0$, we have
    \begin{equation}
        \operatorname{Ric}g_0' -\Lambda g_0'= \mathcal{O}(s^2),
    \end{equation}
    where we write $\mathcal{O}(s^\alpha)$ for terms that decay as $s^\alpha$ and have $H^\infty(\mathcal{I}^+;s^{-2}S^2T^*\mathbb{S}^n)$ behaviour in $\omega$.
\end{lemma}
\begin{proof}
    This follows from a short computation in the frame $e_\mu$.
\end{proof}
Now we turn to the precise expansion at the conformal boundary. Notice how the solution produced by Proposition \ref{prop:gaugeup} applied to the metric from Theorem \ref{thm:solution} fulfills the conditions of Lemma \ref{le:logsmooth}. Therefore, it fulfills the assumptions stated in the following Theorem.
\begin{theorem}[Expansion at the conformal boundary]
\label{thm:expansion} Let $n\geq 3$, $h_0\in H^\infty(\mathcal{I}^+;s^{-2}S^2T^*\mathbb{S}^n)$ with small $H^d$-norm and $\tilde{h}\in s^\beta H^\infty_b(\Omega;S^2\ ^0T^*M)$ for an arbitrary $\beta\in(0,1)$ that is log-smooth at $\mathcal{I}^+$. Suppose that $g:=g_{\mathrm{dS}}+\chi h_0+\tilde{h}$ satisfies
\begin{equation}
\label{eq:Einsteineq}
    \operatorname{Ric}g -\Lambda g=0
\end{equation}
on $\Omega$. Then there exists a diffeomorphism on a collar neighborhood of $\mathcal{I}^+$, which restricts to the identity on $\mathcal{I}^+$, such that the pullback metric $\phi^* g$ has the block diagonal form
\begin{equation}
 \label{eq:blockdiag}
    \phi^*g=\frac{-\mathrm{d}s^2+H(s,\omega;\mathrm{d\omega)}}{s^2 },
\end{equation}
for a one parameter family in $s$ of positive definite metrics $H(s,\omega;\mathrm{d}\omega)$ on $\mathbb{S}^n$. Furthermore, the pulled back metric has the following expansion at $\mathcal{I}^+$. 
\begin{enumerate}
\item If the dimension $n$ is odd:\\
    The conformally rescaled metric $s^2\phi^*g$ is smooth down to $\mathcal{I}^+$. More precisely there exist $\tilde{h}_i\in H^\infty(\mathcal{I}^+;s^{-2}S^2T^*\mathbb{S}^n)$ for $ i \in \mathbb{N}_{\geq2}$ such that for all $N\in \mathbb{N}$ we have
    \begin{equation}
    \label{eq:expansionodd}
        (\phi^* g)(s,\omega)-\left((\underline{g}+h_0)(s,\omega) +\sum_{i=2}^Ns^{i}\tilde{h}_i(\omega)\right)\in s^{N}H^\infty_b(\Omega;S^2\ ^0T^*M).
    \end{equation}
    Furthermore, for $i<n$, $\tilde{h}_i$ is computable using only $h_0$. In fact, for odd $i<n$ $\tilde{h}_i=0$. So the first odd order that appears is $\tilde{h}_n$, which is a transverse-traceless tensor in the following sense: Let $g_{(n)}:=s^2\tilde{h}_n$ and $g_{(0)}:= s^2(g_{\mathrm{dS}} +h_0)\left|_{s=0}\right.$, then $\operatorname{tr}_{g_{(0)}}g_{(n)}=0$ and $\delta_{g_{(0)}}g_{(n)}=0$. All further $\tilde{h}_i$ can then be computed using $h_0,\tilde{h}_n$.
\item If the dimension $n$ is even:\\
    The conformally rescaled metric $s^2\phi^*g$ is smooth down to $\mathcal{I}^+$ if and only if the obstruction tensor of $\underline{g}+h_0$ vanishes. More precisely, for all $i\in \mathbb{N}_{\geq2}$ and $m\leq \lfloor i/n\rfloor$ there exist $\tilde{h}_{i}^m\in H^\infty(\mathcal{I}^+;s^{-2}S^2T^*\mathbb{S}^n)$ such that for all $N\in\mathbb{N}$
    \begin{equation}
    \label{eq:expansioneven}
        (\phi^* g)(s,\omega)- \left((\underline{g}+h_0)(s,\omega) +\sum_{i=2}^N\sum_{m=0}^{\lfloor i/n\rfloor} s^{i}\log(s)^m \tilde{h}_i^m\right)
    \end{equation}
    lies in $s^{N}H^\infty_b(\Omega;S^2\ ^0T^*M)$. Furthermore, $\tilde{h}_i^m=0$ if $i$ is odd. For $i<n$ the $\tilde{h}_i^m=\tilde{h}_i^0$ can be computed with only $h_0$. $\tilde{h}_n^0$ is a traceless tensor and its divergence is a certain one form computable from $h_0$. All further $\tilde{h}_i^m$ are then determined by $h_0,\tilde{h}_i^0$. The first log-order that appears is $\tilde{h}_n^1$, which vanishes if and only if the obstruction tensor of $\underline{g}+h_0$ vanishes. If it does, all further $\tilde{h}_i^m$ for $m>0$ vanish.
\end{enumerate}
\end{theorem}
\begin{remark}
    (Expansion at $\mathcal{I}^-$) The metric $g$ produced in Theorem \ref{thm:solution} is of the form that we can apply all results of this section in the coordinate $\tilde{s}=\pi-s$. Therefore, this expansion can also be done near $\mathcal{I}^-$, with the exact same proof.
\end{remark} 
\begin{proof}
    The proof proceeds in two parts. First, we prove the expansion of equations \eqref{eq:expansionodd}\eqref{eq:expansioneven}. This uses the structure of the indicial operators involved, and the pullbacks constructed using vector fields discussed in the Lemmas in the section leading up to this proof, section \ref{se:expansion}. This already proves the block diagonal structure of \eqref{eq:blockdiag} modulo terms of order $s^\infty$.\\
    
    In part two, we show that these terms of order $s^\infty$ can be made to vanish as well, using a further pullback, proving the \emph{full} metric is block-diagonal. This is done by constructing boundary normal coordinates. By inspecting the proof of the construction of the boundary normal coordinates carefully, we can show that this last pullback does not change the structure of the expansion already constructed beforehand. \\

    \textbf{Part one: Expansion}\\
    We prove the expansion of equations \eqref{eq:expansionodd}\eqref{eq:expansioneven}. The proof proceeds identically for $n$ even or odd until order $n$. For the series expansion in $s$ around $s=0$ we write $\tilde{\mathcal{O}}(s^\alpha)$ for log-smooth terms in $s$ with almost $s^\alpha$ decay compared to $\frac{-\mathrm{d}s}{s^2}$ and $H^\infty(\mathcal{I}^+;s^{-2}S^2T^*\mathbb{S}^n)$ behavior in $\omega\in\mathbb{S}^n$.\\
    
    \underline{1. Elimintating $s^1$ terms.} Set $g_0':=\underline{g}+\chi h_0$ and $\tilde{h}'=\tilde{h} + (g_{\mathrm{dS}}-\underline{g})\in s^\beta H^\infty_b$ log-smooth at $\mathcal{I}^+$. We write 
    \begin{equation}
    \label{eq:linearized einstein}
        0=\operatorname{Ric}(g_0' + \tilde{h}')-\Lambda(g_0' + \tilde{h}') = \operatorname{Ric}g_0' -\Lambda g_0' + \int_0^1 D_{g_0'+\tau\tilde{h}'}\operatorname{Ric}\tilde{h}' -\Lambda \tilde{h}'\mathrm{d}\tau.
    \end{equation}
    We investigate this equation at order $\tilde{\mathcal{O}}(s)$. Because of Lemma \ref{le:einstein0order} the Einstein operator applied to $g_0'$ vanishes at that order. In view of Lemma \ref{le:logsmooth}, we can replace the operator in the integral by its indicial operator $I(D\operatorname{Ric}-\Lambda,s\partial_s)$ and get
    $$
    I(D\operatorname{Ric}-\Lambda,s\partial_s)\sum_{m=0}^{m_1}s\log(s)^m \tilde{h}_{1,m}=0.
    $$
    Now, applying Lemma \ref{le:412} with $H^\infty$ dependence on $\omega\in\mathcal{I}^+$, we get for $m=0,...,m_1$ $\omega_{1,m}\in H^\infty(\mathcal{I}^+;s^{-1}T^*\mathbb{S}^n)$ such that for $\omega_1:=\sum_{m=0}^{m_1}s\log(s)^m\omega_{1,m}$
    \begin{equation}
    \label{eq:blabla}
    \sum_{m=0}^{m_1}s\log(s)^m \tilde{h}_{1,m} + I(\delta^*,s\partial_s)\omega_1=0.
    \end{equation}
    We multiply $\omega_1$ by cutoffs $\chi(s/\epsilon)$ for arbitrary $\epsilon$ then, near $\mathcal{I}^+$, \eqref{eq:blabla} is still fulfilled. By choosing $\epsilon$ sufficiently small, we can make its $s^{1^-}H^{d_{n+1}+1}_b(\Omega;{}^0T^*M)$ arbitrarily small, for $1^-:=1-\delta$ with $\delta>0$ arbitrarily small. We call this cut-off version still $\omega_1$. We define $V_1$ via $\omega_1=2g_0(V_1,\cdot)$ and get for the pullback $\phi^*_{V_1} g$
    $$
    \phi^*_{V_1} g = \underbrace{g_0'}_{:=g_1'}+ \underbrace{\sum_{m=0}^{m_1}s\log(s)^m \tilde{h}_{1,m}+I(\delta^*,s\partial_s)\omega}_{=0 +\mathcal{O}(s^\infty)} + \tilde{\mathcal{O}}(s^2),
    $$
    where we used Lemma \ref{le:pullbackdiff}. 
    \\
    
    \underline{2. Terms of order less than $n$}
    Assume inductively that for $i\in\mathbb{N}_0$ $i\leq n-1$ that $V_{i-1}$ has been constructed, such that 
    $$
    \phi^*_{V_{i-1}} g = \underbrace{g_{0}' +\sum_{j=2}^{i-1}s^{j}\tilde{h}_{j}}_{:=g_{i-1}'} + \tilde{\mathcal{O}}(s^{i})
    $$
    and the $\tilde{h}_j=0$ if $j$ is odd. Assume further that 
    $$
    \operatorname{Ric}g_{i-1}'-\Lambda g_{i-1}'\equiv0 \bmod \tilde{\mathcal{O}}(s^{i}). 
    $$
    Write equation \eqref{eq:linearized einstein} with $g_0'$ replaced with $g_{i-1}'$ and $\tilde{h}'$ replaced by $\phi^*_{V_{i-1}} g-g_{i-1'}$. Notice, how $g_{i-1}'$ is of the form such that we can apply Lemma \ref{le:paritylog} (for $m=0$, $k= i$) to get 
    \begin{equation}
    \label{eq:iorder}
    I(D\operatorname{Ric}-\Lambda,s\partial_s)\sum_{m=0}^{m_i}s^{i}\log(s)^m \tilde{h}_{i,m_i} = s^{i}f+\tilde{\mathcal{O}}(s^{i+1})
    \end{equation}
    for an $f\in H^\infty_b(\mathcal{I}^+;s^{-2}T^*\mathbb{S}^n)$. Because of the linearized second Bianchi identity $I(\delta G,s\partial_s)(s^{i}f)=s^{i}I(\delta G,i)f=0$. Lemma \ref{le:413} then gives 
    $$
    \tilde{h}_i \in H^\infty(\mathcal{I}^+;s^{-2}S^2T^*\mathbb{S}^n),
    $$
    such that $I(D\operatorname{Ric}-\Lambda,i)\tilde{h}_i=f$.\footnote{In the $n=$ even dimensional case, these $\tilde{h}_i$ are called $\tilde{h}_i^0$ in the statement of the Theorem.} By Lemma \ref{le:412} we have 
    $$
    \sum_{m=0}^{m_i}s^{i}\log(s)^m \tilde{h}_{i,m_i} - s^{i}\tilde{h}_i\in \operatorname{ker}I(D\operatorname{Ric}-\Lambda,s\partial_s)
    $$
    can be expressed as $-I(\delta^*,s\partial_s)\sum_{m=0}^{m_i}s^{i}\log(s)^m\omega_{i,m}:= -I(\delta^*,s\partial_s)\omega_i$ for one forms $\omega_{i,m}$ as above. Again, we can cut off $\omega_i$ like before such that it has $s^{1^-}H^{d_n+i}_b$ norm smaller than $1/2^{i}$, which later gives us an error $\mathcal{O}(s^\infty)$. Now, defining $\dot{V_i}$ via $\omega_i=2g_0(\dot{V_i},\cdot)$ and $V_i := V_{i-1}+\dot{V_i}$ we have
    $$
    \phi_{V_i}^* g= \phi^*_{V_{i-1}} + \underbrace{\sum_{m=0}^{m_i}s^{i}\log(s)^m \tilde{h}_{i,m_i}  +I(\delta^*,s\partial_s)\omega_i}_{s^{i}\tilde{h}_i + \mathcal{O}(s^\infty)} +\tilde{\mathcal{O}}(s^{i+1}).
    $$
    Now, if $i$ is odd, then $f$ was of the form $(0,f_2,0,0)$ by Lemma \ref{le:paritylog}. But because the second row of $I(D\operatorname{Ric}-\Lambda,\lambda)$ vanishes for all $\lambda$, equation \eqref{eq:iorder} implies that indeed $f_2=0$ and therefore $f=0$. This then allows us to choose $\tilde{h}_i=0$, so no term of odd order to the expansion is added. This closes the induction and we get the desired expansion up to order $n-1$.\\
    
    \underline{3. The case $n$ odd} \\
    \textit{3.1 Order $n$} We have constructed $V_{n-1}$ such that 
    $$
    \phi^*_V g = \underbrace{g_0' + \sum_{j=2}^{n-1}s^j\tilde{h}_j}_{:=g_{n-1}'} + \tilde{\mathcal{O}}(s^n),
    $$
    with $\tilde{h}_j=0$ for $j$ odd. Also, the lowest order terms of the Einstein operator applied to $g_{n-1}'$ are of order $s^n$. Using log-smoothness and Lemma \ref{le:paritylog}, the same argument as for the inductive step gives
    \begin{equation}
    I(D\operatorname{Ric}-\Lambda,s\partial_s)\sum_{m=0}^{m_n}s^{n}\log(s)^m \tilde{h}_{n,m_n} = s^{n}f+\tilde{\mathcal{O}}(s^{n+1}).
    \end{equation}
    As $n$ is odd, $f$ is of the form $(0,f_2,0,0)$. As the second line of $I(D\operatorname{Ric}-\Lambda,\lambda)$ vanishes for all $\lambda$, also $f_2=0$ and therefore $f=0$. This shows that 
    $$
    \sum_{m=0}^{m_n}s^{n}\log(s)^m \tilde{h}_{n,m_n}\in \operatorname{ker}I(D\operatorname{Ric}-\Lambda,s\partial_s),
    $$
    and now Lemma \ref{le:412} for $\lambda=n$ gives $\omega_{n,m}$ and 
    $$
    \tilde{h}_n\in H^\infty_b(\mathcal{I}^+;s^{-2}\operatorname{ker}\operatorname{tr}_{g_{(0)}})
    $$
    such that for $\omega_n:=\sum_{m=0}^{m_n}s^n\log(s)^m\omega_{n,m_n}$
    $$
    \sum_{m=0}^{m_n}s^{n}\log(s)^m \tilde{h}_{n,m_n} + I(\delta^*,s\partial_s)\omega_n=s^n\tilde{h}_n.
    $$
    Again, by multiplying with a suitable cutoff, we can get its $s^{1^-}H^{d_{n+1}+n}$ norm to be smaller than $1/2^n$. Defining $\dot{V_n}$ via $\omega_n=2g_0(\dot{V_n},\cdot)$ and $V_n:=V_{n-1}+\dot{V_n}$ we get the expansion to the next order:
    $$
    \phi^*_{V_n}g=\phi^*_{V_{n-1}}g + \underbrace{\sum_{m=0}^{m_n}s^{n}\log(s)^m \tilde{h}_{n,m_n} +I(\delta^*,s\partial_s)\omega_n}_{s^n\tilde{h}_n + \mathcal{O}(s^\infty)}+\tilde{\mathcal{O}}(s^{n+1}).
    $$
    \\
    \textit{3.2 Order $n+1$ and extracting information on $\tilde{h}_n$} Set $g_n':=g_{n-1}' + s^n\tilde{h}_n$. Repeating the same argument as before gives 
    \begin{equation}
        \label{eq:extractstart}
       ( \operatorname{Ric}g_n'-\Lambda g_n')+  I(D\operatorname{Ric}-\Lambda,s\partial_s)\sum_{m=0}^{m_{n+1}}s^{n+1}\log(s)^m=0 +\tilde{\mathcal{O}}(s^{n+2}),
    \end{equation}
    and $\operatorname{Ric}g_n'-\Lambda g_n'=s^{n+1}f + \mathcal{O}(s^{n+2})$. Because the entire second row of $I(D\operatorname{Ric}-\Lambda,\lambda)$ vanishes for all $\lambda$, we have that $f=(f_1,0,f_3,f_4)$. The second Bianchi identity gives $I(\delta G,n+1)f=0$ and because $f_2=0$ we can still apply Lemma \ref{le:413} to deduce that there exists $\tilde{h}_{n+1}$ such that $I(D\operatorname{Ric}-\Lambda,n+1)\tilde{h}_{n+1}=f$. From there, we continue exactly like in the inductive step to find $V_{n+1}$ such that
    $$
    \phi^*_{V_{n+1}}g = g_n'+\tilde{h}_{n+1}+\tilde{\mathcal{O}}(s^{n+2}).
    $$
    \\
    
    We extract more information on $\tilde{h}_n$ from the vanishing of $f_2$. Writing $g_n'=g_{n-1}'+s^n\tilde{h}_n$ we get 
    $$
    \operatorname{Ric}g_n'-\Lambda g_n'\equiv \operatorname{Ric}g_{n-1}'-\Lambda g_{n-1}'+(D_{g_{0}'}\operatorname{Ric}(s^n\tilde{h}_n) -\Lambda(s^n\tilde{h}_n)) \bmod \mathcal{O}(s^{n+2}).
    $$
    This uses that $g_0'$ differs from $g_{n-1}'$ only at order $s^2$ when changing the point of linearization. Now, as we know from the argument for order $s^n$, $\operatorname{Ric}g_{n-1}'-\Lambda g_{n-1}'$ vanishes at that order. Therefore, by applying Lemma \ref{le:paritylog} to $g_{n-1}'$ we get that the 2-component of $\operatorname{Ric}g_{n-1}'-\Lambda g_{n-1}'$ also vanishes at order $s^{n+1}$. Together with the vanishing of $f_2$, yields that the 2-component of $(D_{g_{0}'}\operatorname{Ric}(s^n\tilde{h}_n)\equiv \operatorname{Ric}(g_0'+s^n\tilde{h}_n)-\operatorname{Ric}(g_0') \bmod \mathcal{O}(s^{2n})$ is zero at order $s^{n+1}$. Working in the local frame $e_\mu$ this is equivalent $\operatorname{Ric}(g_0'+s^n\tilde{h}_n)_{0i}-\operatorname{Ric}(g_0')_{0i}\equiv0\bmod\mathcal{O}(s^{n+2})$. Let $g_{(n)}:=s^2\tilde{h}_n\in\operatorname{ker}\operatorname{tr}_{g_{(0)}}$. Computing the connection coefficients of the metric $g_0'+s^n\tilde{h}_n$, using \eqref{eq:connectiondef} we get $\Gamma^\lambda_{\mu\nu}\equiv 0$, except for
    $$
    \begin{aligned}
        \Gamma^0_{ij}&=-(g_{(0)})_{ij} -\frac{s^n}{2}(g_{(n)})_{ij},\\ \Gamma^l_{i0}&\equiv-\delta^l_i+\frac{s^n}{2}(g_{(0)})^{lm}(g_{(n)})_{mi},\\
        \Gamma^l_{0j}&\equiv\frac{s^n}{2}(g_{(0)})^{lm}(g_{(n)})_{mi},\\ \Gamma^l_{ij}&\equiv s\Gamma(g_{(0)})^l_{ij}+s^{n+1}(g_{(0)})^{lm}\left(\Gamma(g_{(n)})_{mij}-(g_{(n)})_{mb}\Gamma({g_{(0)}}^b_{ij}\right),
    \end{aligned}
    $$
    where the equivalence is modulo $\mathcal{O}(s^{n+2})$ (In fact even $\mathcal{O}(s^{2n})$). Now we can use \eqref{eq:riccicoords} to compute
    $$
    0\stackrel{!}{\equiv}\operatorname{Ric}(g_0'+s^n\tilde{h}_n)_{0i}-\operatorname{Ric}(g_0')_{0i}\equiv \frac{s^{n+1}}{2} (g_{(0)})^{lm}(g_{(n)})_{mi;l}=-\frac{s^{n+1}}{2}\left(\delta_{g_{(0)}}g_{(n)}\right)_i.
    $$\\
    
    \textit{3.3 Arbitrary higher order $i$} This proceeds now exactly like the proof of the general order $i\leq n-1$. The only difference is that the contribution of $s^{n}\tilde{h}_n$ makes it so that the expansion no longer contains only even powers in $s$. Therefore, we cannot eliminate odd orders for $i\geq n$ anymore. By construction, $V_i$ converges in every $s^{1^-}H^k_b$ and therefore lies in $s^\beta H^\infty_b(\Omega;{}^0TM)$\footnote{In fact its $\mathcal{C}^k_b$-norm can be made arbitrarily small}. This closes the proof of the odd-dimensional case.\\
    
    \underline{4. Even-dimensional case}\\
    \textit{4.1 Order $n$} We have constructed $V_{n-1}$ such that 
    $$
    \phi^*_V g = \underbrace{g_0' + \sum_{j=2}^{n-1}s^j\tilde{h}_j}_{:=g_{n-1}'} + \tilde{\mathcal{O}}(s^n),
    $$
    with $\tilde{h}_j=0$ for $j$ odd. Also the lowest order terms of the Einstein operator applied to $g_{n-1}'$ are of order $s^n$. Using log-smoothness and Lemma \ref{le:paritylog}, the same argument as for the inductive step gives
    \begin{equation}
    I(D\operatorname{Ric}-\Lambda,s\partial_s)\sum_{m=0}^{m_n}s^{n}\log(s)^m \tilde{h}_{n,m_n} = s^{n}f+\tilde{\mathcal{O}}(s^{n+1}).
    \end{equation}
    In contrast to when $n$ is odd, we cannot use Lemma \ref{le:paritylog} to show that $f=0$. Instead, we must use the more complicated Lemma \ref{le:412even}, to get one-forms $\omega_{n,m}\in H^\infty(\mathcal{I}^+;s^{-1}T^*\mathbb{S}^n)$ such that for $\omega_n:=\sum_{m=0}^{m_n}s^n\log(s)^m\omega_{n,m}$ we have 
    \begin{equation}
        \sum_{m=0}^{m_n}s^{n}\log(s)^m \tilde{h}_{n,m_n} + I(\delta^*,s\partial_s)\omega_n=s^n(\tilde{h}_n^0+\log(s)\tilde{h}_n^1),
    \end{equation}
    where 
    $$
    \tilde{h}_n^0,\tilde{h}_n^1\in H^\infty(\mathcal{I}^+;s^{-2}\operatorname{ker}\operatorname{tr}_{g_{(0)}}).
    $$ 
    Furthermore, $\frac{2}{n}\tilde{h}_n^1=f_4$ is the 4-component of $f$. Comparing this to definition \eqref{eq:obstructiontensor} we see that $f_4$ is a constant times the obstruction tensor. Therefore $\tilde{h}_n^1$ vanishes, if and only if the obstruction tensor of $g_0'=\underline{g}+h_0$ vanishes. After cutting off $\omega_n$ to ensure it has small norm in $s^{1^-}H^{d_{n+1}+n}_b$, we define $\dot{V}_n$ via $\omega_n=2g_0(\dot{V}_n,\cdot)$ and $V_n:=V_{n-1}+\dot{V}_n$ to get
    $$
    \phi^*_{V_n}g=\phi^*_{V_{n-1}}g + \underbrace{\sum_{m=0}^{m_n}s^{n}\log(s)^m \tilde{h}_{n,m_n} +I(\delta^*,s\partial_s)\omega_n}_{s^n(\tilde{h}_n^0+\log(s)\tilde{h}_n^1) + \mathcal{O}(s^\infty)}+\tilde{\mathcal{O}}(s^{n+1}).
    $$
    \textit{4.2 General order $i$} We proceed inductively. Let $l\in\mathbb{N}$ and $k=0,...,n-1$ and $i=ln+k$ (Notice that the first case, $l=1$, $k=0$ is $i=n$, which we have already proven, so suppose $i>n$). Suppose we have constructed $V_{i-1}\in s^{1^-}H^\infty_b(\Omega;{}^0TM)$ such that 
    $$
    \phi^*_{V_{i-1}}g=\underbrace{g_0'+\sum_{j=2}^{i-1}\sum_{m=0}^{\lfloor j/n\rfloor} s^j\log(s)^m\tilde{h}_j^m}_{:=g_{i-1}'}+\tilde{\mathcal{O}(s^{i})}
   $$
   with $\tilde{h}_j^m=0$ if $j$ is odd. Assume further that
   $$
   \operatorname{Ric}g_{i-1}'-\Lambda g_{i-1}'\equiv0\bmod \tilde{O}(s^{i}).
   $$
   Now, we write equation \eqref{eq:linearized einstein} with $g_0'$ replaced by $g_{i-1}'$ and $\tilde{h}'$ replaced by $\phi^*_{V_{i-1}}g-g_{i-1}'$ to get 
   $$
   I(D\operatorname{Ric}-\Lambda,s\partial_s)\sum_{m=0}^{m_i}s^{i}\log(s)^m \tilde{h}_{i,m_i}  + (\operatorname{Ric}g_{i-1}'-\Lambda g_{i-1}')\equiv0\bmod\tilde{\mathcal{O}}(s^{i+1}) .
   $$
   Notice, how $g_{i-1}'$ is of the form that we can apply Lemma \ref{le:paritylog} for the same $l$ and $k$ as in this proof. So we get $f^{(q)}\in H^\infty(\mathcal{I}^+;S^2T\mathbb{S}^n)$ for $q=0,...,l$, such that
   $$
    I(D\operatorname{Ric}-\Lambda,s\partial_s)\sum_{m=0}^{m_i}s^{i}\log(s)^m \tilde{h}_{i,m_i} =\sum_{q=0}^ls^{i}\log(s)^qf^{(q)}.
   $$
    Because of the second Bianchi identity, $I(\delta G,s\partial_s)\sum_{q=0}^ls^{i}\log(s)^qf^{(q)}=0$. Expanding this out, we get at the highest log order $I(\delta G,i)f^{(l)}=0$. This lets us use Lemma \ref{le:413}\footnote{The fact that the second line of $I(D\operatorname{Ric}-\Lambda,\lambda)$ vanishes for all $\lambda$, lets us deduce that $f_2^{(q)}$, the two component of all forcing terms, is $=0$. This lets us use Lemma \ref{le:413} also in the case $i=n+1$} to find $\tilde{h}_i^l\in H^\infty_b(\mathcal{I}^+;s^{-2}T^*\mathbb{S}^n)$, such that $I(D\operatorname{Ric}-\Lambda,i)\tilde{h}_i^l=f^{(l)}$ and therefore $$
    \begin{aligned}
    I(D\operatorname{Ric}&-\Lambda,s\partial_s)\left(\sum_{m=0}^{m_i}s^{i}\log(s)^m \tilde{h}_{i,m_i} - s^{i}\log(s)^l\tilde{h}_i^l \right)\\
    =&\sum_{q=0}^{l-1}s^{i}\log(s)^qf^{(q)} -s^{i}\log(s)^{l-1}\partial_\lambda I(D\operatorname{Ric}-\Lambda,\lambda)\left|_{\lambda=i}\right.\\
    &-s^{i}\log(s)^{l-2}\partial_\lambda^2 I(D\operatorname{Ric}-\Lambda,\lambda)\left|_{\lambda=i}\right.\\
    =&\sum_{q=0}^{l-1}s^{i}\log(s)^qf'^{(q)}{},
    \end{aligned}
    $$
    for new functions $f'^{(q)}$, but crucially, with the highest order of log reduced by one. Applying the linearized second Bianchi identity to this equation and looking at the highest log order, we get that $f'^{l-1}$ lies in the kernel of $I(\delta G,i)$. Proceeding inductively, we get $\tilde{h}_i^q\in H^\infty_b(\mathcal{I}^+;s^{-2}T^*\mathbb{S}^n)$ for all $q=0,...,l$, such that 
    $$
    I(D\operatorname{Ric}-\Lambda,s\partial_s)\left(\sum_{m=0}^{m_i}s^{i}\log(s)^m \tilde{h}_{i,m_i} - \sum_{q=0}^{l}s^{i}\log(s)^q\tilde{h}_i^q \right)=0.
    $$
    So, as $i\neq n$, we can apply Lemma \ref{le:412} to get one-forms $\omega_{i,m}$ \footnote{If $l> m_i$, then we get $m_i'=l$ for the highest log order of $\omega_i$} such that for $\omega_i$ defined as $\sum_{m=0}^{m_i}s^{i}\log(s)^m\omega_{i,m}$ we have
    $$
    \sum_{m=0}^{m_i}s^{i}\log(s)^m \tilde{h}_{i,m_i} +I(\delta^*,s\partial_s)\omega_i=0.
    $$
    Using the same cutoff argument as before, we can make the $s^{1^-}H^{d_{n+1}+i}_b$ norm of $\omega_i$ smaller than $1/2^{i}$. Defining $\dot{V_i}$ via $\omega_i=2g_0(\dot{V}_i,\cdot)$ and $V_i:=V_{i-1}+\dot{V}_i$ we get
    $$
    \phi_{V_i}^* g= \phi^*_{V_{i-1}} + \underbrace{\sum_{m=0}^{m_i}s^{i}\log(s)^m \tilde{h}_{i,m_i}  +I(\delta^*,s\partial_s)\omega_i}_{\sum_{q=0}^{l}s^{i}\log(s)^q\tilde{h}_i^q  + \mathcal{O}(s^\infty)} +\tilde{\mathcal{O}}(s^{i+1}).
    $$
    Notice that if $i=nl+k$ is odd, then Lemma \ref{le:paritylog} gives us that the $f^{(q)}$ have the form $(0,f^{(q)}_2,0,0)$. But as we have seen multiple times now, $f^{(q)}_2$ must be 0 by the vanishing of the second line of $I(D\operatorname{Ric}-\Lambda,\lambda)$. This gives $f^{(q)}=0$ and therefore lets us choose $\tilde{h}_i^q=0$. This shows that the expansion contains only even powers of $s$, which closes the induction. 
    \\
    
    If the obstruction tensor of $\underline{g}+h_0$ was equal to zero in the step at order $n$ above, in each subsequent step of the above proof we can use Lemma \ref{le:paritylog} for $l=0$ to conclude that the forcing terms $f^{(q)}=0$ for $q\geq 1$. This allows us to choose $\tilde{h}_i^q=0$ for $q\geq 1$ and shows that the metric is smooth down to $\mathcal{I}^+$.
    \\
    
    Using similar arguments as in the $n=$ odd case, by investigating the vanishing of the $0i$ component of $D\operatorname{Ric}g_n'-\Lambda g_n'$ now at the even order $n+2$, we get an equation for the divergence $\tilde{h}_n^0$, but the expression is no longer simple, so we do not compute it explicitly.
    \\
    
    By construction, the sequence $V_i$ converges in every $s^{1^-}H_b^k$ norm and therefore converges to an element in $s^\beta H_b^\infty(\Omega;{}^0TM)$. This finishes the proof of the expansion in the even dimensional case and of part one.\\
    
\textbf{Part two: Block diagonal structure}\\
Now, we solve away the order $s^\infty$ remainder terms obstructing the block-diagonal structure of the metric near $s=0$. To wit, we construct the boundary normal coordinates analogous to \cite{GRAHAM1991186} in a collar neighborhood of $s=0$. In order to not destroy the already constructed expansion, we need to ensure that this last pullback changes the metric only at order $s^\infty H^\infty_b$. Renaming the $\phi_V^*g$ constructed in part one of the proof to $g$, we assume that $g$ has the form \eqref{eq:expansionodd} or \eqref{eq:expansioneven}.
\\

Analogously to \cite{GRAHAM1991186}, we wish to construct a defining function $\tilde{s}$ for $\mathcal{I}^+$ that fulfills \begin{equation}
\label{eq:deffunccond}
    |\mathrm{d}\tilde{s}|^2_{\tilde{s}^2g}=-1
\end{equation}
in a neighborhood of $\mathcal{I}^+$ in $M$\footnote{In the Riemannian setting discussed in \cite{GRAHAM1991186}, the analogous condition reads $|\mathrm{d}\tilde{s}|_{\tilde{s}^2g}=1$}. Notice that our initial defining function $s$ already fulfills $|\mathrm{d}s|_{s^2g}^2\equiv-1 \bmod s^\infty H^\infty_b$ by construction. Setting $\tilde{s}=e^{u}s$ for a function $u$ on $M$ with $u=0$ on $\mathcal{I}^+$, the condition \eqref{eq:deffunccond} is equivalent to 
\begin{equation}
\label{eq:ueq}
    \frac{1+|\mathrm{d}s|^2_{s^2g}}{s}+s|\mathrm{d}u|^2_{s^2g}+2\langle\mathrm{d}s,\mathrm{d}u\rangle_{s^2g}=0.
\end{equation}
This is a non-characteristic first-order ODE for $u$ and we immediately get a solution in a neighborhood of $\mathcal{I}^+$. But in order to prove that $u$ decays fast enough toward $\mathcal{I}^+$, the proof of existence has to be inspected in detail. This is the content of proposition \ref{prop:technicalblockdiag}, which for legibility is stated after the proof of this theorem. The proposition tells us that $u$ lies in $s^\infty H^\infty.$\\

Now we proceed to construct the final pullback of our metric $g$. Set $\tilde{s}:=e^{u}s$ and $V:=-\operatorname{grad}_{\tilde{s}^2g}\tilde{s}$. Let $\gamma_\omega(\lambda)$ be the integral curve of $V$ starting at $\omega\in \mathcal{I}^+$. Then by direct computation we get $\partial_\lambda \tilde{s}(\gamma_w(\lambda)) = -|\mathrm{d}\tilde{s}|_{\tilde{s}^2g}^2=1$ and $\tilde{s}(0)=0$ and therefore $\lambda=\tilde{s}$. So we get a diffeomorphism $\Phi$ from $[0,\epsilon]_{\tilde{s}}\times\mathcal{I}^+_\omega$ to a neighborhood of $\mathcal{I}^+$ in $M$ by letting $\Phi(\tilde{s},\omega)=\gamma_\omega(\tilde{s})$. By construction (or by direct computation in a chart), the metric $\Phi^*g$ is of the desired block-diagonal form $\frac{-\mathrm{d}\tilde{s}^2 + H(\tilde{s},\omega;\mathrm{d}\omega)}{\tilde{s}^2}$. A quick computation shows that $-\operatorname{grad}_{\tilde{s}^2g}\tilde{s}\equiv \partial_s\bmod s^\infty H^\infty_b\equiv \partial_{\tilde{s}} \bmod s^\infty H^\infty_b$. Using a coordinate system on $\mathcal{I}^+$ we can write
$$
\begin{aligned}
\Phi(\tilde{s},\omega)&=(0,\omega) +\int_0^{\tilde{s}}-\operatorname{grad}_{\tilde{s}^2g}\tilde{s}\left|_{\gamma_\omega(t')}\right.\mathrm{d}t'\\
&\equiv(0,\omega)+\int_0^{\tilde{s}} (1,0)\mathrm{d}t'\bmod s^\infty H^\infty_b\\
&\equiv (\tilde{s},\omega)\bmod s^\infty H^\infty_b
\end{aligned}
$$
and therefore
$$
H_{ij}=\partial_i\Phi^\mu\partial_j\Phi^\nu g_{\mu\nu}\equiv g_{ij} \bmod s^\infty H^\infty_b.
$$
This proves that this last pullback only changes the metric at infinite order of decay and therefore does not change the constructed asymptotic expansion. This concludes the proof of the theorem.
\end{proof}
This last proposition is concerned with the technical part of the proof of the decay behavior of the solution $u$ of \eqref{eq:ueq}. It is used in the proof of the preceding theorem \ref{thm:expansion}.
\begin{proposition}
\label{prop:technicalblockdiag}
    Let $n\geq3$ and let $g$ be the metric produced in part one of the proof of theorem \ref{thm:expansion}, satisfying \eqref{eq:expansionodd} or \eqref{eq:expansioneven}.\footnote{With $g$ in the place of the $ \phi^*g$ of the Theorem} Let $u$ be the solution to \eqref{eq:ueq}, namely
    $$
    \frac{1+|\mathrm{d}s|^2_{s^2g}}{s}+s|\mathrm{d}u|^2_{s^2g}+2\langle\mathrm{d}s,\mathrm{d}u\rangle_{s^2g}=0.
    $$
    Then $u$ lies in $s^\infty H^\infty_b$.
\end{proposition} 
\begin{proof}
    As the solution is unique, we can work in a local coordinate system $(s,x)$ in a neighborhood $[0,\epsilon]\times\mathcal{I}^+$. Taking the expansion \eqref{eq:expansionodd} or \eqref{eq:expansioneven} for $g$ at order $N$, we can write, for any $N\in\mathbb{N}$,
\begin{equation}
    (s^2g)^{-1}=\left(\begin{array}{cc}
        -1 & 0 \\
        0 & G_N(s,x)
    \end{array}\right)+h_N(s,x),
\end{equation}
where $G_N(s,x)$ is a $n\times n$ matrix that is smooth in $x$, log-smooth and $\mathcal{C}^3$ in $s$ up to $\mathcal{I}^+$\footnote{Because the lowest even dimension considered is $n=4$, we get the lowest log contribution at $s^4$ and therefore $G_N$ is $\mathcal{C}^3$}, and $h_N\in s^N H^\infty_b$. We write $G_N^{ij}$ and $h_N^{\mu\nu}$ for the components of the matrices. Because of Sobolev embedding, we have $|(s\partial_s)^kh_N^{\mu\nu}|\leq C_{N,k}s^N$ and therefore $h_N^{\mu\nu}\in \mathcal{C}^{N-1}$. We will drop the subscript $N$ from $G$ and $h$ to increase legibility from here on.\\

We investigate the proof of the existence and uniqueness of the solution to nonlinear first order PDEs, as described in \cite{evans}. We parameterize a  region of $\mathcal{I}^+$ by $r$ in a compact subset of $\mathbb{R}^n$ and write an ODE for the characteristic curves in $t$ starting at $r$. Set $z:=u$, $p_0:=\partial_su$ and $p_i:=\partial_iu$ and consider them, as well as $s$ and $x$, as functions of $r,t$. Writing $v=(s,x,z,p)$ we can rewrite our problem into a solution of a first order ODE in the following way. Defining the left-hand side of \eqref{eq:ueq} as $F$ we get in coordinates
$$F(s,x,z,p_\mu)=-\frac{h^{00}}{s} + 2(-p_0+h^{\mu0}p_\mu) + s(-p_0^2+G^{ij}p_ip_j+h^{\mu\nu}p_{\mu\nu})\stackrel{!}{=}0.$$
Our first order system\footnote{For motivation of this proof, see \cite{evans}} is then
\begin{equation}
\label{eq:ODEsyst}
    \begin{aligned}
        \frac{\mathrm{d}s}{\mathrm{d}t}&=\partial_{p_0}F=2(-1 + h^{00} -sp_0 + h^{\mu0}p_\mu)\\
        \frac{\mathrm{d}x^{i}}{\mathrm{d}t}&=\partial_{p_i}F=2(h^{0i}+sG^{ij}p_j+sh^{i\mu}p_\mu\\
        \frac{\mathrm{d}z}{\mathrm{d}t}&=p_{\mu}\partial_{p_\mu}F=2(-p_0-sp_0^2+sG^{ij}p_ip_j +h^{0\mu}p_\mu + h^{\mu\nu}p_\mu p_\nu)\\
        \frac{\mathrm{d}p_0}{\mathrm{d}t}&=-\partial_sF=\frac{s\partial_sh^{00}-h^{00}}{s^2} + p_0^2 -G^{ij}p_i p_j-h^{\mu\nu}p_\mu p_\nu \\
        &\quad\quad\quad\quad\quad-2\partial_sh^{0\mu}p_\mu-s\partial_sh^{\mu\nu}p_\mu p_\nu-s\partial_s G^{ij}p_ip_j\\
        \frac{\mathrm{d}p_i}{\mathrm{d}t}&=-\partial_{p_i}F=\frac{\partial_ih^{00}}{s} -s\partial_i G^{ab}p_ap_b-s\partial_ih^{\mu\nu}p_\mu p_\nu.
    \end{aligned}
\end{equation}
Define the right-hand side as $f(v)$ and notice that for $N$ large enough, $f(v)$ depends on $s$ in a $\mathcal{C}^{3}$ fashion, and on the other variables in a $\mathcal{C}^\infty$ fashion. The initial conditions for this system are given by $s(0,r)=z(0,r)=p_\mu(0,r)=0$ and $x(0,r)=r$. They fulfill the compatibility conditions as defined in \cite{evans}. Therefore, by Picard-Lindelöf, we get a unique $C^1$ solution $v$ for $t\leq\epsilon>0$. Because of the regularity conditions on $f$, we immediately get that it is $\mathcal{C}^4$ in $t$ and $r$. Apriori, we do not get more, as when differentiating the equation with respect to $t$ or $r$, by the chain rule we get $\partial_sf$ terms, which lose one order of differentiability. We show later by closer examination, that the solution is smooth in $r$ as expected.\\

Now, as $\frac{\mathrm{d}s}{\mathrm{d}t}(0,r)=-2\neq 0$, we see that we can invert $s(t,r),x(t,r)$ to $t(s,x),r(s,x)$ in a collar neighborhood of $\mathcal{I}^+$ in a $\mathcal{C}^4 $ fashion and set $u(s,x)=z(t(s,x),r(s,x))$, which solves the equation \eqref{eq:ueq} by construction. We wish to investigate this solution more closely to show that it decays faster than any $s^M$ towards $s=0$.
\\

Looking at the equation for $\frac{\mathrm{d}s}{\mathrm{d}t}$, we see that we can estimate 
\begin{equation}
    \label{eq:steq}
    C_1t\leq s\leq C_2t,
\end{equation}
for positive constants $C_1$ and $C_2$ independent of $r$ because of the continuous dependence of the solution on $r$. Because of this, we can estimate all terms containing $h$ in $f$ like 
\begin{equation}
\label{eq:hesti}
|\partial_s^lh(s(t,r),x(t,r))|\leq Cs^{N-l}\leq Ct^{N-l}
\end{equation}
We begin by extracting the behavior of $w:=(z,p)$ as $t$ goes to 0. By considering $s$ and $x$ to as functions of $t$ we can write the $w$ part of system \eqref{eq:ODEsyst}
\begin{equation}
\label{eq:weqstruct}
w'=f_1(t,w)+f_2(t)
\end{equation}
with $f_1,f_2\in \mathcal{C}^4$. They further fulfill $f_1(t,0)=0$ and $|f_2(t)|\leq C t^{N-2}$. All terms of $f$ not containing $h^{\mu\nu}$ can be put into $f_1$ (the explicit $t$ dependence of $f_1$ stems from the fact that we no longer consider $s$ and $x$ to be variables of our equation). All terms containing $h^{\mu\nu}$ can be put into $f_2$ and by estimates of the type $\eqref{eq:hesti}$, we get the decay\footnote{The $\frac{s\partial_sh^{00}-h^{00}}{s^2}$ term in $f$ is responsible for the loss of two orders of decay, but as in the end $N$ was arbitrary, it does not matter}. Now integrating the differential equation for $w$ for $t\leq\epsilon$ we get
\begin{equation}
    \begin{aligned}
        |w(t)|&=\left|\int_0^tf_1(t',w(t'))+f_2(t')\mathrm{d}t'\right|\\
        &\leq \int_0^t C|w(t')| + C t'^{N-2}\mathrm{d}t'\\
        &\leq Ct \operatorname{max}_{t'\in[0,t]}|w(t')| + Ct^{N-1}\\
        &\leq C\epsilon \operatorname{max}_{t'\in[0,\epsilon]}|w(t')| + Ct^{N-1}.
    \end{aligned}
\end{equation}
By choosing $0<\epsilon'\leq 1/2C\operatorname{max}_{t'\in[0,\epsilon]}|w(t')|$, we thus get 
\begin{equation}
\label{eq:westimate}
    |w(t)|\leq C t^{N-1}
\end{equation}
for $t \leq \epsilon'$. This gives for our solution $|u(s,x)|= |z(t(s,x),r(s,x))|\leq Ct^{N-1}\leq Cs^{N-1}$.
\\

Later, we will also need that arbitrary $r$ derivatives of $w$ decay as $t^{N-1}$ toward $t=0$. First we prove that indeed the solution is $\mathcal{C}^\infty$ in $r$.\\

Differentiating the ODE $\partial_tv=f$ in $r$ gives us a system of ODEs for $\partial_rv$. The right-hand side of this ODE contains the terms $\partial_rf=\partial_sf\partial_rs+\partial_xf\partial_rx+\partial_pf\partial_rp$, which lies in $C^2$, as we apriori lose one order of regularity in the term $\partial_sf$. The initial conditions for the components $I$ of $v$ are  $\partial_rv^{I}(0,r)=0$, except for $\partial_rx^{i}(0,r)=1$. By Piccard-Lindelöf we get a $\mathcal{C}^3$ solution. A standard argument shows that this solution is indeed the derivative of $v$ with respect to $r$. Furthermore, as $\partial_rs(0,r)=0$, we see that $|\partial_rs|\leq Ct\leq C's$. Now we can inductively derive the ODE $k$ times with respect to $r$, always with right-hand sides $\partial_r^kf$ being in $\mathcal{C}^2$. We do not lose regularity, as the potentially problematic terms, $\partial_s^lf$, are always accompanied by terms that are bounded by $Cs^l$. They come from the application of the chain rule $\partial_rf=\partial_sf\partial_rs +\ldots$ and are $l$ factors of terms of the form $\partial^m_rs$ for $m< k$. As the initial condition for $\partial_r^ks$ is $\partial_r^ks(0,r)=0$, by the result at order $k-1$, they are bounded by $Cs$. The initial conditions for all components read $\partial_r^kv(0,r)=0$ for $k\geq 2$. Again, it can be shown that the unique solutions to the ODEs are in fact the real derivatives of $v$. In this way, we get smoothness of $v$ in $r$.\\

To now show that arbitrary $r$ derivatives of $w=(z,p)$ decay as $s^{N-2}$, we again look at the ODEs for $\partial_r^k w$ and proceed inductively. Using that for the orders $l<k$ we already have $|\partial_r^lw|\leq Ct^{N-1}$, we can show that we can split the right-hand side of $\partial_t\partial_r^kw=\partial_r^kf$ into the same form as \eqref{eq:weqstruct}. By the same proof as above we therefore get the decay of $|\partial_r^kw|\leq C t^{N-1}$.
\\

Now, we are prepared to investigate the behavior of our solution to the original equation \eqref{eq:ueq},  $u(s,x)=z(t(s,x),r(s,x))$. We wish to show that an arbitrary number of b-derivatives of $u$ exist and are bounded by $Cs^{N-1}$. This relies on iterative applications of the chain rule together with the explicit expression of the derivative of the inverse using the inverse function Theorem and Cramer's rule. We call $f=(f^s,f^{x^{i}},f^z,f^{p_\mu})$ the components of $f$. Notice that because of \eqref{eq:westimate}, all $f^{I}$ except for $f^s$ are bounded by $Cs^{N-2}$. We have to be careful when examining the derivatives of $f$. Now, we are considering $f$ as a function of only $s,x$, as we consider $p_\mu$ to be functions of $s,x$ through its $t,r$ dependence. We write primed derivatives when we mean partial derivatives with respect to only the explicit dependence of $f$ on the variables $s,x$. This means for example $\partial_xf=\partial'_xf + \partial_pf\partial_xp$.\\

To illustrate the general procedure, we first consider one b-derivative of $u$ in detail. Let $A = \partial_ts\partial_rx-\partial_tx\partial_rs=f^s\partial_rx-f^x\partial_rs$ be the determinant of the Jacobi matrix of the change of coordinates. It is uniformly bounded away from $0$ and $\infty$ on our domain. By the chain rule, we have $s\partial_su=s(\partial_t{z}\partial_st+\partial_rz\partial_sr)$. Using the inverse function formula
\begin{equation}
\label{eq:inversion}
\left(\begin{array}{cc}
     \partial_st&\partial_xt  \\
     \partial_sr& \partial_xr
\end{array}\right)=\frac{1}{A}\left(\begin{array}{cc}
    \partial_rx &-\partial_rs  \\
     -\partial_tx&\partial_ts 
\end{array}\right)
\end{equation}
we get $s\partial_su=\frac{1}{A}(f^zs\partial_rx-f^xs\partial_rz)$. This can easily be estimated by $Cs^{N-2}$ because of the bounds on $f^z,\ f^x$ and the continuity of the other terms ($A\approx -2\neq0$). For a derivative along $\partial_x$ of $u$ we compute $\partial_xu=\partial_tz\partial_xt+\partial_rz\partial_xr=\frac{1}{A}(-f^z\partial_rs+\partial_rzf^s)$. The first summand can be estimated by $Cs^{N-2}$ because of the decay of $f^z$, the second summand because of the decay of $\partial_rz$. \\

The proof in general dimensions is analogous, the only change is that the formula for the inverse gets more complicated, the numerator will now be a sum of products. The relevant terms for the decay will always appear analogously in each term, the other factors can just be estimated by a constant.\\

Now to general order (again for $n=1$ to increase readability). Assume by induction that for $k\in \mathbb{N}$ we have that $\partial_r^{a}(s\partial_s)^bu$ with $a+b\leq k-1$ exists and is of the following form, with $c+d,l\leq k-1$ and $m\leq2$:
\begin{equation}
    \label{eq:bderivform}
    P\left(R_c(\partial'_s)^c(\partial'_x)^df^{I},\partial_r^l v^{I},\frac{1}{A},\partial_p^m(R_c(\partial'_s)^c(\partial'_x)^df^{I})\right).
\end{equation}
Here $P$ is a polynomial and the $R_e$ are a product of $e$ factors of either $s$ or terms of the form $\partial_r^bs$ with $b\leq e$. This means that the $R_e$ can be estimated by $Cs^{e}$. Furthermore, we assume each summand of $P$ contains at least one \textit{critical factor}. A critical factor is either a term of the form $R_d(\partial'_s)^d(\partial'_x)^cf^{I}$ with $I\neq s$ if there are no derivatives, or a term of the form $\partial_r^lv^{I}$ with $I=z,p$. By estimating the critical factor by $Cs^{N-2}$ and all the remaining factors by a constant, we get that $|\partial_r^{a}(s\partial_s)^bu|\leq Cs^{N-2}$. We wish to show that $s\partial_sP$ and $\partial_xP$ exists and also is of the same form, with $k$ changed to $k+1$.\\

We investigate what happens when either a $s\partial_s$ or a $\partial_x$ derivative hits a factor of a summand of $P$. First, suppose it hits a critical factor of the form $\partial_r^lz$. We compute for example
$$
\begin{aligned}
    \partial_x(\partial_r^lz)&=\partial_r^{l+1}z\partial_xr+\partial_r^l\partial_tz\partial_xt\\
    &=\partial_r^{l+1}z\frac{f^s}{A}-\partial_r^lf^z\frac{\partial_rs}{A},
\end{aligned}
$$
using the inversion formula \eqref{eq:inversion} and $\partial_tv=f$. The first summand contains a critical factor of $\partial_r^{l+1}z$, multiplied by allowed inputs to $P$. For the $\partial_r^lf^z$ term we compute
$$
    \partial_r^lf^z=\partial_r^{l-1} (\partial_s'f^z\partial_rs+\partial_x'f^z\partial_rx +\partial_pf^z\partial_rp).
$$
Performing all $\partial_r$ derivatives in this fashion, we see that the amount of $\partial_s'$ derivatives acting on $f^s$ is identical to the number of factors of terms of the form $\partial_r^ms$. So the terms with at least one $\partial_s'f$, are of the form $R_a(\partial_s')^a(\partial_x')^mf^z$ and therefore contain a critical factor. In the terms with only $\partial_x'$ derivatives, $(\partial_x')^mf^z$ is a critical factor, and in the remaining terms, we have at least one factor of $\partial_r^mp$. Therefore, each summand contains at least one critical factor and we have shown that $\partial_x(\partial_r^lz)$ is of the same form as $P$ with $k$ increased by one. The terms $s\partial_s\partial_r^lz,\partial_x\partial_r^lp,s\partial_s\partial_r^lp$ work analogously.\\

Suppose next, a b-derivative hits a critical factor of the form $R_d(\partial'_s)^d(\partial'_x)^cf^{I}$. Taking for example the $\partial_x$ derivative, we get two terms, 
\begin{equation}
\label{eq:Rderivs}
\partial_x (R_d)\cdot (\partial'_s)^d(\partial'_x)^cf^{I}+R_d\cdot\partial_x ((\partial'_s)^d(\partial'_x)^cf^{I}).
\end{equation}
When in the first term the derivative hits $R_d$, we have to calculate
$$
\partial_x\partial_r^ls=\partial_r^{l+1}s\partial_xr+\partial_r^lf^s\partial_xt=\partial_r^{l+1}s\frac{f^s}{A}-\partial_r^lf^s\frac{\partial_rs}{A}.
$$
We see that both summands still contain a factor of $\partial_r^ms$ and therefore $\partial_x R_d$ is given by a sum of $R_d$-type terms multiplied by factors that are inputs of $P$ (The $\partial_r^lf^s$ can be evaluated into terms that are inputs to $P$ like $\partial_r^lf^z$ above). Therefore, in the first term in \eqref{eq:Rderivs} each summand has a critical factor. Turning to the second term in \eqref{eq:Rderivs}, we compute
$$
\begin{aligned}
\partial_x ((\partial'_x)^c(\partial'_s)^df^{I})&=(\partial'_x)^{c+1}(\partial'_s)^df^{I} + \partial_p((\partial'_x)^c(\partial'_s)^df^{I})\cdot\partial_xp\\
&=(\partial'_x)^{c+1}(\partial'_s)^df^{I} + \partial_p((\partial'_x)^c(\partial'_s)^df^{I})\cdot (\partial_tp\partial_xt +\partial_rp\partial_xr)\\
&=(\partial'_x)^{c+1}(\partial'_s)^df^{I} + \partial_p((\partial'_x)^c(\partial'_s)^df^{I})\cdot (-f^p\frac{\partial_rs}{A}+\partial_rp\frac{f^s}{A}).
\end{aligned}
$$
After multiplying each summand by $R_d$, it has the form of factors of inputs of $P$ multiplied by a critical factor. The first summand is the critical factor itself, in the second it is $f^p$ and in the third it is $\partial_rp$. The computation for $s\partial_s$ is analogous. This shows that a b-derivative of a critical factor is a product of the inputs of $P$ multiplied by at least one critical factor.\\

For the rest of the inputs of $P$ we only have to show that b-derivatives only give rise to terms that are again inputs of $P$. We no longer need to keep track of critical factors. They follow after similar but simpler versions of the arguments above.\\

This finally closes the induction and shows that $|(s\partial_s)^a\partial_x^bu|\leq C s^{N-2}$ for arbitrary $a,b\in\mathbb{N}$. But as $N\in\mathbb{N}$ was arbitrary, we get that $u\in s^\infty H^\infty_b$ as was our claim.\\
\end{proof}
\newpage
\nocite{*}
\bibliography{references} 

\end{document}

%% file: Mani.pdf_tex
\begingroup%
  \makeatletter%
  \providecommand\color[2][]{%
    \errmessage{(Inkscape) Color is used for the text in Inkscape, but the package 'color.sty' is not loaded}%
    \renewcommand\color[2][]{}%
  }%
  \providecommand\transparent[1]{%
    \errmessage{(Inkscape) Transparency is used (non-zero) for the text in Inkscape, but the package 'transparent.sty' is not loaded}%
    \renewcommand\transparent[1]{}%
  }%
  \providecommand\rotatebox[2]{#2}%
  \newcommand*\fsize{\dimexpr\f@size pt\relax}%
  \newcommand*\lineheight[1]{\fontsize{\fsize}{#1\fsize}\selectfont}%
  \ifx\svgwidth\undefined%
    \setlength{\unitlength}{197.1877427bp}%
    \ifx\svgscale\undefined%
      \relax%
    \else%
      \setlength{\unitlength}{\unitlength * \real{\svgscale}}%
    \fi%
  \else%
    \setlength{\unitlength}{\svgwidth}%
  \fi%
  \global\let\svgwidth\undefined%
  \global\let\svgscale\undefined%
  \makeatother%
  \begin{picture}(1,0.90262844)%
    \lineheight{1}%
    \setlength\tabcolsep{0pt}%
    \put(0,0){\includegraphics[width=\unitlength,page=1]{Mani.pdf}}%
    \put(-0.00142744,0.52604297){\color[rgb]{0,0,0}\makebox(0,0)[lt]{\lineheight{0.1}\smash{\begin{tabular}[t]{l}$s=s_0$\end{tabular}}}}%
    \put(-0.00008261,0.78311843){\color[rgb]{0,0,0}\makebox(0,0)[lt]{\lineheight{0.1}\smash{\begin{tabular}[t]{l}$s=0$\end{tabular}}}}%
    \put(0.00617942,0.13048092){\color[rgb]{0,0,0}\makebox(0,0)[lt]{\lineheight{0.1}\smash{\begin{tabular}[t]{l}$s=\frac{\pi}{2}$\end{tabular}}}}%
    \put(0.51584615,0.51843611){\color[rgb]{0,0,0}\makebox(0,0)[lt]{\lineheight{0.1}\smash{\begin{tabular}[t]{l}$^\uparrow\Omega$\end{tabular}}}}%
  \end{picture}%
\endgroup%

%% file: Doma.pdf_tex
\begingroup%
  \makeatletter%
  \providecommand\color[2][]{%
    \errmessage{(Inkscape) Color is used for the text in Inkscape, but the package 'color.sty' is not loaded}%
    \renewcommand\color[2][]{}%
  }%
  \providecommand\transparent[1]{%
    \errmessage{(Inkscape) Transparency is used (non-zero) for the text in Inkscape, but the package 'transparent.sty' is not loaded}%
    \renewcommand\transparent[1]{}%
  }%
  \providecommand\rotatebox[2]{#2}%
  \newcommand*\fsize{\dimexpr\f@size pt\relax}%
  \newcommand*\lineheight[1]{\fontsize{\fsize}{#1\fsize}\selectfont}%
  \ifx\svgwidth\undefined%
    \setlength{\unitlength}{196.99636865bp}%
    \ifx\svgscale\undefined%
      \relax%
    \else%
      \setlength{\unitlength}{\unitlength * \real{\svgscale}}%
    \fi%
  \else%
    \setlength{\unitlength}{\svgwidth}%
  \fi%
  \global\let\svgwidth\undefined%
  \global\let\svgscale\undefined%
  \makeatother%
  \begin{picture}(1,0.90289459)%
    \lineheight{1}%
    \setlength\tabcolsep{0pt}%
    \put(0,0){\includegraphics[width=\unitlength,page=1]{Doma.pdf}}%
    \put(-0.00142868,0.39152323){\color[rgb]{0,0,0}\makebox(0,0)[lt]{\lineheight{0.1}\smash{\begin{tabular}[t]{l}$s=s_0$\end{tabular}}}}%
    \put(-0.00008262,0.77448538){\color[rgb]{0,0,0}\makebox(0,0)[lt]{\lineheight{0.1}\smash{\begin{tabular}[t]{l}$s=0$\end{tabular}}}}%
    \put(-0.00134601,0.67407246){\color[rgb]{0,0,0}\makebox(0,0)[lt]{\lineheight{0.1}\smash{\begin{tabular}[t]{l}$s=s_1$\end{tabular}}}}%
  \end{picture}%
\endgroup%

%% file: references.bib
@misc{hintz2024stability,
      title={Stability of the expanding region of Kerr-de Sitter spacetimes and smoothness at the conformal boundary}, 
      author={Peter Hintz and András Vasy},
      year={2024},
      eprint={2409.15460},
      archivePrefix={arXiv},
      primaryClass={gr-qc},
      howpublished={arXiv:2409.15460},
      url={https://arxiv.org/abs/2409.15460}, 
}

@book{taylor,
  author    = {Michael E. Taylor},
  title     = {Partial Differential Equations III: Nonlinear Equations},
  series    = {Applied Mathematical Sciences},
  volume    = {117},
  year      = {2023},
  publisher = {Springer},
  address   = {New York}
}

@article {nashmoser,
    AUTHOR = {Saint Raymond, Xavier},
     TITLE = {A simple {N}ash-{M}oser implicit function theorem},
   JOURNAL = {Enseign. Math. (2)},
  FJOURNAL = {L'Enseignement Math\'ematique. Revue Internationale. 2e
              S\'erie},
    VOLUME = {35},
      YEAR = {1989},
    NUMBER = {3-4},
     PAGES = {217--226},
      ISSN = {0013-8584},
   MRCLASS = {58C15},
  MRNUMBER = {1039945},
MRREVIEWER = {Taruvai\ Subramaniam},
}

@BOOK{ringstromcauchy,
  title     = "The Cauchy problem in general relativity",
  author    = "Ringstrom, Hans",
  publisher = "European Mathematical Society",
  series    = "ESI Lectures in Mathematics \& Physics",
  month     =  jun,
  year      =  2009,
  address   = "Z{\"u}rich, Switzerland",
  language  = "en"
}

@ARTICLE{DeTurck,
  title     = "Existence of metrics with prescribed Ricci curvature: Local
               theory",
  author    = "DeTurck, Dennis M",
  journal   = "Invent. Math.",
  publisher = "Springer Nature",
  volume    =  65,
  number    =  2,
  pages     = "179--207",
  month     =  jun,
  year      =  1981,
  language  = "en"
}

@article{GRAHAM1991186,
title = {Einstein metrics with prescribed conformal infinity on the ball},
journal = {Advances in Mathematics},
volume = {87},
number = {2},
pages = {186-225},
year = {1991},
issn = {0001-8708},
doi = {https://doi.org/10.1016/0001-8708(91)90071-E},
url = {https://www.sciencedirect.com/science/article/pii/000187089190071E},
author = {C.Robin Graham and John M Lee}
}

@book{evans,
  author    = {Lawrence C. Evans},
  title     = {Partial Differential Equations},
  edition   = {2},
  series    = {Graduate Studies in Mathematics},
  volume    = {19},
  publisher = {American Mathematical Society},
  year      = {2010}
}

@article{FRIEDRICH1986101,
title = {Existence and structure of past asymptotically simple solutions of Einstein's field equations with positive cosmological constant},
journal = {Journal of Geometry and Physics},
volume = {3},
number = {1},
pages = {101-117},
year = {1986},
issn = {0393-0440},
doi = {https://doi.org/10.1016/0393-0440(86)90004-5},
url = {https://www.sciencedirect.com/science/article/pii/0393044086900045},
author = {Helmut Friedrich}
}

@ARTICLE{Ringstrom2008-zu,
  title     = "Future stability of the Einstein-non-linear scalar field system",
  author    = "Ringstr{\"o}m, Hans",
  journal   = "Invent. Math.",
  publisher = "Springer Science and Business Media LLC",
  volume    =  173,
  number    =  1,
  pages     = "123--208",
  month     =  jul,
  year      =  2008,
  language  = "en"
}

@incollection{fefferman1985,
  author    = {Charles Fefferman and C. Robin Graham},
  title     = {Conformal invariants},
  booktitle = {{\'E}lie Cartan et les math{\'e}matiques d'aujourd'hui -- Lyon, 25--29 juin 1984},
  series    = {Ast{\'e}risque},
  volume    = {S131},
  year      = {1985},
  pages     = {95--116},
  publisher = {Société Mathématique de France},
  url       = {https://www.numdam.org/item/AST_1985__S131__95_0/}
}

@misc{fefferman2008ambientmetric,
  author  = {Charles Fefferman and C. Robin Graham},
  title   = {The ambient metric},
  howpublished = {arXiv:0710.0919},
  year    = {2008}
}

@ARTICLE{Hintzasydesi,
  title    = "Asymptotically de Sitter metrics from scattering data in all
              dimensions",
  author   = "Hintz, Peter",
  journal  = "Philos. Trans. A Math. Phys. Eng. Sci.",
  volume   =  382,
  number   =  2267,
  pages    = "20230037",
  month    =  mar,
  year     =  2024,
  language = "en"
}

@ARTICLE{Rodnianski2018-lt,
  title     = "The asymptotically self-similar regime for the Einstein vacuum
               equations",
  author    = "Rodnianski, Igor and Shlapentokh-Rothman, Yakov",
  journal   = "GAFA Geom. Funct. Anal.",
  publisher = "Springer Science and Business Media LLC",
  volume    =  28,
  number    =  3,
  pages     = "755--878",
  month     =  jun,
  year      =  2018,
  language  = "en"
}

@ARTICLE{Melrose1981-mz,
  title     = "Transformation of boundary problems",
  author    = "Melrose, Richard B",
  journal   = "Acta Math.",
  publisher = "International Press of Boston",
  volume    =  147,
  number    =  0,
  pages     = "149--236",
  year      =  1981,
  language  = "en"
}

@techreport{Melrose1983EllipticOO,
  author      = {Richard B. Melrose and Gerardo Mendoza},
  title       = {Elliptic operators of totally characteristic type},
  institution = {Mathematical Sciences Research Institute},
  year        = {1983},
  address     = {Berkeley, CA}
}

@BOOK{Melrose1993-uz,
  title     = "The {Atiyah-Patodi-Singer} Index Theorem",
  author    = "Melrose, Richard",
  publisher = "CRC Press",
  series    = "Research Notes in Mathematics",
  month     =  mar,
  year      =  1993,
  address   = "Boca Raton, FL"
}

@article{Anderson2005,
  title = {Existence and Stability of Even-dimensional Asymptotically de Sitter Spaces},
  volume = {6},
  ISSN = {1424-0661},
  url = {http://dx.doi.org/10.1007/s00023-005-0224-x},
  DOI = {10.1007/s00023-005-0224-x},
  number = {5},
  journal = {Annales Henri Poincaré},
  publisher = {Springer Science and Business Media LLC},
  author = {Anderson,  Michael T.},
  year = {2005},
  month = oct,
  pages = {801–820}
}

@article{Friedrich1986,
  title = {On the existence ofn-geodesically complete or future complete solutions of Einstein’s field equations with smooth asymptotic structure},
  volume = {107},
  ISSN = {1432-0916},
  url = {http://dx.doi.org/10.1007/BF01205488},
  DOI = {10.1007/bf01205488},
  number = {4},
  journal = {Communications in Mathematical Physics},
  publisher = {Springer Science and Business Media LLC},
  author = {Friedrich,  Helmut},
  year = {1986},
  month = dec,
  pages = {587–609}
}

@article{Kichenassamy2004,
  title = {On a conjecture of Fefferman and Graham},
  volume = {184},
  ISSN = {0001-8708},
  url = {http://dx.doi.org/10.1016/S0001-8708(03)00145-2},
  DOI = {10.1016/s0001-8708(03)00145-2},
  number = {2},
  journal = {Advances in Mathematics},
  publisher = {Elsevier BV},
  author = {Kichenassamy,  Satyanad},
  year = {2004},
  month = jun,
  pages = {268–288}
}

@article{Rendall2004,
  title = {Asymptotics of Solutions of the Einstein Equations with Positive Cosmological Constant},
  volume = {5},
  ISSN = {1424-0661},
  url = {http://dx.doi.org/10.1007/s00023-004-0189-1},
  DOI = {10.1007/s00023-004-0189-1},
  number = {6},
  journal = {Annales Henri Poincaré},
  publisher = {Springer Science and Business Media LLC},
  author = {Rendall,  Alan D.},
  year = {2004},
  month = dec,
  pages = {1041–1064}
}

@book{Christodoulou1994,
  title = {The Global Nonlinear Stability of the Minkowski Space (PMS-41)},
  ISBN = {9781400863174},
  url = {http://dx.doi.org/10.1515/9781400863174},
  DOI = {10.1515/9781400863174},
  publisher = {Princeton University Press},
  author = {Christodoulou,  Demetrios and Klainerman,  Sergiu},
  year = {1994},
  month = dec 
}

@article{Choquet_Bruhat_1969, title={Global aspects of the Cauchy problem in general relativity}, volume={14}, ISSN={1432-0916}, url={http://dx.doi.org/10.1007/BF01645389}, DOI={10.1007/bf01645389}, number={4}, journal={Communications in Mathematical Physics}, publisher={Springer Science and Business Media LLC}, author={Choquet-Bruhat, Yvonne and Geroch, Robert}, year={1969}, month=dec, pages={329–335} }

@article{ChoquetBruhatLocalEinstein,
  title={Th{\'e}or{\`e}me d'existence pour certains syst{\`e}mes d'{\'e}quations aux d{\'e}riv{\'e}es partielles non lin{\'e}aires},
  author={Choquet-Bruhat, Yvonne},
  journal={Acta mathematica},
  volume={88},
  number={1},
  pages={141--225},
  year={1952},
  publisher={Springer}
}

@misc{cicortas2024nonlinearscatteringtheoryasymptotically,
      title={Nonlinear Scattering Theory for Asymptotically de Sitter Vacuum Solutions in All Even Spatial Dimensions}, 
      author={Serban Cicortas},
      year={2024},
      eprint={2410.01558},
      archivePrefix={arXiv},
      primaryClass={gr-qc},
      howpublished={arXiv:2410.01558},
      url={https://arxiv.org/abs/2410.01558}, 
}

@article{vasy2007-to,
  author  = {Andr{\'a}s Vasy},
  title   = {The wave equation on asymptotically de Sitter-like spaces},
  journal = {Advances in Mathematics},
  volume  = {223},
  number  = {1},
  pages   = {49--97},
  year    = {2010},
  doi     = {10.1016/j.aim.2009.07.005}
}

@INCOLLECTION{obstruction,
  title     = "The ambient obstruction tensor and Q-curvature",
  booktitle = "{IRMA} Lectures in Mathematics and Theoretical Physics",
  author    = "Graham, C Robin and Hirachi, Kengo",
  publisher = "EMS Press",
  pages     = "59--71",
  month     =  may,
  year      =  2005,
  address   = "Zuerich, Switzerland",
  language  = "en"
}

@article{Chrusciel2004-xh,
  author  = {Piotr T. Chru{\'s}ciel and Erwann Delay and John M. Lee and Dale N. Skinner},
  title   = {Boundary regularity of conformally compact Einstein metrics},
  journal = {Journal of Differential Geometry},
  volume  = {69},
  number  = {1},
  pages   = {111--136},
  year    = {2005},
  doi     = {10.4310/jdg/1121540341}
}
